\documentclass[11pt,reqno,upgreek]{amsart}
\usepackage{amsmath,amsthm,amssymb,bm,stmaryrd,mathscinet,amsxtra}
\usepackage{mathrsfs}
\usepackage[hypertex]{hyperref}
\usepackage[all]{xy}
\usepackage{tikz}
\usetikzlibrary{cd}
\usepackage{comment}
\UseRawInputEncoding

\makeatletter
\renewcommand{\theequation}{%
\thesubsection.\arabic{equation}}
\@addtoreset{equation}{subsection}
\makeatother

\newtheorem{Thm*}{Theorem}[section]
\newtheorem{Thm}{Theorem}[subsection]
\newtheorem{Lem}[Thm]{Lemma}
\newtheorem{Cor}[Thm]{Corollary}
\newtheorem{Prop}[Thm]{Proposition}
\newtheorem{Prop*}[Thm*]{Proposition}
\newtheorem{Rem}[Thm]{Remark}
\newtheorem{Def*}[Thm*]{Definition}
\newtheorem{Def}[Thm]{Definition}

\newcommand{\C}{\mathbb{C}}           
\newcommand{\Z}{\mathbb{Z}}
\newcommand{\Q}{\mathbb{Q}}

\newcommand{\Hom}{\mathrm{Hom} \,}
\newcommand{\End}{\mathrm{End}}
\newcommand{\Ker}{\text{Ker} \,}

\newcommand{\id}{\mathrm{id}}

\newcommand{\cl}{\mathrm{cl}}

\newcommand{\fg}{{\mathfrak g}}

\newcommand{\fm}{{\mathfrak m}}

\newcommand{\ga}{\alpha}
\newcommand{\gb}{\beta}

\newcommand{\gl}{\lambda}
\newcommand{\gL}{\Lambda}
\newcommand{\gd}{\delta}
\newcommand{\gD}{\Delta}

\renewcommand{\ggg}{\gamma}
\newcommand{\gs}{\sigma}
\newcommand{\gS}{\Sigma}

\newcommand{\gee}{\varepsilon}

\newcommand{\ol}{\overline}

\newcommand{\wti}{\widetilde}

\newcommand{\bk}{\mathbf{k}}

\newcommand{\modfg}[1]{{#1}\mathrm{\mathchar`-mod}_{\mathrm{fg}}}
\newcommand{\modfd}[1]{{#1}\mathrm{\mathchar`-mod}_{\mathrm{fd}}}

\renewcommand{\Hom}[1]{\mathrm{Hom}_{#1}}

\newcommand{\opp}[1]{{#1}^{\mathrm{opp}}}

\renewcommand{\ss}{\mathsf{s}}
\newcommand{\sg}{\mathsf{g}}

\newcommand{\bP}{\mathbb{P}}
\newcommand{\bS}{\mathbb{S}}
\newcommand{\bV}{\mathbb{V}}
\newcommand{\bE}{\mathbb{E}}

\newcommand{\bfm}{\mathbf{m}}
\newcommand{\bfn}{\mathbf{n}}

\newcommand{\bQ}{\mathbf{Q}}
\newcommand{\bA}{\mathbf{A}}

\newcommand{\cC}{\mathcal{C}}
\newcommand{\cF}{\mathcal{F}}
\newcommand{\cP}{\mathcal{P}}
\newcommand{\cQ}{\mathcal{Q}}
\newcommand{\cA}{\mathcal{A}}
\newcommand{\cI}{\mathcal{I}}
\newcommand{\cJ}{\mathcal{J}}

\newcommand{\aff}{\mathrm{aff}}
\newcommand{\wh}{\widehat}

\newcommand{\bmi}{\bm{\imath}}
\newcommand{\bmj}{\bm{\jmath}}

\title[Equivalence between quiver Hecke algebra modules and HL categories]
{Equivalence between module categories over quiver Hecke algebras and Hernandez--Leclerc's categories in general types}

\author[K.~Naoi]{Katsuyuki Naoi}
\address[K.~Naoi]{%
Institute of Engineering \\
Tokyo University of Agriculture and Technology\\
2-24-16 Naka-cho, Koganei-shi, Tokyo 184-8588, JAPAN}
\email{naoik@cc.tuat.ac.jp}

\keywords{generalized quantum affine Schur--Weyl duality, quantum affine algebra, quiver Hecke algebra, affine highest weight category}
\subjclass[2010]{17B37, 81R50}

\begin{document}

\begin{abstract}
 We prove in full generality that the generalized quantum affine Schur--Weyl duality functor, introduced by Kang--Kashiwara--Kim, gives an equivalence
 between the category of finite-dimensional modules over a quiver Hecke algebra and a certain full subcategory of finite-dimensional modules over a quantum affine algebra
 which is a generalization of the Hernandez--Leclerc's category $\cC_Q$.
 This was previously proved in untwisted $ADE$ types by Fujita using the geometry of quiver varieties, which is not applicable in general.
 Our proof is purely algebraic, and so can be extended uniformly to general types.
\end{abstract}

\maketitle

\section{Introduction}

\textit{Quiver Hecke algebras} (also called Khovanov--Lauda--Rouquier algebras) are a family of $\Z$-graded algebras introduced by Khovanov--Lauda \cite{khovanov2009diagrammatic}
and Rouquier \cite{rouquier20082}.
The main motivation of the study of these algebras comes from the fact that they categorify the half of a quantum group.
More precisely, for a Kac-Moody Lie algebra $\sg$, the Grothendieck group of the direct sum $\bigoplus_\gb R^{\sg}(\gb)\mathrm{\mathchar`-gmod}_{\mathrm{fd}}$ of the categories of finite-dimensional 
graded modules over the quiver Hecke algebras $R^\sg(\gb)$ associated with $\sg$
is isomorphic to the integral form $U^-_{\Z[q^{\pm 1}]}(\sg)^\vee$ of the dual of the negative half of the quantized enveloping algebra $U_q(\sg)$.
Here the $\Z[q^{\pm 1}]$-algebra structure is given by the convolution product and the grading shift.
Moreover for symmetric $\sg$, it is proved in \cite{varagnolo2011canonical} and \cite{rouquier2012quiver} that 
this isomorphism sends the classes of self-dual simple modules to the upper global basis.
In particular, when $\sg$ is a simple Lie algebra of type $ADE$, the above results imply that, if we forget the grading, $\bigoplus_\gb \modfd{R^\sg(\gb)}^0$ gives a categorification of
the coordinate ring $\C[N]$ of the unipotent group associated with $\sg$, since $U^-_{\Z[q^{\pm 1}]}(\sg)^\vee$ specializes to $\C[N]$ at $q=1$.
Here we denote by $\modfd{R^\sg(\gb)}^0$ the category of finite-dimensional $R^\sg(\gb)$-modules obtained from graded ones by forgetting the gradings.

In \cite{hernandez2015quantum}, Hernandez--Leclerc introduced another categorification of $\C[N]$ by using finite-dimensional representations over a quantum affine algebra $U_q'(\wh{\sg})$,
where $\wh{\sg}$ is the untwisted affine Lie algebra associated with $\sg$ of type $ADE$.
Denote by $\cC_{\wh{\sg}}$ the category of finite-dimensional integrable $U_q'(\wh{\sg})$-modules.
For each Dynkin quiver $Q$ whose underlying graph is the Dynkin diagram of $\sg$, 
they defined a full subcategory $\cC_Q$ of $\cC_{\wh{\sg}}$ which contains some fundamental modules determined from given data, and is stable under taking tensor products.
Then they proved that the complexified Grothendieck ring $\C \otimes_{\Z} K(\cC_Q)$ is isomorphic to 
$\C[N]$,
and that this isomorphism sends the classes of simple modules to the upper global basis of $\C[N]$.

Therefore there are two categorifications of $\C[N]$ and its upper global basis, via quiver Hecke algebras and quantum affine algebras.
Now it would be natural to ask (see \cite[Subsection 1.6]{hernandez2015quantum})
whether there is some functorial relationship between them or not.
This question has been solved completely by Kang--Kashiwara--Kim \cite{kang2018symmetric,kang2015symmetric}, and Fujita \cite{fujita2017affine}.
In \cite{kang2018symmetric} they developed a machinery of the \textit{generalized quantum affine Schur--Weyl duality functor} for a quantum affine algebra $U_q'(\fg)$ of general type.
They proved that, if a suitable family $J$ of simple modules in the category $\cC=\cC_{\fg}$ of finite-dimensional integrable $U_q'(\fg)$-modules is given, then we can construct a
$(U_q'(\fg),R^J(\gb))$-bimodule $\wh{W}^{\otimes \gb}$ for each $\gb$, and hence a functor 
\[ \cF_\gb=\wh{W}^{\otimes \gb}\otimes_{R^J(\gb)} -\colon \modfd{R^J(\gb)}^0 \to \cC
\]
is obtained.
Here the type of the corresponding quiver Hecke algebras $R^J(\gb)$ is determined from the given datum $J$, and  a priori irrelevant to that of $\fg$.
Moreover, they proved that the direct sum $\cF=\bigoplus_\gb \cF_\gb\colon \bigoplus_\gb \modfd{R^J(\gb)}^0\to \cC$ 
is a monoidal functor in general, and exact if $R^J(\gb)$ are of type $ADE$.
Then in \cite{kang2015symmetric}, they applied this construction for $\fg=\hat{\sg}$ of type $ADE$ 
to obtain a functor whose target is the Hernandez--Leclerc's category $\cC_Q$.
They proved that $R^J(\gb)$ in this case are of type $\sg$,
and the obtained functor gives a one-to-one correspondence between simple modules and an isomorphism of their Grothendieck rings\footnote{In type $E$ this was proved
assuming some simpleness of normalized $R$-matrices, which has been proved later in \cite{fujita2020geometric,oh2019categorical}.}.
This gives a functorial interpretation of the fact that two algebras give the same categorifications. 
In the following we write $\mathcal{R}_\gb^J = \modfd{R^J(\gb)}^0$.

After that, Fujita \cite{fujita2017affine,fujita2020geometric} proved a stronger relation between $\cC_Q$ and $\bigoplus_\gb \mathcal{R}^J_\gb$,
namely, the categorical equivalence.
Let us recall the proof of \cite{fujita2017affine} briefly, which strongly motivated the present work.
A main idea is to consider some ``completions" $\wh{\cC}_{Q,\gb}$ and $\wh{\mathcal{R}}^J_\gb$ of blocks of these categories, each of which has the structure of an 
\textit{affine highest weight category}
introduced by Kleshchev \cite{kleshchev2015affine}.
Here the completion $\wh{\mathcal{R}}_\gb^J$ is naturally defined as the category of finitely generated modules over a completion 
$\wh{R}^J(\gb)$ of $R^J(\gb)$.
On the other hand,  there is no obvious candidate for the completion $\wh{\cC}_{Q,\gb}$.
In \cite{fujita2017affine}, $\wh{\cC}_{Q,\gb}$ is defined as the category of finitely generated modules over a completion of the convolution algebra of the equivariant $K$-groups of
the quiver varieties associated with $\sg$, whose object becomes a $U_q'(\wh{\sg})$-module by \cite{nakajima2001quiver}.
Then by applying a sufficient condition for a functor between two affine highest weight categories to be an equivalence
(see Theorem \ref{Thm:Fujita_theorem} of the present paper), he proved that $\wh{\cC}_{Q,\gb}$ and $\wh{\mathcal{R}}^J_\gb$ are equivalent.
In addition, he also proved that the full subcategory of $\wh{\cC}_{Q,\gb}$ (resp.~$\wh{\mathcal{R}}^J_\gb$) 
consisting of finite-dimensional modules coincides with $\cC_{Q,\gb}$ (resp.~$\mathcal{R}^J_\gb$),
which completed the proof.

The aforementioned result of \cite{kang2015symmetric} has been extended to all the other quantum affine algebras:
types $A_{n}^{(2)}$ and $D_{n}^{(2)}$ in \cite{kang2016symmetric}, types $B_n^{(1)}$ and $C_n^{(1)}$ in \cite{kashiwara2019categorical}, 
and the remaining exceptional types in \cite{oh2019categorical}.
More precisely, in these papers a family of full subcategories of $\cC$ is defined for each type of a quantum affine algebra as a generalization of the Hernandez--Leclerc's categories $\cC_Q$, 
and it was shown in each case that the corresponding $R^J(\gb)$ are of type $ADE$, and 
the generalized quantum affine Schur--Weyl duality functor gives a bijection of simple modules and an isomorphism of their Grothendieck rings.
Recently, Fujita--Oh \cite{fujita2020q} developed a construction of these subcategories in a unified way for general types using the notion of a \textit{$Q$-datum} (which will be
recalled in Subsection \ref{Subsection:Q-data} of the present paper). 
We denote by $\cC_{\bQ}$ the full subcategory of $\cC$ corresponding to a $Q$-datum $\bQ$.

The purpose of this paper is to show that the result of Fujita mentioned above is also extended to all the other quantum affine algebras, namely, 
that the generalized quantum affine Schur--Weyl duality functor gives an equivalence between $\cC_{\bQ}$ and $\bigoplus_\gb \mathcal{R}_\gb^J$ in general types (Theorem \ref{Thm}).
This gives a proof to \cite[Conjecture 5.7]{kang2016symmetric} and \cite[Conjecture 6.11]{kashiwara2019categorical}.
Our proof is purely algebraic, and results of \cite{fujita2017affine} are not used.
In some cases, it occurs for different quantum affine algebras that the corresponding quiver Hecke algebras coincide.
In such cases our main theorem gives interesting equivalences of categories of finite-dimensional modules over \textit{different} types of quantum affine algebras (see Corollary \ref{Cor}).

Let us explain our strategy of the proof.
Similarly as in the untwisted $ADE$ types in \cite{fujita2017affine}, 
we consider some completions $\wh{\cC}_{\bQ,\gb}$ and $\wh{\mathcal{R}}^J_\gb$ of blocks.
Here $\wh{\mathcal{R}}^J_\gb$ is the same with the one used previously.
However, we cannot apply the geometry of quiver varieties in general types, and hence we adopt a completely different category for $\wh{\cC}_{\bQ.\gb}$,
the category of finitely generated modules over a ``Schur-like" algebra.
Precisely to say, we define an algebra $\bE^{\gb}$ as the $\wh{R}^J(\gb)^{\mathrm{opp}}$-linear endomorphism 
\[ \bE^{\gb} = \mathrm{End}_{\wh{R}^J(\gb)^{\mathrm{opp}}}(\wh{W}^{\otimes \gb})
\]
of the bimodule $\wh{W}^{\otimes \gb}$ appearing in the construction of the functor, and study the category $\modfg{\bE^{\gb}}$ of finitely generated $\bE^{\gb}$-modules, which
 is our $\wh{\cC}_{\bQ,\gb}$.
Note that there is an algebra homomorphism $\Phi_\gb\colon U_q'(\fg) \to \bE^{\gb}$ by the construction, and hence an $\bE^{\gb}$-module is automatically a $U_q'(\fg)$-module. 
In this paper, we show that $\modfg{\bE^{\gb}}$ is an affine highest weight category (Proposition \ref{Prop:Main_Prop} (ii)), 
and using this the equivalence of $\modfg{\bE^{\gb}}$ and $\wh{\mathcal{R}}^J_\gb$ is proved.
To complete the proof, we also have to show that the category of finite-dimensional  $\bE^{\gb}$-modules coincides with $\cC_{\bQ,\gb}$ (Proposition \ref{Prop:Main_Prop} (i)).
A main idea for the proof of this assertion is to use \textit{affine cellular} structures (introduced by Koenig--Xi \cite{koenig2012affine}) of the algebras.
By Kleshchev \cite{kleshchev2015affine}, $\wh{R}^J(\gb)$ has a chain of ideals $\wh{R}^J(\gb) =R_0  \supseteq \cdots \supseteq R_p=\{0\}$
whose subquotients are described in terms of \textit{standard modules} appearing in the axioms of affine highest weight categories.
By setting 
\[ \wh{W}_k=\wh{W}^{\otimes \gb} \otimes_{\wh{R}^J(\gb)} R_k\subseteq \wh{W}^{\otimes \gb} \ \ \text{ and } \ \ \bE_k = \{f \in \bE^{\gb} \mid f(\wh{W}^{\otimes \gb}) \subseteq \wh{W}_k\}
   \subseteq \bE^{\gb}
\]
for each $k$, we also obtain a chain of ideals $\bE^{\gb} = \bE_0 \supseteq \cdots \supseteq \bE_p = \{0\}$.
On the other hand, (a quotient of) $U_q'(\fg)$ also has a similar chain of ideals by \cite{cui2015affine,nakajima2015affine},
and these are compatible with $\Phi_\gb$.
By comparing the subquotients of these two chains we show a kind of the \textit{centralizing property}, namely, that the image of $\Phi_\gb$ is dense (in a suitable sense) in $\bE^{\gb}$.
Then the coincidence of the category of finite-dimensional $\bE^\gb$-modules and $\cC_{\bQ,\gb}$ follows from this.

The paper is organized as follows.
In Section \ref{Section:affine_quasi_hereditary}, we recall the definitions of affine quasihereditary algebras and affine highest weight categories,
and their properties.
In Section \ref{section:Quantum affine algebras}, we prepare several definitions and facts concerning with quantum affine algebras and their representations,
and in Section \ref{section:QHA} we prepare these concerning with quiver Hecke algebras.
In Subsection \ref{Subsection:Q-data} we recall the notion of $Q$-data, and after some preparation mainly for combinatorial objects in Subsection \ref{subsection:Full_subcategories},
we recall the definition of generalized quantum affine Schur--Weyl duality functors in Subsection \ref{subsection:GQASWD}.
In Subsection \ref{Subsection:Image}, we prove that the functor sends standard modules to standard modules.
In Subsection \ref{subsection:Main theorems and corollaries} we state key propositions, and assuming this we give a proof of the main theorem and a corollary.
The proof of the key propositions is given in Subsection \ref{subsection:proof_of_main}.\\

\noindent\textbf{Acknowledgments.}\ \ \
The author is very grateful to Ryo Fujita for a lot of valuable comments.
In particular, he told the author the proof of Lemma \ref{Lem:Lem_for_comparing_orders} using the inverse of a quantum Cartan matrix,
which is more conceptual than that in the previous version.
The author was supported by JSPS Grant-in-Aid for Scientific Research (C) No.~20K03554.\\

\noindent\textbf{Convention}\ \ \ 
For a ring $\cA$, $\opp{\cA}$ denotes the opposite ring of $\cA$.
On an $\cA$-module $M$, the endomorphism ring $\mathrm{End}_{\cA}(M)$  acts from the left.
For a two-sided ideal $\cI \subseteq \cA$, we denote the quotient $M/\cI M$ by $M/\cI$ to simplify the notation.
Let $\bk$ be a field.
When the base field is $\bk$, we write $\otimes$ for $\otimes_{\bk}$ if no confusion will occur.
For $i=1,2$, let $R_i$ be a complete local commutative $\bk$-algebra with maximal ideal $\mathfrak{m}_i$ with $R_i/\mathfrak{m}_i \cong \bk$.
For an $R_i$-module $M_i$ ($i =1,2$), we denote by $M_1 \hat{\otimes} M_2$ the completion of the $(R_1\otimes R_2)$-module $M_1 \otimes M_2$ with respect to the maximal
ideal $\mathfrak{m}_1 \otimes R_2 +R_1 \otimes \mathfrak{m}_2$. 
Note that $M_1 \hat{\otimes} M_2$ is a module over the complete local ring $R_1 \hat{\otimes} R_2$.

\section{affine quasihereditary algebras and affine highest weight categories}\label{Section:affine_quasi_hereditary}

\makeatletter
\renewcommand{\theequation}{%
\thesection.\arabic{equation}}
\@addtoreset{equation}{section}
\makeatother

Let $\mathcal{A}$ be a left Noetherian algebra over an algebraically closed field $\mathbf{k}$, and denote by $\mathcal{I}$ the Jacobson radical of $\cA$.
Throughout this section, we assume that $\cA/\cI$ is finite-dimensional and $\cA$ is complete with respect to the $ \cI $-adic topology (i.e., $\underset{n}{\varprojlim}\, \cA/\cI^n \cong \cA$).
Let $\modfg{\cA}$ denote the category of finitely generated left $\cA$-modules, and let $\{L(\pi) \mid \pi \in \Pi\}$ be
the set of isomorphism classes of simple modules in $\modfg{\cA}$, where $\Pi$ is a parametrizing set.
By our assumptions, $\Pi$ is a finite set and all $L(\pi)$ are finite-dimensional.
For each $\pi \in \Pi$, fix a projective cover $P(\pi)$ of $L(\pi)$.

Here we recall the definitions and some properties of affine quasihereditary algebras and affine highest weight categories in a topologically complete
setting.
The main reference is \cite{kleshchev2015affine}. 
Although the paper treats a graded setting,
the proofs (at least for the results appearing below) also work in our complete setting with obvious modification\footnote{Note that, instead of a graded structure,
each $M \in \modfg{\cA}$ has a filtration
$M\supseteq  \cI M\supseteq \cI^2 M\supseteq \cdots$, which satisfies $\dim M/\cI^nM <\infty$ and 
$\underset{n}{\varprojlim}\, M/\cI^nM \cong M$ (see e.g.~\cite{zbMATH03777679}).}
(see also \cite[Section 4]{fujita2018tilting}).
\begin{Def*}\label{Def:affine_heredity_ideal}\normalfont
 A two-sided ideal $\cJ \subseteq \cA$ is called \textit{affine heredity} if it satisfies the following conditions:
 \begin{itemize}
  \setlength{\leftskip}{-0.5cm}
  \item[(i)] We have $\Hom{\modfg{\cA}} (\cJ,\cA/\cJ) =0$;
  \item[(ii)] As a left $\cA$-module, we have $\cJ \cong P(\pi)^{\oplus m}$ for some $\pi \in \Pi$ and $m \in \Z_{>0}$;
  \item[(iii)] The endomorphism $\bk$-algebra $\End_\cA\big(P(\pi)\big)$ is isomorphic to a ring of formal power series $\bk\llbracket z_1,\ldots,z_n\rrbracket$ with some 
   $n \in \Z_{\geq 0}$,
   and $P(\pi)$ is free of finite rank over $\End_\cA\big(P(\pi)\big)$.
  \end{itemize} 
\end{Def*}
\begin{Def*}\normalfont
 The algebra $\cA$ is called \textit{affine quasihereditary} if there exists a finite chain of ideals
  \[ \cA=\cJ_0 \supsetneq \cJ_1 \supsetneq \cdots \supsetneq \cJ_l = \{0\}
  \]
  with $\cJ_{i-1}/\cJ_{i}$ an affine heredity ideal of $\cA/\cJ_{i}$ for all $1\leq i \leq l$.
  Such a chain of ideals is called an \textit{affine heredity chain}.
\end{Def*}
For a subset $\gS \subseteq \Pi$ and $M \in \modfg{\cA}$, define a submodule $\mathcal{O}^{\gS}(M) \subseteq M$ by
\[ \mathcal{O}^{\gS}(M)=\sum_{\pi \notin \gS, f \in \Hom{\cA}(P(\pi),M)} \mathrm{Im} f,
\]
and set $\mathcal{Q}^{\gS} (M) = M/ \mathcal{O}^{\gS}(M)$.
We have $\mathcal{O}^{\gS} (M) = \mathcal{O}^{\gS}(\cA) M$ (\cite[Lemma 3.12]{kleshchev2015affine}), and $\mathcal{Q}^{\gS}$ can be regarded as a right exact functor 
from $\modfg{\cA}$ to 
$\modfg{(\cA/\mathcal{O}^{\gS}(\cA))}$.
When the set $\Pi$ is endowed with a partial order $\leq$,
we set 
\[ \Pi_{\leq \pi}= \{ \gs \in \Pi \mid \gs \leq \pi\} \ \ \text{ for $\pi \in \Pi$},
\]
and we write $\mathcal{O}^{\leq \pi} = \mathcal{O}^{\Pi_{\leq \pi}}$ and $\mathcal{Q}^{\leq \pi} = \mathcal{Q}^{\Pi_{\leq \pi}}$.
For $\pi \in \Pi$, 
\begin{equation} \label{eq:definition_of_standard}
 \gD(\pi) = \mathcal{Q}^{\leq \pi} (P(\pi))
\end{equation}
is called the \textit{standard module} associated with $\pi$. 

\begin{Def*}\normalfont
 The category $\modfg{\cA}$ is called an \textit{affine highest weight category} with respect to the poset $(\Pi, \leq)$ if it satisfies the following conditions:
 \begin{itemize}
    \setlength{\leftskip}{-0.5cm}
  \item[(i)] For every $\pi \in \Pi$, there exists a filtration
   \[ \mathcal{O}^{\leq \pi} (P(\pi)) =V_0 \supseteq V_1 \supseteq V_2\supseteq \cdots \supseteq V_{p_\pi} = \{0\}
   \]
   such that each $V_{i-1}/V_i$ is isomorphic to $\gD(\gs)$ for some $\gs > \pi$;
 \item[(ii)] For every $\pi \in \Pi$, the endomorphism $\bk$-algebra $\End_\cA\big(\gD(\pi)\big)$ is isomorphic to a ring of formal power series 
  $\bk\llbracket z_1,\ldots,z_{n_\pi}\rrbracket$ with some $n_{\pi} \in \Z_{\geq 0}$, and $\gD(\pi)$ is free of finite rank over $\End_{\cA}\big(\gD(\pi)\big)$.
 \end{itemize}
 For $\pi \in \Pi$, let $N_\pi$ be the unique maximal ideal of $\mathrm{End}_\cA\big(\gD(\pi)\big)$.
 The module $\bar{\gD}(\pi) = \gD(\pi)/N_\pi$ is called the \textit{proper standard module} associated with $\pi$.
\end{Def*}

\begin{Thm*}[{\cite[Theorem 6.7]{kleshchev2015affine}}]\label{Thm:hereditary_corresp}\ \\
 {\normalfont(i)} Assume that $\modfg{\cA}$ is an affine highest weight category with respect to a poset $(\Pi, \leq)$.
  Then for any total order $\{\pi_1,\ldots,\pi_l\}$ on $\Pi$ such that $i < j$ holds whenever $\pi_i < \pi_j$, the chain of ideals
  \[ \cA \supsetneq \mathcal{O}^{\Pi(1)}(\cA) \supsetneq \mathcal{O}^{\Pi(2)}(\cA) \supsetneq \cdots \supsetneq \mathcal{O}^{\Pi(l)}(\cA)=\{0\}
  \]
  is affine heredity, where we set $\Pi(i)= \{\pi_1,\ldots,\pi_i\} \subseteq \Pi$. In particular, $\cA$ is an affine quasihereditary algebra.\\
 {\normalfont(ii)} Conversely, if $\cA$ is an affine quasihereditary algebra with an affine heredity chain $\{\cJ_i\}_{i=0}^l$, 
  then $\modfg{\cA}$ is an affine highest weight category with respect to the totally ordered set $\Pi = \{ \pi_1,\ldots,\pi_l\}$, where the order is given by setting $\pi_i = \pi$ if
  \[ i = \min\{j \mid \Hom{\modfg{\cA}}(P(\pi),\cA/\cJ_j)\neq 0\}.
  \]
  Moreover for $1\leq i \leq l$, we have $\cJ_{i-1}/\cJ_i \cong \gD(\pi_i)^{\oplus m_i}$ as left $\cA$-modules with some $m_i \in \Z_{>0}$.
\end{Thm*}

\begin{Prop*}
 Assume that $\modfg{\cA}$ is an affine highest weight category with respect to a poset $(\Pi, \leq)$.
 Then for each $\pi \in \Pi$, there is an indecomposable module $\bar{\nabla}(\pi)$, which is unique up to isomorphism, satisfying
 \[ \mathrm{Ext}_\cA^i(\gD(\gs),\bar{\nabla}(\pi)) = \begin{cases} \bk & \text{\normalfont{if} } i=0 \text{ \normalfont{and} } \gs = \pi,\\ 0 & \text{\normalfont{otherwise}}.\end{cases}
 \]
 The module $\bar{\nabla}(\pi)$ is called the \textit{proper costandard module} associated with $\pi$.
\end{Prop*}
\begin{proof}
 See \cite[Lemmas 7.2, 7.4, and 7.9]{kleshchev2015affine}.
\end{proof}

The following theorem by Fujita is quite important in this paper.

\begin{Thm*}[{\cite[Theorem 3.9 and Subsection 4.2]{fujita2018tilting}}]\label{Thm:Fujita_theorem}
 For $i=1,2$, let $\modfg{\cA_i}$ be an affine highest weight category with respect to a poset $(\Pi_i, \leq_i)$.
 Assume that there is an exact functor $F\colon \modfg{\cA_1} \to \modfg{\cA_2}$ and the following conditions are satisfied:
 \begin{itemize}
  \setlength{\leftskip}{-0.5cm}
  \item[(i)] For $i =1,2$, $\cA_i$ is a finitely generated module over its center;
  \item[(ii)] There exists a bijection $f\colon \Pi_1 \to \Pi_2$ preserving partial orders, and we have
   \[ F(\gD(\pi)) \cong \gD(f(\pi)) \ \text{ and } \ F(\bar{\nabla}(\pi)) \cong \bar{\nabla}(f(\pi)) \ \ \text{ for all $\pi \in \Pi_1$}.
   \]
 \end{itemize}
 Then the functor $F$ gives an equivalence of categories $F\colon \modfg{\cA_1} \stackrel{\sim}{\to} \modfg{\cA_2}$.
\end{Thm*}

Finally let us recall the connection with affine cellular algebras introduced by Koenig and Xi \cite{koenig2012affine}.
Assume that $\cA$ has a $\bk$-algebra anti-involution $\tau$.
For $M \in \modfg{\cA}$, let $M^{\tau}$ denote the right $\cA$-module obtained from $M$ by twisting the action of $\cA$ via $\tau$.
When $M$ is finite-dimensional, a left $\cA$-module structure on $M^* = \Hom{\,\bk}(M,\bk)$ is given by
\begin{equation}\label{eq:module_on_dual}
 (af)(v) = f(\tau(a)v) \ \ \text{ for } f \in M^*, \ a \in \cA,\ v \in M.
\end{equation}

\begin{Thm*}\label{Thm:affine_cellular}
 Let $\pi \in \Pi$, and assume that $\cA$ has an affine heredity ideal $\cJ$ with $\cJ\cong P(\pi)^{\oplus m}$ for some $\pi \in \Pi$,
 and an anti-involution $\tau$ satisfying $L(\pi)^*\cong L(\pi)$ as $\cA$-modules.
 Then $\cJ$ satisfies the following conditions:
 \begin{itemize}
  \setlength{\leftskip}{-0.5cm}
  \item[(i)] The ideal $\cJ$ is $\tau$-invariant;
  \item[(ii)] Set $\mathcal{B} = \End_\cA P(\pi) (\cong \bk\llbracket z_1,\ldots,z_n\rrbracket)$.
   Then there exists an idempotent $e$ in $\cA$ such that $\cJ =  \cA e \cA$, $e\cA e \cong \mathcal{B}$ as $\bk$-algebras, 
   $\cA e \cong P(\pi)$ as $(\cA, \mathcal{B})$-bimodules,
   $e\cA \cong P(\pi)^{\tau}$ as $(\mathcal{B},\cA)$-bimodules, and
   \[ \cA e\cA \cong P(\pi) \otimes_{\mathcal{B}} P(\pi)^{\tau} \ \text{ as $(\cA,\cA)$-bimodules}.
   \]
\end{itemize}
\end{Thm*}
\begin{proof}
 See \cite[Proposition 9.8]{kleshchev2015affine} and its proof.
\end{proof}
This theorem implies that an affine quasihereditary algebra $\cA$ with an anti-involution $\tau$ satisfying $L(\pi)^* \cong L(\pi)$ for all $\pi \in \Pi$ is an affine cellular algebra.


\makeatletter
\renewcommand{\theequation}{%
\thesubsection.\arabic{equation}}
\@addtoreset{equation}{subsection}
\makeatother

\section{Quantum affine algebras}
\label{section:Quantum affine algebras}

\subsection{Cartan datum}

A \textit{Cartan datum} is a quintuple $(A,P,\Pi,P^{\vee},\Pi^\vee)$ with index set $I$ consisting of
\begin{itemize}
 \setlength{\leftskip}{-0.5cm}
 \item[(a)] a symmetrizable generalized Cartan matrix $A = (a_{ij})_{i,j \in I}$,
 \item[(b)] a free abelian group $P$, called the \textit{weight lattice},
 \item[(c)] $\Pi=\{\ga_i \in P \mid i \in I\}$, called the set of \textit{simple roots},
 \item[(d)] $P^\vee = d^{-1}\Hom{\Z}(P,\Z)$ with some positive integer $d$, called the \textit{coweight lattice},
 \item[(e)] $\Pi^\vee=\{h_i \in P^\vee\mid i\in I\}$, called the set of \textit{simple coroots},
\end{itemize}
satisfying 
\begin{itemize}
 \setlength{\leftskip}{-0.5cm}
 \item[(i)] $\langle h_i,\ga_j\rangle = a_{ij}$ for $i,j\in I$, \ \ \ \ \ \ \ (ii) $\Pi$ is linearly independent,
 \item[(iii)] for each $i \in I$, there exists $\gL_i \in P$ such that $\langle h_j,\gL_i\rangle =\gd_{ij}$ for all $j\in I$. 
\end{itemize}
We call $\gL_i$ the \textit{fundamental weights}.
Let $D=\mathrm{diag}(\mathsf{s}_i\!\mid\! i\in I)$ be a diagonal matrix such that $DA$ is symmetric and $\mathsf{s}_i$ are positive rational numbers.
A free abelian group $Q=\bigoplus_{i \in I} \Z \ga_i$ is called the \textit{root lattice}, and set $Q^+ = \sum_{i \in I} \Z_{\geq 0} \ga_i \subseteq Q$.
Let $(\ ,\ )$ be a $\Q$-valued symmetric bilinear form on $P$ satisfying
\[ (\ga_i,\ga_j) =\ss_ia_{ij} \ \ \text{for $i,j \in I$} \ \ \ \text{ and } \ \ \ \langle h_i,\gl\rangle = \frac{2(\ga_i,\gl)}{(\ga_i,\ga_i)} \ \ 
   \text{for $i \in I$ and $\gl\in P$}.
\]

\subsection{Definition of quantum affine algebras}

Throughout the rest of this paper, we assume that $A=(a_{ij})_{i,j\in I}$ is a generalized Cartan matrix of affine type with index set $I=\{0,1,\ldots,n\}$.
We assume that indices are ordered as in \cite[Section 4.8]{MR1104219}, except $A_{2n}^{(2)}$-case in which we reverse the numbering (i.e., $i \leadsto n-i$).
Set $I_0 = I \setminus \{0\}$. 
Let $(A,P,\Pi,P^\vee,\Pi^\vee)$ be an affine Cartan datum, and let 
$\gd= \sum_{i \in I} c_i\ga_i$ be the imaginary root and $c=\sum_{i \in I} c_i^\vee h_i$ the canonical central element.
We assume that $P = \bigoplus_{i \in I} \Z \gL_i \oplus \Z \gd$.
We fix the symmetrizer $D=\mathrm{diag}(\ss_0,\ldots,\ss_n)$ so that
\[ \min\{\ss_i\mid i \in I\} =\begin{cases} 1/2 & \text{if $\fg$ is of type $A_{2n}^{(2)}$},\\
                                            1   & \text{otherwise},\end{cases}
\]
and also fix the integer $d$ defining $P^\vee$ by $d=2$ if $\fg$ is of type $A_{2n}^{(2)}$, and $d=1$ otherwise.
Let $\fg$ be the affine Kac-Moody Lie algebra associated with the datum $(A,P,\Pi,P^\vee,\Pi^\vee)$,
and $\fg_0$ the subalgebra corresponding to $I_0$.

From now on, we take the algebraic closure of $\C(q)$ in $\bigcup_{m>0} \C(\!(q^{1/m})\!)$ as a base field $\bk$.
Let $q_i = q^{\ss_i}$ for $i \in I$, $[k]_i = (q_i^{k}-q_i^{-k})/(q_i-q_i^{-1})$ for $k \in \Z$ and $[m]_i! = \prod_{s=1}^m [s]_i$ for $m \in \Z_{\geq 0}$.
Let $U_q=U_q(\fg)$ be the quantum group associated with $(A,P,\Pi,P^\vee,\Pi^\vee)$, 
which is a $\bk$-algebra generated by $e_i,f_i$ ($i \in I$) and $q^h$ ($h \in P^\vee$) subject to the following relations:
\begin{gather*}
 q^0=1, \ \ q^hq^{h'} = q^{h+h'}, \ \ q^he_iq^{-h} = q^{\langle h,\ga_i\rangle} e_i, \ \ q^hf_iq^{-h}=q^{-\langle h,\ga_i\rangle}f_i \ \text{ for } h,h' \in P^\vee,\\
 e_if_j-f_je_i = \gd_{ij} \frac{K_i-K_i^{-1}}{q_i-q_i^{-1}}, \text{ where } K_i = q^{\ss_ih_i},\\
 \sum_{k=0}^{1-a_{ij}} (-1)^k e_i^{(1-a_{ij}-k)}e_je_i^{(k)} = \sum_{k=0}^{1-a_{ij}} (-1)^k f_i^{(1-a_{ij}-k)}f_jf_i^{(k)}=0 \ \text{ for } i\neq j.
\end{gather*}
Here $e_i^{(k)} = e_i^k/[k]_i!$ and $f_i^{(k)} = f_i^k/[k]_i!$ for $k \in \Z_{\geq 0}$.

Set $P_\cl=P/\Z\gd$, and let $\cl\colon P \to P_\cl$ be the canonical projection.
We have $P_\cl = \bigoplus_{i \in I} \Z \,\cl(\gL_i)$.
For simplicity of notation, we will write $\ga_i$ for $\cl(\ga_i)$ when there should be no confusion.
Let $U_q'=U_q'(\fg)$ be the subalgebra of $U_q$ generated by $e_i,f_i$ ($i \in I$) and $q^h$ ($h \in P_\cl^\vee = d^{-1}\Hom{\Z} (P_\cl,\Z)$),
and denote by $U_q^0=U_q^0(\fg)$ 
the quotient algebra $U_q'/\langle q^c-1\rangle$, where $\langle q^c-1\rangle$ is the two-sided ideal of $U_q'$ generated by
$q^c-1$.
By a slight abuse of notation, we will use same symbols $e_i,f_i,q^h$ for their images under the canonical projection $U_q' \to U_q^0$.
Let
$P_0 = \{\gl \in P_\cl\mid \langle c,\gl\rangle = 0\}$,
and define a partial order $\leq$ on $P_0$ by setting $\gl \leq \mu$ if and only if $\mu-\gl \in \sum_{i \in I_0}\Z_{\geq 0} \cl(\ga_i)$.
Set 
\[ \varpi_i = \mathrm{gcd}(c_0^\vee,c_i^\vee)^{-1}\big(c_0^\vee\cl(\gL_i)-c_i^\vee\cl(\gL_0)\big) \in P_0 \ \ \ \text{for } i \in I_0.
\]
Then we have $P_0 = \sum_{i \in I_0} \Z\varpi_i$.
Note that we have $\langle h_i, \varpi_j\rangle = \gd_{ij}$ for all $i,j\in I_0$ except for $i=j=n$ in type $A_{2n}^{(2)}$, in which we have $\langle h_n,\varpi_n\rangle =2$.
Set $P_0^+ = \sum_{i \in I_0} \Z_{\geq 0} \varpi_i$.
The bilinear form $(\ ,\ )$ on $P$ induces a bilinear form on $P_0$, which we also denote by $(\ ,\ )$.
Denote by $\gD$ the coproduct of $U_q$ defined by
\[ \gD(q^h) = q^h \otimes q^h, \ \ \ \gD(e_i)=e_i \otimes K_i^{-1} + 1\otimes e_i, \ \ \ \gD(f_i)=f_i\otimes 1 + K_i \otimes f_i,
\]
and by $S$ the antipode of $U_q$ defined by
\[ S(e_i)=-e_iK_i, \ \ \ S(f_i)=-K_i^{-1}f_i, \ \ \ S(q^h)=q^{-h}.
\]
These $\gD$ and $S$ also define coproducts and antipodes on $U_q'$ and $U_q^0$, respectively.

For $i \in I_0$, let $d_i \in \Z_{>0}$ be the smallest positive integer such that $\ga_i + d_i \gd$ is a root of $\fg$.
Let $x_{i,m}^{\pm} \in U_q^0$ (resp.~$h_{i,m} \in U_q^0$) for $i \in I_0$ and $m \in \Z$ (resp.~$m\in \Z\setminus\{0\}$) 
be the Drinfeld generator corresponding to the root $\pm \ga_i + md_i\gd$ (resp.~$md_i\gd$).
For the definitions, see \cite{MR1301623} in untwisted types, and \cite{MR2973398,MR3367090} in twisted types.
Let $U_q(0)$ (resp.\ $U_q(>)$, $U_q(<)$) be the $\bk$-subalgebra of $U_q^0$ generated by $h_{i,m}$'s and $q^h$'s 
(resp.~$x^+_{i,m}$'s, $x^-_{i,m}$'s).
We have 
\[ U^0_q=U_q(<)U_q(0)U_q(>),
\]
and the subalgebra $U_q(0)$ is commutative.

\subsection{Finite-dimensional representations over a quantum affine algebra}

For a $U_q$-module (resp.\ $U_q'$-module) $M$ and $\gl \in P$ (resp.\ $\gl \in P_\cl$), write
\[ M_\gl = \{ v \in M \mid q^h v = q^{\langle h, \gl\rangle}v \ \text{for $h \in P^\vee$ (resp.\ $h \in P_\cl^\vee$)}\}.
\]
We say a $U_q$-module (or $U_q'$-module) $M$ is \textit{integrable} if $M = \bigoplus_{\gl} M_\gl$
and the actions of $e_i$ and $f_i$ ($i\in I$) are locally nilpotent.
We say the \textit{level} of $M$ is $p \in \Z$ if $q^c$ acts as a multiplication by $q^p$ on $M$.
We denote by $\mathcal{C}$ the category of finite-dimensional integrable $U_q'$-modules, which has a monoidal structure via the coproduct $\gD$.
Any $M \in \cC$ is of level-zero, and hence a $U_q^0$-module.

Let $\mathcal{P}=\bigoplus_{(i,a)\in I_0\times \bk^\times} \Z \bm{\varpi}_{i,a}$, which is a free abelian group with a basis $\{\bm{\varpi}_{i,a}\mid i \in I_0,a \in \bk^\times\}$.
We call $\mathcal{P}$ the \textit{$\ell$-weight lattice}.
We define a $\Z$-linear map $\mathsf{cl}\colon \mathcal{P} \to P_0$ by $\mathsf{cl}(\bm{\varpi}_{i,a})=\varpi_i$.
An element in the submonoid $\mathcal{P}^+= \sum \Z_{\geq 0}\bm{\varpi}_{i,a}$ is said to be \textit{$\ell$-dominant}.

Let $M \in \mathcal{C}$.
Since $U_q(0)$ is commutative,
$M$ decomposes into a direct sum of generalized eigenspaces for $U_q(0)$ as $M = \bigoplus_{\Psi} M_{\Psi}$, where 
\[ M_{\Psi} = \{ v \in M \mid (u-\Psi(u))^Nv = 0 \ \text{ for } u \in U_q(0) \text{ and } N \gg 0\}
\]
for $\Psi \in \Hom{\bk\mathrm{\mathchar`-alg}}\left(U_q(0),\bk\right)$.
It is known  (see \cite{MR1745260} in untwisted types and \cite{MR2576287} in twisted types) that if $M_{\Psi}\neq 0$,
there exists a unique element $\bm{\pi}=\sum l_{i,a}\bm{\varpi}_{i,a} \in \mathcal{P}$ such that
\begin{equation}\label{eq:ell-weight_space}
 \Psi\big(\phi_i^{\pm}(z)\big)= q^{r_i\langle h_i,\mathsf{cl}(\bm{\pi})\rangle}\left(\prod_{a \in \bk^\times} \left(\frac{1-aq^{-r_i}z}{1-aq^{r_i}z}\right)^{l_{i,a}}\right)^{\pm}
 \ \ \ \text{for all } i\in I_0,
\end{equation}
where $r_i = \ss_i$ (resp.~$r_i=d_i$) if $\fg$ is untwisted (resp.~twisted), $\phi_i^{\pm}(z) \in U_q(0)\llbracket z^{\pm 1} \rrbracket$ are defined by
\[ \phi_i^{\pm}(z) = K_i^{\pm 1} \mathrm{exp}\Big(\pm (q_i-q_i^{-1})\sum_{m \geq 1} h_{i,\pm m}z^{\pm m}\Big),
\]
and $(-)^{\pm}$ denotes the formal expansion at $z=0$ and $z=\infty$ respectively.
In this case we write $M_{\bm{\pi}}= M_{\Psi}$, and call it the \textit{$\ell$-weight space} of $M$ with $\ell$-weight $\bm{\pi}$.
For $M_1,M_2 \in \cC$, it is known that
\begin{equation}\label{eq:tensor_lweight}
 \dim(M_1 \otimes M_2)_{\bm{\pi}} = \sum_{\bm{\pi}_1+\bm{\pi}_2=\bm{\pi}} \dim(M_1)_{\bm{\pi}_1} \cdot \dim (M_2)_{\bm{\pi}_2},
\end{equation}
see \cite[Lemma 3]{MR1745260} and \cite[Theorem 3.4]{MR2576287}.
We say $M \in \mathcal{C}$ is \textit{$\ell$-highest weight} with $\ell$-highest weight vector $v$ and $\ell$-highest weight $\bm{\pi} \in \cP^+$ if $M_{\bm{\pi}} = \bk v$,
$U_q'v = M$ and $x_{i,m}^{+}v=0$ for all $i \in I_0$ and $m \in \Z$.
For any $\bm{\pi} \in \mathcal{P}^+$, there is a simple $\ell$-highest weight module $L(\bm{\pi}) \in \mathcal{C}$ with $\ell$-highest weight $\bm{\pi}$ unique up to isomorphism,
and $\{ L(\bm{\pi})\mid \bm{\pi} \in \mathcal{P}^+\}$ is a complete set of isomorphism classes of simple modules in $\mathcal{C}$.
A module $L(\bm{\varpi}_{i,a})$ is called a \textit{fundamental module}.
For $k =1,2$, given an $\ell$-highest weight module $M_k$ with $\ell$-highest weight $\bm{\pi}_k \in \cP^+$,
we have 
\begin{equation}\label{eq:tensor_lweight2}
 (M_1 \otimes M_2)_{\bm{\pi}_1+\bm{\pi}_2} = (M_1)_{\bm{\pi}_1} \otimes (M_2)_{\bm{\pi}_2}.
\end{equation}

\begin{Rem}\normalfont
 In several literatures, for example \cite{MR1607008,MR1890649,kang2018symmetric}, a fundamental module is denoted by $V(\varpi_i)_a$,
 a spectral shift of the fundamental module having a global basis,
 and this is slightly different from our $L(\bm{\varpi}_{i,a})$.
 The precise relations are
 \[ V(\varpi_i)_a \cong L(\bm{\varpi}_{i,o(i)(-a\vartheta)^{d_i}}) \ \text{ with }\vartheta =(-1)^{\langle \rho^\vee,\gd\rangle}q^{-(\rho,\gd)},
 \]
 where $\rho$ (resp.~$\rho^\vee$) is an element such that $\langle h_i,\rho\rangle = 1$ (resp.~$\langle \rho^\vee,\ga_i\rangle = 1$) for all $i \in I$,
 and $o\colon I_0 \to \{\pm 1\}$ is a map such that $o(i) = -o(j)$ if $a_{ij}\neq 0$.
 In untwisted types this is proved in \cite{nakajima2004extremal}, and can be proved similarly in twisted types.
\end{Rem}

If $\fg$ is of untwisted type, we define an \textit{$\ell$-root} $\bm{\ga}_{i,a} \in \cP$ for $i \in I_0$ and $a \in \bk^\times$ by
\begin{equation}\label{eq:l_root}
 \bm{\ga}_{i,a} = \bm{\varpi}_{i,aq_i}+\bm{\varpi}_{i,aq_i^{-1}}-\sum_{j\neq i} \sum_{k=1}^{-a_{ji}}\bm{\varpi}_{i,aq^{-a_{ji}+1-2k}}.
\end{equation}
When $\fg$ is of twisted type, following \cite{MR2576287} we define $\ell$-roots $\bm{\ga}_{i,a}$ as follows. 
Let $(\mathsf{g},\tau)$ be the pair of a simple Lie algebra and an automorphism of the Dynkin diagram of $\sg$ used in the loop realization of $\fg$ (see \cite[Ch.~7]{MR1104219}).
Let $A^{\mathsf{g}}=(a_{ij}^{\mathsf{g}})_{i,j\in J}$ be the Cartan matrix of $\mathsf{g}$, 
and choose a subset $J^\tau \subseteq J$ of representatives for the $\tau$-orbits, which is naturally identified with $I_0$.
We choose $J^\tau$ so that for all $i,j \in J^\tau(=I_0)$, $a_{ij} \neq 0$ implies $a_{ij}^{\mathsf{g}}\neq 0$.
Let $r \in \{2,3\}$ be the order of $\tau$, and $\zeta$ a primitive $r$-th root of unity.
Let $\mathcal{P}_{\mathsf{g}} = \bigoplus \Z \bm{\varpi}_{i,a}^{\mathsf{g}}$ be the $\ell$-weight lattice of $U_q'(\wh{\mathsf{g}})$ with $\wh{\sg}$ the untwisted affine Lie algebra associated with
$\sg$,
and define a $\Z$-linear homomorphism $\mathbf{t}\colon \mathcal{P}_{\mathsf{g}} \to \mathcal{P}$ by setting 
\begin{equation}\label{eq:map_t}
 \mathbf{t}\big(\bm{\varpi}_{\tau^p(i),a}^{\mathsf{g}}\big)= \bm{\varpi}_{i,(\zeta^pa)^{d_i}} \ \ \text{ for $i \in J^{\tau}$, $a \in \bk^\times$ and $0\leq p <r$}.
\end{equation} 
Then for $i \in J^\tau(=I_0)$ and $a \in \bk^\times$, we set $\bm{\ga}_{i,a^{d_i}}=\mathbf{t}(\bm{\ga}_{i,a}^{\mathsf{g}})\in \cP$,
where $\bm{\ga}_{i,a}^{\mathsf{g}} \in \cP_{\sg}$ is an $\ell$-root of $U_q'(\wh{\mathsf{g}})$ defined above.
Note that we have $\mathsf{cl}(\bm{\ga}_{i,a}) = \ga_i$.

Now we return to the case where $\fg$ is of general affine type, that is, assume that $\fg$ is untwisted or twisted.
We call the sublattice $\mathcal{Q} = \bigoplus_{(i,a)} \Z \bm{\ga}_{i,a} \subseteq \mathcal{P}$ the \textit{$\ell$-root lattice}, and set 
$\mathcal{Q}^+= \sum_{(i,a)} \Z_{\geq 0}\bm{\ga}_{i,a} \subseteq \mathcal{Q}$.
We define a partial order $\leq$ on $\mathcal{P}$ by the condition that for $\bm{\pi},\bm{\gs} \in \mathcal{P}$, we say $\bm{\pi} \leq \bm{\gs}$ if and only if $\bm{\gs}-\bm{\pi} \in 
\mathcal{Q}^+$.

\begin{Thm}\label{Thm:Chari_Moura}
 If $M \in \mathcal{C}$ is indecomposable,
 then there exists $\bm{\pi}\in \mathcal{P}^+$ such that any composition factor of $M$ is isomorphic to a simple module of the form $L(\bm{\gs})$ with $\bm{\gs} 
  \in (\bm{\pi}+\mathcal{Q})\cap \mathcal{P}^+$.
\end{Thm}

\begin{proof}
 If $\fg$ is untwisted this assertion follows from \cite[Theorem 8.3 (i)]{chari2005characters},
 and in twisted types the same proof also works by applying \cite[Theorem 5.5]{MR2576287}.
\end{proof}

For $M \in \cC$, define its left dual module ${ }^*M$ (resp.~right dual module $M{ }^*$) by the dual space $\Hom{\bk}(M,\bk)$ on which $U_q'$ acts from the left via $S$ (resp.~$S^{-1}$).
For any $M_1,M_2 \in \cC$, we have 
\begin{equation}\label{eq:tensor_and_dual}
 {}^*(M_1 \otimes M_2) \cong {}^*(M_2) \otimes {}^*(M_1) \ \text{ and } \ (M_1\otimes M_2)^* \cong (M_2)^* \otimes (M_1)^*.
\end{equation}
By ${}^*(-)\colon \cP \to \cP$ we denote the $\Z$-linear map determined by ${}^*L(\bm{\pi}) \cong L({}^*\bm{\pi})$ for $\bm{\pi} \in \cP^+$, and define $(-)^*\colon \cP \to \cP$ similarly.
It is known that there is $p_* \in \bk^\times$ satisfying ${}^*\bm{\varpi}_{i,a}=\bm{\varpi}_{i,ap_*^{d_i}}$ and $\bm{\varpi}_{i,a}^*=\bm{\varpi}_{i,ap_*^{-d_i}}$ for all 
$i \in I_0$ and $a \in \bk^\times$
(see \cite{MR1607008} for an explicit description of $p_*$, which we do not need in this paper).


\subsection{Affinizations, extremal weight modules and Weyl modules}

Take a $\Z$-linear section $\iota\colon P_\cl \to P$ of the projection $\cl$ such that $\ga_i \in \iota(P_\cl)$ for all $i \in I_0$.
For $M \in \cC$, the \textit{affinization} $M_{\mathrm{aff}}$ of $M$ is an integrable $U_q$-module 
such that $M_{\mathrm{aff}} = M \otimes \bk[z^{\pm 1}]$ as a $\bk$-vector space,
\[ (M_{\mathrm{aff}})_{\iota(\gl)+m\gd} = M_{\gl} \otimes z^m \ \text{ for } \gl \in P_\cl, \ m \in \Z,
\]
and the actions of $e_i$ and $f_i$ are given by $e_i \otimes z^{\gd_{0i}}$ and $f_i \otimes z^{-\gd_{0i}}$ respectively.
We denote by $z_M$ the $U_q'$-module automorphism of $M_{\aff}$ given by the multiplication of $z$.
We sometimes write $M_z$ for $M_{\aff}$.
We also define
\begin{equation}\label{eq:completion}
 \widehat{M}_{\aff} = M_{\aff} \otimes_{\bk[z^{\pm 1}_M]} \bk\llbracket (z_M -1) \rrbracket,
\end{equation}
where $\bk\llbracket (z_M-1)\rrbracket$ denotes the completion $\underset{N}{\varprojlim} \,\,\bk[z_M]\Big/(z_M-1)^N\bk[z_M]$.

For $\gl \in P_0^+$, let $\mathbb{V}(\gl)$ denote the \textit{extremal weight module} associated with $\gl$, which is introduced in \cite{MR1262212}.
The module $\mathbb{V}(\gl)$ is an integrable $U_q$-module generated by a single vector $v_\gl$ of $P$-weight $\iota(\gl)$ with certain defining relations, 
and characterized by the following universal property \cite[Theorem 5.3]{MR1890649}.

\begin{Prop}\label{Prop:universality}
 For $\gl \in P_0^+$, the vector $v_\gl \in \bV(\gl)$ satisfies $x_{i,m}^+v_\gl=0$ for all $i \in I_0$ and $m \in \Z$. 
 Moreover, if $M$ is an integrable $U_q'$-module and $v \in M_\gl$ satisfies $x_{i,m}^+v=0$ for all $i\in I_0$ and $m \in \Z$, then there is a unique $U_q'$-module homomorphism
 from $\mathbb{V}(\gl)$ to $M$ mapping $v_\gl$ to $v$. 
\end{Prop}

\begin{Prop}[{\cite[Theorem 5.17]{MR1890649}}]\label{Prop:extremal_and_affinization}
 Let $i \in I_0$. \\
 {\normalfont(i)} There exists a $U_q'$-linear automorphism $z_{\bV(\varpi_i)}$ of $\bV(\varpi_i)$ of weight $d_i\gd$, and we have $\End_{U_q'}\big(\bV(\varpi_i)\big) 
  = \bk[z_{\bV(\varpi_i)}^{\pm 1}]$.\\
 {\normalfont(ii)} For any $a \in \bk^\times$, there is a $U_q$-module isomorphism
  \[ L(\bm{\varpi}_{i,a})_{\aff} \cong \bV(\varpi_i) \otimes_{\bk[z_{\bV(\varpi_i)}^{\pm 1}]} \bk[z_{\bV(\varpi_i)}^{\pm 1/d_i}].
  \]
  In particular, $L(\bm{\varpi}_{i,a})_{\aff} \cong L(\bm{\varpi}_{i,b})_{\aff}$ as $U_q$-modules for any $a,b\in \bk^\times$.
\end{Prop}

\begin{Rem}\normalfont
 If $a \neq b$, $L(\bm{\varpi}_{i,a})_{\aff}$ and $L(\bm{\varpi}_{i,b})_{\aff}$ are not isomorphic as $(U_q,\bk[z^{\pm 1}])$-bimodules.
 Indeed, for a $U_q$-module isomorphism $f\colon L(\bm{\varpi}_{i,1})_{\aff} \stackrel{\sim}{\to} L(\bm{\varpi}_{i,a})_{\aff}$,
 we have
 \begin{equation}\label{eq:isom_of_affinizations}
  z_{L(\bm{\varpi}_{i,1})}^{d_i} = af^{-1}z_{L(\bm{\varpi}_{i,a})}^{d_i}f 
 \end{equation} 
 by (\ref{eq:ell-weight_space}).
\end{Rem}

In the sequel, we normalize the automorphism $z_{\bV(\varpi_i)}$ in Proposition \ref{Prop:extremal_and_affinization} 
so that $z_{\bV(\varpi_i)}$ coincides with $z_{L(\bm{\varpi}_{i,1})}^{d_i}$ through the isomorphism 
$L(\bm{\varpi}_{i,1})_{\aff} \cong \bV(\varpi_i) \otimes_{\bk[z_{\bV(\varpi_i)}^{\pm 1}]} \bk[z_{\bV(\varpi_i)}^{\pm 1/d_i}]$.
For $i \in I_0$ and $a \in \bk^\times$, it follows from (\ref{eq:isom_of_affinizations}) that
\begin{equation}\label{eq:relations_among_fundamental}
 \bV(\varpi_i)/ \langle z_{\bV(\varpi_i)}-a\rangle \cong L(\bm{\varpi}_{i,a}) \ \text{ as $U_q'$-modules},
\end{equation}
where $\langle z_{\bV(\varpi_i)}-a \rangle$ is the ideal of $\mathrm{End}_{U_q'}\big(\bV(\varpi_i)\big)$ generated by $z_{\bV(\varpi_i)}-a$.

For $\gl \in P_0^+$, set 
\begin{equation}\label{eq:Annihilators}
 \mathrm{Ann}_\gl = \{ H \in U_q(0)\mid Hv_{\gl} = 0\},
\end{equation}
where $v_\gl \in \bV(\gl)$ is the generator.
Define a commutative algebra $\mathbf{A}_\gl$ by the quotient algebra $U_q(0)/\mathrm{Ann}_\gl$.
Then we have $\bV(\gl)_\gl \cong \mathbf{A}_\gl$.
It is easily seen from Proposition \ref{Prop:universality} that a $U_q'$-linear right $\mathbf{A}_\gl$-action on $\bV(\gl)$ is defined by
\[ (Xv_\gl)a = Xav_\gl \ \ \text{ for } X \in U_q', \ a \in \mathbf{A}_\gl,
\]
and hence $\bV(\gl)$ becomes a $(U_q',\bA_\gl)$-bimodule.
Moreover, this action induces an isomorphism $\bA_\gl \stackrel{\sim}{\to} \mathrm{End}_{U_q'}\big(\bV(\gl)\big)$ of $\bk$-algebras.
The following lemma is a quantum analog of results in \cite{chari2010categorical} (see also \cite{chari2013prime}).

\begin{Lem}\label{Lem:Chari-Fourier-Khandai} 
 Let $\gl \in P_0^+$, and assume that $M$ is an integrable $U_q'$-module satisfying $M_\mu = 0$ for $\mu \in P_0^+$ such that $\mu > \gl$.\\
 {\normalfont(i)} The ideal $\mathrm{Ann}_\gl$ acts trivially on the weight space $M_\gl$, and hence $M_\gl$ is an $\bA_\gl$-module.\\
 {\normalfont(ii)} Let $T$ be an $\bA_\gl$-module and $f\in \Hom{\bA_\gl}(T,M_{\gl})$. 
  Then the map 
  \begin{equation}\label{eq:induced_map_A_to_W}
   \bV(\gl) \otimes_{\bA_\gl} T \to M \colon Xv_{\gl} \otimes t \mapsto Xf(t) \ \text{ for } X\in U_q', \ t \in T
  \end{equation}
  gives a $U_q'$-module homomorphism, and this induces an isomorphism
  \begin{equation}\label{eq:adjoint}
   \Hom{\bA_\gl}(T,M_\gl) \cong \Hom{U_q'}(\bV(\gl)\otimes_{\bA_\gl}T,M).
  \end{equation}
\end{Lem}

\begin{proof}
 For any $v \in M_\gl$, there is a $U_q'$-module homomorphism from $\bV(\gl)$ to $M$ mapping $v_\gl$ to $v$ by Proposition \ref{Prop:universality}, and hence the assertion (i) follows.
 By the same proposition, we see that the map $\bV(\gl) \otimes_\bk T \to M$ defined similarly as (\ref{eq:induced_map_A_to_W}) gives a $U_q'$-module homomorphism.
 Since this map sends vectors $Xv_\gl a \otimes t - Xv_\gl \otimes at$ ($X \in U_q'$, $a \in \bA_\gl$, $t\in T$) to zero,
 the map (\ref{eq:induced_map_A_to_W}) is a well-defined $U_q'$-module homomorphism.
 Since the restriction to $(\bV(\gl)\otimes_{\bA_\gl}T)_\gl \cong T$ gives the inverse, we see that the isomorphism (\ref{eq:adjoint}) holds,
 and the assertion (ii) is proved.
\end{proof}



The following important results follow from \cite[Section 4]{MR2066942}.

\begin{Thm}\label{Thm:properties_of_extremal}
 Let $\gl = \sum_{i \in I_0} \gl_i \varpi_i \in P_0^+$.\\
 {\normalfont (i)} There is a unique injective $U_q$-module homomorphism 
 \[ \Phi_\gl \colon \bV(\gl) \hookrightarrow \bigotimes_{i \in I_0} \bV(\varpi_i)^{\otimes \gl_i} 
 \]
 mapping $v_{\gl}$ to $\bigotimes v_{\varpi_i}^{\otimes \gl_i}$. 
 Here we can take any ordering to define the right-hand side.\\
 {\normalfont (ii)} The image $\mathrm{Im}\,\Phi_\gl$ is preserved by the subalgebra 
  \begin{equation}\label{eq:isom_of_A}
   \bigotimes_{i \in I_0} \bk[z_{i1}^{\pm 1},\ldots,z_{i\gl_i}^{\pm 1}]^{\mathfrak{S}_{\gl_i}}\subseteq \End_{U_q'}\Big(\bigotimes_{i \in I_0} \bV(\varpi_i)^{\otimes \gl_i}\Big),
  \end{equation}
  where $z_{ik}=\cdots \otimes 1 \otimes z_{\bV(\varpi_i)} \otimes 1 \otimes \cdots$ acts on the $k$-th factor of $\bV(\varpi_i)^{\otimes \gl_i}$,
  and $\mathfrak{S}_m$ is the symmetric group on $m$ letters.
  Moreover, $\End_{U_q'}\big(\bV(\gl)\big)$ coincides with this subalgebra through $\Phi_\gl$.\\
 {\normalfont(iii)} 
 The module $\mathbb{V}(\gl)$ is free over $\End_{U_q'}\big(\bV(\gl)\big)$ of finite rank.
\end{Thm}

Since $\bA_{\gl} \cong \End_{U_q'}\big(\bV(\gl)\big)$ as mentioned above, we have
\begin{equation}\label{eq:isomorphism_of_A_End}
 \bA_\gl \cong \bigotimes_{i \in I_0} \bk[z_{i1}^{\pm 1},\ldots,z_{i\gl_i}^{\pm 1}]^{\mathfrak{S}_{\gl_i}}
\end{equation}
by (ii).
For $\gl, \gl' \in P_0^+$, we define an injective homomorphism $\gD_{\gl,\gl'}\colon \bA_{\gl+\gl'} \to \bA_{\gl}\otimes \bA_{\gl'}$ by
\begin{align*}
 \gD_{\gl,\gl'}\colon \bA_{\gl+\gl'} &\cong \bigotimes_{i\in I_0} \bk[z_{i1}^{\pm 1},\ldots, z_{i(\gl_i+\gl_i')}^{\pm 1}]^{\mathfrak{S}_{\gl_i+\gl_i'}}\\ &\hookrightarrow
   \bigotimes_{i \in I_0} \bk[z_{i1}^{\pm 1},\ldots,z_{i\gl_i}^{\pm 1}]^{\mathfrak{S}_{\gl_i}} \otimes \bigotimes_{i \in I_0} \bk[z_{i(\gl_i+1)}^{\pm 1},\ldots,
   z_{i(\gl_i+\gl_i')}^{\pm 1}]^{\mathfrak{S}_{\gl_i'}} \cong \bA_{\gl} \otimes \bA_{\gl'}.
\end{align*}
Similarly we define $\gD_{\gl_1,\ldots,\gl_p}\colon \bA_{\gl_1+\cdots+\gl_p} \to \bA_{\gl_1}\otimes \cdots \otimes \bA_{\gl_p}$ for $\gl_1,\ldots,\gl_p \in P_0^+$ with $p>1$.

\begin{Lem}\label{Lem:coproduct_of_A}
 For $k=1,2$, let $\gl_k \in P_0^+$ and $M_k$ be an integrable $U_q'$-module satisfying $(M_k)_\mu =0$ for $\mu \in P^+_0$ such that $\mu > \gl_k$.
 Then $\bA_{\gl_1+\gl_2}$ acts on $(M_1)_{\gl_1} \otimes (M_2)_{\gl_2}$ via $\gD_{\gl_1,\gl_2}$.
\end{Lem}

\begin{proof}
 Since there are injective homomorphisms
 \[ \bV(\gl_1+\gl_2) \hookrightarrow \bV(\gl_1)\otimes \bV(\gl_2) \hookrightarrow \bigotimes_{i \in I_0} \bV(\varpi_i)^{\otimes (\gl_1)_i+(\gl_2)_i},
 \] 
 we see from Theorem \ref{Thm:properties_of_extremal} that the assertion holds for $M_1=\bV(\gl_1)$ and $M_2=\bV(\gl_2)$.
 Then for any $v_k \in (M_k)_{\gl_k}$ ($k=1,2$), by considering the $U_q'$-module homomorphism $\bV(\gl_1) \otimes \bV(\gl_2) \to M_1 \otimes M_2$ mapping $v_{\gl_1}\otimes v_{\gl_2}$ to
 $v_1\otimes v_2$, $H(v_1 \otimes v_2) = \gD_{\gl_1,\gl_2}(H)(v_1 \otimes v_2)$ follows for $H\in \bA_{\gl_1+\gl_2}$. The proof is complete.
\end{proof}

For $\bm{\pi} \in \cP^+$ with $\gl = \mathsf{cl}(\bm{\pi})$, let $\bk_{\bm{\pi}}$ be the $1$-dimensional module over $U_q(0)$ corresponding to $\Psi$ in (\ref{eq:ell-weight_space}).
Because of the existence of a homomorphism $\bV(\gl) \twoheadrightarrow L(\bm{\pi})$, $\bk_{\bm{\pi}}$ is also an $\bA_\gl$-module.
Let $\mathfrak{m}_{\bm{\pi}}$ be the corresponding maximal ideal of $\bA_\gl$.
Note that, by (\ref{eq:isomorphism_of_A_End}), $\mathrm{Specm}\, \bA_{\gl}$ is identified with the product $\prod_{i \in I_0} (\bk^\times)^{\gl_i}\big/ \mathfrak{S}_{\gl_i}$ of quotient spaces.

\begin{Lem}\label{Lem:specm}
 If $\bm{\pi} = \sum_{(i,a)} l_{i,a}\bm{\varpi}_{i,a}$, then
 the maximal ideal $\mathfrak{m}_{\bm{\pi}}$ corresponds to the point $\left(\sum_{a \in \bk^\times} l_{i,a} 
 [a]\right)_{i\in I_0} \in \prod_{i \in I_0} (\bk^\times)^{\gl_i}\big/ \mathfrak{S}_{\gl_i} = \mathrm{Specm}\, \bA_{\gl}$.
 Here we identify a point of the space $(\bk^\times)^N\big/\mathfrak{S}_N$ with a $\Z_{\geq 0}$-linear combination of formal symbols $[a]$ 
 $(a \in \bk^\times)$ whose coefficients sum up to $N$.
\end{Lem}

\begin{proof}
 By (\ref{eq:tensor_lweight2}), we have
 \[ \bk_{\bm{\pi}} \cong \bigotimes_{i,a}\big(L(\bm{\varpi}_{i,a})_{\bm{\varpi}_{i,a}}\big)^{\otimes l_{i,a}} \ \text{ as $\bA_{\gl}$-modules}.
 \]
 Then the assertion follows from (\ref{eq:relations_among_fundamental}) and Lemma \ref{Lem:coproduct_of_A}.
\end{proof}


\begin{Def}[\cite{MR1850556}]\normalfont
 We define the \textit{local Weyl module} $W(\bm{\pi})$ associated with $\bm{\pi} \in \cP^+$ by
 \[ W(\bm{\pi}) = \bV(\gl) \otimes_{\bA_\gl} \bk_{\bm{\pi}}.
 \]
 We set $w_{\bm{\pi}} = v_{\gl} \otimes_{\bA_\gl} 1 \in W(\bm{\pi})$.
\end{Def}

The following lemma is immediate from the construction.

\begin{Lem}
 {\normalfont(i)} The module $W(\bm{\pi})$ is an $\ell$-highest weight module with $\ell$-highest weight vector $w_{\bm{\pi}}$ and $\ell$-highest weight $\bm{\pi}$.
 Moreover, if $M \in \cC$ is an $\ell$-highest weight module with $\ell$-highest weight $\bm{\pi}$, then there is a surjective $U_q'$-module homomorphism 
 from $W(\bm{\pi})$ to $M$.\\
 {\normalfont(ii)} The head of $W(\bm{\pi})$ is isomorphic to $L(\bm{\pi})$.
\end{Lem}

For $\bm{\pi}\in \cP^+$, we also define the \textit{dual local Weyl module} $W^\vee(\bm{\pi})$ by
\[ W^\vee(\bm{\pi}) = W({}^*\bm{\pi})^*.
\]

Finally we introduce deformed local Weyl modules, following \cite{fujita2017affine}.
For $\bm{\pi}  \in \cP^+$ with $\mathsf{cl}(\bm{\pi}) = \gl$, define a $\bk$-algebra
\[ \wh{\bA}_{\bm{\pi}} = \varprojlim_N \bA_\gl / \fm_{\bm{\pi}}^N,
\]
and let $\wh{\fm}_{\bm{\pi}}$ be the unique maximal ideal of $\wh{\bA}_{\bm{\pi}}$.
Write $\gl=\sum_{i\in I_0} \gl_i\varpi_i$ and $\bm{\pi}= \sum_{(i,a)} l_{i,a}\bm{\varpi}_{i,a}$,
and let us identify $\bA_{\gl}$ with $\bigotimes_{i \in I_0} \bk[z_{i1}^{\pm 1},\ldots,z_{i\gl_i}^{\pm 1}]^{\mathfrak{S}_{\gl_i}}$ via the isomorphism (\ref{eq:isomorphism_of_A_End}).
Then by Lemma \ref{Lem:specm}, the canonical injection 
\[ \bigotimes_{i \in I_0} \bk[z_{i1}^{\pm 1},\ldots,z_{i\gl_i}^{\pm_1}]^{\mathfrak{S}_{\gl_i}}\hookrightarrow \bigotimes_{i\in I_0} \bigotimes_{a \in \bk^\times} 
   \bk[z_{(i,a),1}^{\pm 1},\ldots,z_{(i,a),l_{i,a}}^{\pm 1}]^{\mathfrak{S}_{l_{i,a}}}
\]
induces an isomorphism
\begin{equation}\label{eq:natural_isomorphism_of_hA}
 \wh{\bA}_{\bm{\pi}} \cong \underset{i \in I_0}{\hat{\bigotimes}}\ \underset{a \in \bk^\times}{\hat{\bigotimes}} \bk\llbracket (z_{(i,a),1} - a),\ldots,
   (z_{(i,a),l_{i,a}}-a)\rrbracket^{\mathfrak{S}_{l_{i,a}}}.
\end{equation}


\begin{Def}[\cite{fujita2017affine}]\normalfont
 For $\bm{\pi} \in \cP^+$ with $\mathsf{cl}(\bm{\pi})=\gl$, we define the \textit{deformed local Weyl module} $\wh{W}(\bm{\pi})$ associated with $\bm{\pi}$ by
 \[ \wh{W}(\bm{\pi}) = \bV(\gl) \otimes_{\bA_\gl} \wh{\bA}_{\bm{\pi}}.
 \]
 We set $\wh{w}_{\bm{\pi}}=v_\gl \otimes 1 \in \wh{W}(\bm{\pi})$.
\end{Def}

Recall the definition of $\wh{M}_{\mathrm{aff}}$ in (\ref{eq:completion}).
Since the natural injection $\bk[z^{\pm d}] \hookrightarrow\bk[z^{\pm 1}]$ induces an isomorphism $\bk\llbracket (z^{d}-1)\rrbracket \stackrel{\sim}{\to} 
\bk\llbracket (z-1)\rrbracket$ for any $d \in \Z_{>0}$, we see from Proposition \ref{Prop:extremal_and_affinization} and (\ref{eq:relations_among_fundamental}) that
\begin{equation}\label{eq:W_and_affinization}
 \wh{W}(\bm{\varpi}_{i,a}) \cong \wh{L(\bm{\varpi}_{i,a})}_{\aff} \ \text{ for } i \in I_0, \ a \in \bk^\times.
\end{equation}

\begin{Lem}\label{Lem:properties_of_deformed_local}
 Let $\bm{\pi} \in \cP^+$, and set $\gl = \mathsf{cl}(\bm{\pi})$.\\
 {\normalfont(i)} If $M \in \cC$ satisfies $M_\mu = 0$ for $\mu \in P_0^+$ such that $\mu > \gl$, taking the image of $\wh{w}_{\bm{\pi}}$ induces a natural isomorphism
 \begin{equation}\label{eq:taking_highest_weight}
  \Hom{U_q'}\big(\wh{W}(\bm{\pi}),M\big) \stackrel{\sim}{\to} M_{\bm{\pi}}.
 \end{equation}
 {\normalfont(ii)} We have $\End_{U_q'}\big(\wh{W}(\bm{\pi})\big)\cong \wh{\bA}_{\bm{\pi}}$, and $\wh{W}(\bm{\pi})$ is free over $\wh{\bA}_{\bm{\pi}}$ of finite rank.\\
 {\normalfont(iii)} We have $\wh{W}(\bm{\pi}) / \wh{\fm}_{\bm{\pi}} \cong W(\bm{\pi})$.
\end{Lem}

\begin{proof}
 By Lemma \ref{Lem:Chari-Fourier-Khandai} (ii), we have 
 \[   \Hom{U_q'}\big(\wh{W}(\bm{\pi}),M\big) \cong \Hom{\bA_\gl} (\wh{\bA}_{\bm{\pi}}, M_\gl).
 \]
 Since $M_\gl$ is finite-dimensional, there is some $N$ such that the right-hand side coincides with $\Hom{\bA_\gl} (\wh{\bA}_{\bm{\pi}}/\wh{\fm}_{\bm{\pi}}^N, M_\gl)$,
 which is isomorphic to $M_{\bm{\pi}}$ since $\wh{\bA}_{\bm{\pi}}/\wh{\fm}_{\bm{\pi}}^N \cong \bA_{\gl}/\fm_{\bm{\pi}}^N$,
 and hence the assertion (i) is proved.
 The assertion (ii) follows from the same lemma and Theorem \ref{Thm:properties_of_extremal} (iii),
 and the assertion (iii) is obvious from the construction.
\end{proof}


\subsection{$R$-matrices}

For $k=1,2$, given a simple module $M_k \in \cC$ with $\ell$-highest weight vector $v_k\in M_k$, 
there is a unique $(U_q',\bk[z_1^{\pm 1},z_2^{\pm 1}])$-bimodule homomorphism
\begin{align*}
 R_{M_1,M_2}^{\mathrm{norm}}\colon
 (M_1)_{z_1} \otimes (M_2)_{z_2} \to \bk(z_2/z_1) \otimes_{\bk[(z_2/z_1)^{\pm 1}]} \big((M_2)_{z_2}
   \otimes (M_1)_{z_1}\big)
\end{align*}
satisfying $R_{M_1,M_2}^{\mathrm{norm}}(v_1 \otimes v_2) = v_2 \otimes v_1$ (\cite[\S8]{MR1890649}),
which is called the \textit{normalized $R$-matrix}.
The \textit{denominator} of $R_{M_1,M_2}^{\mathrm{norm}}$ is the monic polynomial $d_{M_1,M_2}(z) \in \bk[z]$ with the smallest degree
among polynomials satisfying
\[ d_{M_1,M_2}(z_2/z_1)\mathrm{Im}\,R_{M_1,M_2}^{\mathrm{norm}} \subseteq (M_2)_{z_2}
   \otimes (M_1)_{z_1}.
\]
It is known for any $i,j \in I_0$ that zeros of $d_{L(\bm{\varpi}_{i,1}),L(\bm{\varpi}_{j,1})}(z)$ belong to 
$q^{1/d}\,\ol{\Q}\llbracket q^{1/d}\rrbracket$ for some $d \in \Z_{>0}$ (see \cite[Proposition 9.3]{MR1890649}).

\begin{Thm}\label{Thm:equivalence_for_Weyl_mod}
 Let $(i_1,a_1),\ldots,(i_p,a_p)$ be a sequence of elements of $I_0 \times \bk^\times$.
 Set $\bm{\pi} = \sum_{k=1}^p \bm{\varpi}_{i_k,a_k}$,
 and $L_k = L(\bm{\varpi}_{i_k,a_k})$ for $1\leq k \leq p$.\\[2pt]
 {\normalfont(i)} The following statements are equivalent:
 \begin{itemize}
  \item[(a)] The tensor product $L_1 \otimes \cdots \otimes L_p$ is generated by the tensor product of $\ell$-highest weight vectors;
  \item[(b)] The local Weyl module $W(\bm{\pi})$ is isomorphic to $L_1 \otimes \cdots \otimes L_p$ as $U_q'$-modules;
  \item[(c)] For any $1\leq r < s \leq p$, we have $d_{L_r,L_s}(1) \neq 0$.\\[-8pt]
 \end{itemize}
 {\normalfont(ii)} If $d_{L_r,L_s}(1)\neq 0$ for all $1\leq r < s \leq p$,
 then we have 
 \[ W^{\vee}(\bm{\pi}) \cong L_p \otimes L_{p-1} \otimes \cdots \otimes L_1
 \]
 as $U_q'$-modules.
\end{Thm}

\begin{proof}
 In assertion (i), the equivalence of (a) and (b) is proved in \cite[Theorem 7.5]{chari2005characters} and \cite[Theorem 5.5]{MR2576287},
 and the equivalence of (a) and (c) is proved in \cite[Proposition 9.4]{MR1890649}.
 The assertion (ii) follows from (i) and (\ref{eq:tensor_and_dual}).
\end{proof}



For any $p >1$ and $\bm{\pi}_1,\ldots,\bm{\pi}_p \in \cP^+$ with $\mathsf{cl}(\bm{\pi}_k)=\gl_k$, 
the map $\gD_{\gl_1,\ldots,\gl_p}\colon \bA_{\gl_1+\cdots+\gl_p} \to \bA_{\gl_1}\otimes \cdots \otimes \bA_{\gl_p}$ naturally extends to 
$\gD=\gD_{\bm{\pi}_1,\ldots,\bm{\pi}_p}\colon \wh{\bA}_{\bm{\pi}_1+\cdots+\bm{\pi}_p} \to \wh{\bA}_{\bm{\pi}_1} \hat{\otimes} \cdots 
\hat{\otimes} \wh{\bA}_{\bm{\pi}_p}$.
The following proposition is proved in \cite[Corollary 2.21]{fujita2017affine}.

\begin{Prop}\label{Prop:Fujita_isom}
 Let $\bm{\pi} = \sum_{k=1}^p l_k\bm{\varpi}_{i_k,a_k} \in \cP^+$ with $(i_r,a_r) \neq (i_s,a_s)$ for $r\neq s$,
 and assume that $d_{L_r,L_s}(1) \neq 0$ for any $1\leq r < s \leq p$, where we write $L_r = L(\bm{\varpi}_{i_r,a_r})$.
 Then we have 
 \[ \wh{W}(\bm{\varpi}_{i_1,a_1})^{\hat{\otimes} l_1}\hat{\otimes} \cdots  \hat{\otimes} \wh{W}(\bm{\varpi}_{i_p,a_p})^{\hat{\otimes} l_p} \cong \wh{W}(\bm{\pi})^{\oplus l_1!\cdots l_p!}
 \]
 as $(U_q',\wh{\bA}_{\bm{\pi}})$-bimodules, where $\wh{\bA}_{\bm{\pi}}$ acts on the left-hand side via $\gD$.
\end{Prop}

\subsection{Modified quantum affine algebras}\label{Subsection:Modified_QAA}

The \textit{level zero modified quantum affine algebra} denoted by $\widetilde{U}_q=\widetilde{U}_q(\fg)$ is defined by
\[ \widetilde{U}_q= \bigoplus_{\gl \in P_0} U_q'a_\gl, \ \ \ U_q'a_\gl= U_q' \Big/ \sum_{h \in P_\cl^{\vee}} U_q'(q^h-q^{\langle h,\gl\rangle}),
\]
where $a_\gl$ stands for the image of $1$ in the quotient.
The multiplication is given by $a_\gl a_\mu = \gd_{\gl\mu}a_\gl$ and $a_\gl x = x a_{\gl-\mu}$ for $x \in (U_q')_\mu$, where we set
\[ (U_q')_\mu = \{ X \in U_q'\mid q^hXq^{-h} =  q^{\langle h, \mu\rangle} X \text{ for } h \in P_\cl^\vee\}.
\]
By definition an integrable $U_q^0$-module, and in particular a module in $\cC$, is a $\widetilde{U}_q$-module.
We define a $\bk$-algebra anti-involution $\sharp$ on $\wti{U}_q$ by
\[ \sharp(e_i) = f_i, \ \ \ \sharp(f_i) = e_i,\ \ \ \sharp(q^h)=q^h, \ \ \ \sharp(a_\gl)=a_\gl
\]
for $i \in I$, $h \in P_\cl^\vee$, and $\gl \in P_0$.
Suppose that $\gl =\sum_{i\in I_0}\gl_i\varpi_i \in P_0^+$.
By \cite[Proposition 6.9]{MR2066942}, $\sharp$ preserves $U_q(0)a_\gl$ and $(\mathrm{Ann}_\gl) a_\gl$ (see (\ref{eq:Annihilators})),
and hence defines an involution on $\bA_\gl$.
Let $\bV(\gl)^{\sharp}$ be an $(\bA_\gl,\widetilde{U}_q)$-bimodule obtained by twisting the action of $\wti{U}_q\otimes \bA_\gl$ on $\bV(\gl)$ by $\sharp$.

Fix $\gl \in P_0^+$. 
We set $P_{0,\leq \gl}^+ =\{ \mu \in P_0^+\mid \mu \leq \gl\}$, and
define a $\bk$-algebra $U_{\leq \gl}$ by 
\[ U_{\leq \gl} = \wti{U}_q \Big/ \bigcap_{\mu \in P_{0,\leq \gl}^+} \mathrm{Ann}_{\wti{U}_q}\bV(\mu),
\]
where $\mathrm{Ann}_{\wti{U}_q}M$ denotes the annihilator of a $\wti{U}_q$-module $M$.
Fix a numbering $\{\gl_1,\ldots,\gl_l\}$ on the elements of $P_{0,\leq \gl}^+$ such that 
$\gl_l = \gl$ and $r<s$ holds whenever $\gl_r < \gl_s$.
Define two-sided ideals $U_r$ ($0\leq r \leq l$) of $U_{\leq \gl}$ by
\[ U_0 = U_{\leq \gl}, \ \ \ U_r = \bigcap_{1\leq s\leq r} \mathrm{Ann}_{U_{\leq \gl}} \bV(\gl_s) \ \text{ for } 1\leq r \leq l.
\] 
Note that $U_l=\{0\}$.

\begin{Thm}[\cite{cui2015affine,nakajima2015affine}]\label{Thm:affine_cellularity_of_U}
 For any $1\leq r \leq l$, there is a unique $(U_q',U_q')$-bimodule isomorphism
 \[ U_{r-1} / U_r \stackrel{\sim}{\to} \bV(\gl_r) \otimes_{\bA_{\gl_r}} \bV(\gl_r)^{\sharp}
 \]
 mapping $\bar{a}_{\gl_r}$ to $v_{\gl_r} \otimes v_{\gl_r}$, where $\bar{a}_{\gl_r}$ is the image of $a_{\gl_r} \in U_{r-1}$ under the canonical projection. 
\end{Thm}

Later we need the following lemmas.

\begin{Lem}\label{Lem:quotient_to_U_gl}
 Let $1\leq r \leq l$, and assume that $M$ is an integrable $U_q'$-module such that $M_\mu = 0$ for $\mu \in P^+_0 \setminus \{\gl_1,\ldots,\gl_{r}\}$.
 Then $M$ is annihilated by $\bigcap_{1\leq s \leq r} \mathrm{Ann}_{\wti{U}_q}\bV(\gl_s)$, and hence 
 $M$ becomes a $U_{\leq \gl}/U_r$-module.
\end{Lem}

\begin{proof}
 Extend the above numbering of elements of $P_{0,\leq \gl}^+$ to all the elements of $P_0^+$ as $P_0^+ = \{\gl_1,\gl_2,\ldots,\}$,
 and consider the completion $\wh{U}_q = \underset{N}{\varprojlim}\,\, \wti{U}_q/ \sum_{s \geq N} \wti{U}_q a_{\gl_s} \wti{U}_q$, which has a natural $\bk$-algebra structure.
 Obviously $M$ extends to a $\wh{U}_q$-module, and it follows from \cite[Proposition 6.27]{MR2066942} that 
 \[ \bigcap_{1\leq s \leq r} \mathrm{Ann}_{\wti{U}_q}\bV(\gl_s) \subseteq \sum_{s> r} \wti{U}_q a_{\gl_s}\wti{U}_q
 \]
 in $\wh{U}_q$, which acts trivially on $M$.
 The assertion is proved.
\end{proof}
\begin{Lem}\label{Lem:dual_deformed}
 Let $\bm{\pi} = \sum_{k=1}^p l_k\bm{\varpi}_{i_k,a_k} \in \cP^+$ with $(i_r,a_r) \neq (i_s,a_s)$ for $r\neq s$, and set $\gl = \mathsf{cl}(\bm{\pi})$.
 Assume that $d_{L_r,L_s}(1) \neq 0$ for any $1\leq r < s \leq p$, where we write $L_r=L(\bm{\varpi}_{i_r,a_r})$. 
 Then we have
 \begin{equation}\label{eq:isom_of_dual} 
  \Hom{\wh{\bA}_{\bm{\pi}}}\big(\wh{W}(l_p\bm{\varpi}_{i_p,a_p}) \hat{\otimes} \cdots \hat{\otimes} \wh{W}(l_1\bm{\varpi}_{i_1,a_1}),\wh{\bA}_{\bm{\pi}}\big)
    \cong \wh{\bA}_{\bm{\pi}} \otimes_{\bm{\bA}_{\gl}} \bV(\gl)^{\sharp}
 \end{equation}
 as $(\wh{\bA}_{\bm{\pi}},U_q')$-bimodules.
\end{Lem}

\begin{proof}
 Since $\big(\wh{W}(l_p\bm{\varpi}_{i_p,a_p}) \hat{\otimes} \cdots \hat{\otimes} \wh{W}(l_1\bm{\varpi}_{i_1,a_1})\big)_\gl \cong \wh{\bA}_{\bm{\pi}}$,
 the operator $a_\gl$ can be considered as an element of the left-hand side of (\ref{eq:isom_of_dual}).
 Then by applying Lemma \ref{Lem:Chari-Fourier-Khandai}, we see that there is an $(\wh{\bA}_{\bm{\pi}},U_q')$-bimodule homomorphism from the right-hand side of 
 (\ref{eq:isom_of_dual}) to the left-hand side mapping $1 \otimes v_\gl$ to $a_\gl$. 
 By taking the quotient by $\wh{\fm}_{\bm{\pi}}$, we obtain a right $U_q'$-module homomorphism
 \[ \bk_{\bm{\pi}} \otimes_{\bA_{\gl}} \bV(\gl)^{\sharp} \to \Hom{\bk}\big(L(\bm{\varpi}_{i_p,a_p})^{\otimes l_p}\otimes \cdots \otimes L(\bm{\varpi}_{i_1,a_1})^{\otimes l_1},\bk\big),
 \]
 which is surjective since the right-hand side is generated by $a_\gl$ by \cite[Proposition 9.4 (ii)]{MR1890649}.
 Since the both sides of (\ref{eq:isom_of_dual}) are free $\wh{\bA}_{\bm{\pi}}$-modules of same rank, the assertion follows from Nakayama's lemma. 
\end{proof}

\section{Quiver Hecke algebras}\label{section:QHA}

In this section we shall review symmetric quiver Hecke algebras of type $ADE$.

\subsection{Symmetric quiver Hecke algebras}\label{Subsection:SQHA}

Throughout the rest of this paper, let $\sg$ denote a finite-dimensional simple Lie algebra of type $ADE$. 
Let $(A^{\sg},P_{\sg},\Pi_{\sg},P_\sg^{\vee},\Pi_\sg^\vee)$ be a Cartan datum for $\sg$
with index set $J=\{1,\ldots,N\}$, $A^{\sg}=(a_{\imath \jmath}^{\sg})_{\imath,\jmath \in J}$, $\Pi_{\sg}=\{\ga_1^{\sg},\ldots,\ga_N^{\sg}\}$, and $\Pi^\vee_\sg=\{h_1^\sg,\ldots,h_N^{\sg}\}$.
In the following, we mainly use the symbols $\imath,\jmath$ for elements of $J$, in order to save the symbols $i,j$ for elements of $I$.
Denote by $\gL_\imath^\sg$ ($\imath \in J$) the fundamental weights, and let $(\ , \ )$ denote the bilinear form on $P_\sg$ defined by $(\gL_\imath^\sg,\ga^\sg_{\jmath}) 
= \gd_{\imath\jmath}$ for $\imath,\jmath \in J$.
Let  $Q_\sg=\bigoplus_{\imath \in J} \Z\ga_{\imath}^\sg$ be the root lattice, and $Q^+_\sg=\sum_{\imath \in J} \Z_{\geq 0}\ga_{\imath}^\sg$.
We take a family of polynomials $(\mathsf{Q}_{\imath \jmath})_{\imath,\jmath\in J}$ in $\bk[u,v]$ which are of the form
\[ \mathsf{Q}_{\imath \jmath}(u,v) = \begin{cases} c_{\imath \jmath}(u-v)^{-a^\sg_{\imath \jmath}} & \text{if } \imath \neq \jmath,\\
                               0 & \text{if } \imath = \jmath,
                 \end{cases}
\]
where $c_{\imath \jmath} \in \bk^\times$ and we assume that $\mathsf{Q}_{\imath \jmath}(u,v) = \mathsf{Q}_{\jmath \imath}(v,u)$ for all $\imath,\jmath \in J$.
For $\gb =\sum_{\imath \in J} m_\imath \ga_\imath \in Q_\sg^+$, we set $|\gb| = \sum_{\imath \in J} m_\imath \in \Z_{\geq 0}$.
For $\gb\in Q_\sg^+$ such that $|\gb| = m$, we set
\[ J^\gb = \{\bm{\imath} =(\imath_1,\ldots,\imath_m) \in J^m \mid \ga_{\imath_1}^\sg+\cdots + \ga_{\imath_m}^\sg=\gb\}.
\]
We denote by $\gs_i$ the transposition of $i$ and $i+1$ in the symmetric group $\mathfrak{S}_m$.

\begin{Def}\normalfont
 For $\gb \in Q_\sg^+$ with $|\gb|=m$, the \textit{symmetric quiver Hecke algebra} $R(\gb)$ at $\gb$ associated with $\sg$ and $(Q_{\imath\jmath})_{\imath,\jmath\in J}$ 
 is the $\bk$-algebra with generators 
 \[ \{e(\bm{\imath})\mid \bm{\imath} \in J^\gb\}, \ \ \ \{x_k \mid 1\leq k \leq m\}, \ \ \ \{\tau_l \mid 1\leq l <m\}
 \]
 with the following relations:
 \begin{gather*}
  e(\bmi)e({\bm{\imath}}')=\gd_{\bmi\bmi'}e(\bmi), \ \ \sum_{\bmi \in J^\gb}e(\bmi)=1, \ \ x_kx_l=x_lx_k, \ \ x_ke(\bmi)=e(\bmi)x_k,\\
  \tau_le(\bmi)=e(\gs_l(\bmi))\tau_l, \ \ \tau_k\tau_l=\tau_l\tau_k \text{ if } |k-l| > 1,\ \ \tau_l^2e(\bmi)=Q_{\imath_l\imath_{l+1}}(x_l,x_{l+1})e(\bmi),\\
  (\tau_kx_l-x_{\gs_k(l)}\tau_k)e(\bmi)=\begin{cases} -e(\bmi) & \text{if $l=k$, $\imath_k=\imath_{k+1}$},\\
                                                   e(\bmi) & \text{if $l=k+1$, $\imath_k=\imath_{k+1}$},\\
                                                   0      & \text{otherwise},
                                     \end{cases}\\
  (\tau_{k+1}\tau_k\tau_{k+1}-\tau_k\tau_{k+1}\tau_k)e(\bmi) = \gd_{\imath_k,\imath_{k+2}}\frac{Q_{\imath_k\imath_{k+1}}(x_k,x_{k+1})-Q_{\imath_k\imath_{k+1}}(x_{k+2},x_{k+1})
   }{x_k-x_{k+2}}e(\bmi).
 \end{gather*}
\end{Def}

Let $\bP_\gb$ be a commutative subalgebra of $R(\gb)$ defined by 
\[ \bP_\gb = \bigoplus_{\bmi \in J^\gb} \bk[x_1,\ldots,x_m]e(\bmi).
\]
For each $w \in \mathfrak{S}_m$, we fix a reduced expression $w=\gs_{i_1}\cdots \gs_{i_{\ell}}$, and set $\tau_w = \tau_{i_1}\cdots \tau_{i_{\ell}}$.
Note that this $\tau_w$ depends on the choice of the reduced expression.
It is known (see \cite{khovanov2009diagrammatic}) that
\begin{equation}\label{eq:freeness}
 R(\gb) = \bigoplus_{w\in \mathfrak{S}_m} \bP_\gb\tau_w= \bigoplus_{w\in \mathfrak{S}_m} \tau_w \bP_\gb.
\end{equation}
The symmetric group $\mathfrak{S}_m$ acts on $\bP_\gb$ by 
\[ w(f(x_1,\ldots,x_m)e(\imath_1,\ldots,\imath_m)) = f(x_{w(1)},\ldots,x_{w(m)})e(\imath_{w^{-1}(1)},\ldots,\imath_{w^{-1}(m)}),
\]
and the center of $R(\gb)$ coincides with the invariant part $\left(\bP_\gb\right)^{\mathfrak{S}_m}$ (\cite{khovanov2009diagrammatic}).
Write $\gb = \sum_{\imath \in J} m_\imath\ga_\imath^\sg$, and define a commutative $\bk$-algebra $\bS_{\gb}$ by
\[ \bS_{\gb} = \bigotimes_{\imath=1}^N \bk[w_{\imath,1},\ldots,w_{\imath,m_\imath}]^{\mathfrak{S}_{m_\imath}}.
\]
Then the following $\bk$-algebra isomorphism holds:
\begin{align*}
  \bS_\gb \stackrel{\sim}{\to} &\left(\bP_\gb\right)^{\mathfrak{S}_m} \colon\\ &f_1\otimes \cdots\otimes f_N \mapsto \prod_{\imath=1}^N \big(m_\imath!\big)^{-1}\cdot \sum_{w \in \mathfrak{S}_m} 
  w\Big(\prod_{\imath=1}^N f(x_{\imath,1},\ldots,x_{\imath,m_\imath})e(1^{m_1}\cdots N^{m_N})\Big),
\end{align*}
where we put $x_{\imath,k} = x_{m_1+\cdots +m_{\imath-1}+k}$.
In the sequel we identify the center $\left(\bP_\gb\right)^{\mathfrak{S}_m}$ with $\bS_\gb$ via this isomorphism.
By (\ref{eq:freeness}), $R(\gb)$ is a free module over $\bS_\gb$ of rank $(m!)^2$.

The algebra $R(\gb)$ is equipped with a $\Z$-grading by assigning 
\[ \deg(e(\bmi))=0,\ \ \  \deg {x_k}e(\bmi) = (\ga^\sg_{\imath_k},\ga^\sg_{\imath_k}),\ \ \  \deg \tau_l e(\bmi) = -(\ga^\sg_{\imath_l},\ga^\sg_{\imath_{l+1}}).
\]
Let $\wh{R}(\gb)$ (resp.~$\wh{\bP}_\gb$) denote the completion of $R(\gb)$ (resp.~$\bP_\gb$) with respect to the grading.
We have $\wh{\bP}_\gb = \bigoplus_{\bmi \in J^\gb} \bk\llbracket x_1,\ldots,x_m\rrbracket e(\bmi)$, and it follows from (\ref{eq:freeness}) that
\[ \wh{R}(\gb) = \bigoplus_{w \in \mathfrak{S}_m} \wh{\bP}_\gb \tau_w = \bigoplus_{w \in \mathfrak{S}_m} \tau_w \wh{\bP}_\gb.
\]
Set 
\[ \wh{\bS}_{\gb} = \underset{1\leq \imath \leq N}{\hat{\bigotimes}} \bk\llbracket w_{\imath,1},\ldots,w_{\imath,m_{\imath}}\rrbracket^{\mathfrak{S}_{m_\imath}}.
\]
The center of $\wh{R}(\gb)$ is the $\mathfrak{S}_{m}$-invariant part $(\wh{\bP}_\gb)^{\mathfrak{S}_m}$ of $\wh{\bP}_\gb$, which is identified with $\wh{\bS}_{\gb}$.
Denote by $\modfd{\wh{R}(\gb)}$ the full subcategory of finite-dimensional modules in $\modfg{\wh{R}(\gb)}$.

\begin{Rem}\normalfont\label{Rem:completion}
 (1) The completion $\wh{R}(\gb)$ is naturally isomorphic to the central completion $R(\gb) \otimes_{\bS_\gb} \wh{\bS}_\gb$
 along the trivial central character.
 Hence $\modfd{\wh{R}(\gb)}$ is identified with the category of finite-dimensional $R(\gb)$-modules on which $x_k$'s act nilpotently.
 In particular, there is a forgetful functor from $R(\gb)\mathrm{\mathchar`-gmod}_{\mathrm{fd}}$, the category of finite-dimensional graded $R(\gb)$-modules, to $\modfd{\wh{R}(\gb)}$.\\
 (2) Let $\wh{\bS}_{\gb}^+$ denote the unique maximal ideal of $\wh{\bS}_{\gb}$. 
  Then for any $M \in \modfg{\wh{R}(\gb)}$, we have $M \cong \underset{k}{\varprojlim}\, M/\big(\wh{\bS}_\gb^+\big)^k$,
  and hence $M$ can be written as a projective limit of finite-dimensional modules.
\end{Rem}

For $\gb, \ggg\in Q_\sg^+$ with $|\gb|=m$, $|\ggg|=m'$, set
\[ e(\gb,\ggg) = \sum_{\begin{smallmatrix} \bmi \in J^{\gb+\ggg};\\ (\imath_1,\ldots,\imath_m)\in J^\gb\end{smallmatrix}} e(\bmi) \in R(\gb+\ggg).
\]
There is a $\bk$-algebra homomorphism $R(\gb)\otimes R(\ggg) \to e(\gb,\ggg)R(\gb+\ggg)e(\gb,\ggg)$ such that
\begin{gather*}
 e(\bmi) \otimes e(\bmi') \mapsto e(\bmi*\bmi'), \ \ \ \tau_k\otimes 1 \mapsto \tau_ke(\gb,\ggg),\ \ \ 1\otimes \tau_k \mapsto \tau_{m+k}e(\gb,\ggg),\\
 x_k \otimes 1 \mapsto x_ke(\gb,\ggg), \ \ \ 1\otimes x_k \mapsto x_{m+k}e(\gb,\ggg),
\end{gather*}
where $\bmi*\bmi'$ is the concatenation.
Similarly, a $\bk$-algebra homomorphism $\wh{R}(\gb) \otimes \wh{R}(\ggg) \to e(\gb,\ggg)\wh{R}(\gb+\ggg)e(\gb,\ggg)$ is defined.
For $M \in \modfg{\wh{R}(\gb)}$ and $N \in \modfg{\wh{R}(\ggg)}$, define the convolution product $M\circ N$ by
\[ M\circ N = \wh{R}(\gb+\ggg) e(\gb,\ggg) \otimes_{\wh{R}(\gb)\otimes \wh{R}(\ggg)}(M\otimes N),
\]
which belongs to $\modfg{\wh{R}(\gb+\ggg)}$.
For any $L \in \modfg{\wh{R}(\gb+\ggg)}$ there is a functorial isomorphism 
\begin{equation}\label{eq:induction}
 \Hom{\wh{R}(\gb+\ggg)} (M \circ N, L) \cong \Hom{\wh{R}(\gb)\otimes \wh{R}(\ggg)}( M  \otimes N, e(\gb,\ggg)L).
\end{equation}
We also have the following.
\begin{Prop}\label{Prop:coinduction}
 For $\gb,\ggg \in Q_\sg^+$ and $M \in \modfg{\wh{R}(\gb)}$, $N \in \modfg{\wh{R}(\ggg)}$, and $L \in \modfg{\wh{R}(\gb+\ggg)}$, there is a functorial isomorphism
 \begin{equation}\label{eq:coinduction}
  \Hom{\wh{R}(\gb+\ggg)} (L, N\circ M) \cong \Hom{\wh{R}(\gb)\hat{\otimes}\wh{R}(\ggg)}(e(\gb,\ggg)L, M \hat{\otimes} N).
 \end{equation}
\end{Prop}
\begin{proof}
 The assertion follows from \cite[Theorem 2.2]{lauda2011crystals} and Remark \ref{Rem:completion} (2).
\end{proof}
These definitions and isomorphisms (\ref{eq:induction}) and (\ref{eq:coinduction}) are extended in an obvious way to situations where there are more than two factors.



Let $\psi$ be the $\bk$-algebra anti-involution on $\wh{R}(\gb)$ which fixes all the generators $e(\bmi)$, $x_k$, and $\tau_l$.
For any $M \in \modfd{\wh{R}(\gb)}$, a left $\wh{R}(\gb)$-module structure is given on the dual space $M^* = \Hom{\bk}(M,\bk)$ via $\psi$ as in (\ref{eq:module_on_dual}).
If $M$ is simple, we have $M^* \cong M$ as $\wh{R}(\gb)$-modules, see \cite{khovanov2009diagrammatic}.

\subsection{Affine highest weight structures}\label{Subsection:Affine_highest_weight_category}


Let $R_\sg$ be the root system of $\sg$, $R_\sg^+ \subseteq R_\sg$ the set of positive roots,
$W_\sg \subseteq \mathrm{Aut}_\Z(P_\sg)$ the Weyl group with simple reflections $s_\imath$ ($\imath \in J$), and $w_0 \in W_\sg$ the longest element.
Denote by $\ell(w)$ the length of $w \in W_\sg$, and set $L = \ell(w_0)$.
We say a sequence $\bmi=(\imath_1,\ldots,\imath_l)$ of elements of $J$ is a \textit{reduced word} of $w_0$ if $l=L$ and $w_0=s_{\imath_1}\cdots s_{\imath_l}$.
Two reduced words $\bmi$ and $\bmj$ of $w_0$ are said to be \textit{commutation equivalent} if $\bmj$ is obtained from $\bmi$ by applying
a sequence of operations which transform some adjacent component $(\imath,\jmath)$ such that $a_{\imath\jmath}^\sg=0$ into $(\jmath,\imath)$.
This is an equivalence relation.

For $\gb \in Q_\sg^+$, let $KP(\gb) \subseteq \Z_{\geq 0}^{R_\sg^+}$ be the set of maps $\mathbf{m}$ from $R^+_\sg$ to nonnegative integers
such that $\sum_{\ga \in R_\sg^+} \bfm(\ga)\ga =\gb$.
We call $\bfm$ a \textit{Kostant partition} of $\gb$.
Throughout the rest of this section, we fix a reduced word $\bmi=(\imath_1,\ldots,\imath_L)$ of $w_0$,
and set
\begin{equation}\label{eq:def_of_bk}
 \gb_k = s_{\imath_1}\cdots s_{\imath_{k-1}} (\ga_{\imath_k}^\sg) \ \ \ \text{for } 1\leq k \leq L.
\end{equation}
Note that $R_\sg^+ = \{\gb_1,\ldots,\gb_L\}$, and this numbering is convex, that is,
\[ \gb_k+\gb_l = \gb_m \text{ implies } k < m < l \text{ or } l < m < k.
\]
We identify $\bfm \in KP(\gb)$ with an $L$-tuple $(\bfm_1,\ldots,\bfm_L)$ of nonnegative integers via $\bfm_k=\bfm(\gb_k)$.
Once we fix the reduced word $\bmi$, by \cite{kato2014poincare,mcnamara2015finite}, simple modules in $\modfg{\wh{R}(\gb)}$ are parametrized by Kostant partitions of $\gb$.
We denote the simple module corresponding to $\mathbf{m}$ by $S_{\bmi}(\mathbf{m})$,
or $S(\bfm)$ when no confusion is likely.
For $\ga \in R^+_\sg$, we denote by $\bfm_\ga \in KP(\ga)$ the Kostant partition such that $\mathbf{m}_\ga(\ga')=\gd_{\ga\ga'}$ for $\ga' \in R_\sg^+$,
and write $S(\ga)$ for $S(\mathbf{m}_\ga)$ for ease of notation.
For $\bfm \in KP(\gb)$, define a module $\bar{\gD}(\bfm)=\bar{\gD}_{\bmi}(\bfm) \in \modfd{\wh{R}(\gb)}$ by
\[ \bar{\gD}(\bfm) = S(\gb_1)^{\circ \bfm_1}\circ \cdots \circ S(\gb_L)^{\circ \bfm_L}.
\]
Then $S(\bfm)$ is isomorphic to the head of $\bar{\gD}(\bfm)$.
In addition, we define $\bar{\nabla}(\bfm)=\bar{\nabla}_{\!\bmi}(\bfm) \in \modfd{\wh{R}(\gb)}$ by $\bar{\nabla}(\bfm) = S(\gb_L)^{\circ \bfm_L} \circ \cdots \circ S(\gb_1)^{\circ \bfm_1}$.

Let $M \in \modfd{\wh{R}(\gb)}$, and $w$ be an indeterminate.
Set $\wh{M}_{\mathrm{aff}}= M \otimes \bk\llbracket w \rrbracket$, and define an action of $R(\gb)$ on $\wh{M}_{\mathrm{aff}}$ by
\[ e(\bmj)(v\otimes a) = \big(e(\bmj)v\big)\otimes a, \ \ x_k(v\otimes a) = v\otimes (wa)+(x_kv)\otimes a, \ \ \tau_k(v\otimes a) = (\tau_k v)\otimes a
\]
for $\bmj \in J^\gb$, $a\in \bk\llbracket w \rrbracket$ and $v \in M$. 
This action naturally extends to an action of $\wh{R}(\gb)$, and then $\wh{M}_{\mathrm{aff}}$ belongs to $\modfg{\wh{R}(\gb)}$.
For $\ga \in R_\sg^+$, let $\gD(\ga) = \gD_{\bmi}(\ga)$ denote the affinization $\wh{S(\ga)}_{\mathrm{aff}}$.

\begin{Rem}\normalfont
 The module $\gD(\ga)$ coincides with the completion of the module in \cite[Theorem 3.3]{brundan2014homological}. 
 This is easily shown using the uniqueness of extensions in [\textit{loc.~cit.}, Lemma 3.2].
\end{Rem}

For $\ga \in R^+_\sg$ and $m \in \Z_{>0}$, $\gD(\ga)^{\circ m}$ is isomorphic to the direct sum of $m!$-copies of an indecomposable module $\gD(m\ga)$ 
\cite[Theorem 3.11]{brundan2014homological}.
For $\bfm \in KP(\gb)$, define $\gD(\bfm)=\gD_{\bmi}(\bfm)$ by $\gD(\bfm)=\gD(\bfm_1\gb_1) \circ \cdots \circ \gD(\bfm_L\gb_L)$, which is also indecomposable.
Note that 
\[ \gD(\gb_1)^{\circ \bfm_1} \circ \cdots \circ \gD(\gb_L)^{\circ \bfm_L} \cong \gD(\bfm)^{\oplus \bfm_1!\cdots\bfm_L!}.
\]
For $\bfm \in \Z_{\geq 0}^{R^+_{\sg}}$, define a commutative $\bk$-algebra $\wh{\bS}_{\bfm}$ by
\[ \wh{\bS}_{\bfm} = \underset{1\leq k \leq L}{\hat{\bigotimes}} \bk\llbracket w_{k,1},\ldots,w_{k,\bfm_k}\rrbracket^{\mathfrak{S}_{\bfm_k}}.
\]
Note that, if $\gb = \sum_{\imath=1}^N m_{\imath} \ga_{\imath}^{\sg}$, then $\wh{\bS}_{\gb} = \wh{\bS}_{\bfm}$ with $\bfm = \sum_{\imath} m_\imath \bfm_{\ga_{\imath}^\sg}$.
For $\ga \in R^+_\sg$, let $\iota=\iota_{\bfm_\ga}\colon \wh{\bS}_\ga \to \wh{\bS}_{\bfm_\ga}=\bk\llbracket w\rrbracket$ denote the $\bk$-algebra homomorphism $f \mapsto f(w,w,\ldots,w)$.
Then for $\bfm \in KP(\gb)$, set $\iota=\iota_{\bfm} \colon \wh{\bS}_\gb \to \wh{\bS}_{\bfm}$ to be the $\bk$-algebra homomorphism obtained by restricting 
$\hat{\bigotimes}_{1\leq k \leq L}\iota_{\bfm_{\gb_k}}^{\hat{\otimes} \bfm_k}$ to $\wh{\bS}_\gb$.

\begin{Lem}\label{Lem:action_of_centers}
 For any $\bfm \in KP(\gb)$, the endomorphism $\bk$-algebra $\mathrm{End}_{\wh{R}(\gb)}\big(\gD(\bfm)\big)$ is isomorphic to $\wh{\bS}_{\bfm}$,
 and the action of the center $\wh{\bS}_\gb$ is given via $\iota$.
\end{Lem}

\begin{proof}
 The assertion follows from the construction of $\gD(\bfm)$ in \cite{brundan2014homological} (see also the proof of \cite[Theorem 6.11]{mcnamara2017representation}).
\end{proof}

Following \cite[Section 6]{mcnamara2017representation}, we define a partial order on $KP(\gb)$.
For each $1\leq k \leq L$, set 
\begin{equation}\label{eq:def_of_ggg}
 \nu_k = -s_{\imath_1}\cdots s_{\imath_k}(\gL^\sg_{\imath_k}) \in P_\sg.
\end{equation}
Although $\nu_k$ depends on the reduced word $\bmi$, we suppress this dependence.
For $\bfm,\bfn\in KP(\gb)$, we write $\bfm \preceq_{\bmi} \bfn$ (or $\bfm \preceq \bfn$) if and only if
\[ \sum_{t=1}^k \bfm_t( \nu_k, \gb_t) \leq \sum_{t=1}^k \bfn_t ( \nu_k, \gb_t) \ \ \text{ for all $1\leq k \leq L$}.
\]
The following lemma is easily checked from the definition.

\begin{Lem}[\cite{mcnamara2017representation}]\label{Lem:McNamara_order}
 {\normalfont(i)} For any $1\leq k \leq L$, we have 
  \[ (\nu_k, \gb_l) \geq 0 \ \text{ for } 1\leq l <k, \ \ \ (\nu_k,\gb_k) = 1, \ \text{ and } \ 
     (\nu_k,\gb_l) \leq 0 \ \text{ for } k< l \leq L.
  \]
 {\normalfont(ii)} Let $\rho=\rho^{\bmi}\colon \Z^L \to \Z^{L}$ denote a map defined by
  \[ \rho (\bfm) = \Big(\sum_{t=1}^k \bfm_t( \nu_k,\gb_t) \Big)_{1\leq k \leq L}.
  \]
  Then $\rho$ is injective.\\
 {\normalfont(iii)} Assume that $a_{\imath_k,\imath_{k+1}}^\sg=0$ for some $1\leq k \leq L-1$, and set $\bmj = \gs_{k}(\bmi)$, where $\gs_k \in \mathfrak{S}_L$ is the transposition of $k$ and $k+1$.
  Then for any $\bfm \in KP(\gb)$, $\rho^{\bmj}(\bfm) = \gs_{k}\big(\rho^{\bmi}(\bfm)\big)$.
  In particular, the partial order $\preceq_{\bmi}$ on $KP(\gb)$ depends only on the commutation equivalence class of $\bmi$.
\end{Lem}
Now we state the following theorem.

\begin{Thm}\label{Thm:affine_hw_cat_of_R}
 For any reduced word $\bmi$ of $w_0$ and $\gb \in Q_\sg^+$, the category $\modfg{\wh{R}(\gb)}$ is an affine highest weight category with respect to the poset $(KP(\gb),\preceq_{\bmi})$.
 For $\bfm \in KP(\gb)$, the associated standard module is $\gD_{\bmi}(\bfm)$, the proper standard module is $\bar{\gD}_{\bmi}(\bfm)$, and the proper costandard module is $\bar{\nabla}_{\!\bmi}(\bfm)$.
\end{Thm}

\begin{proof}
 It is proved in \cite[Theorem 6.11]{mcnamara2017representation} (or essentially in \cite{kato2014poincare} and \cite{brundan2014homological}) 
 that the category of finitely generated graded $R(\gb)$-modules
 is an affine highest weight category with respect to the order $\preceq_{\bmi}$.
 By applying \cite[Theorem 4.6]{fujita2018tilting}, these results also hold for $\modfg{\wh{R}(\gb)}$.
\end{proof}

\begin{Rem}\label{Rem:independent_of_word}\normalfont
 By Lemma \ref{Lem:McNamara_order},
 the affine highest weight structure of $\modfg{\wh{R}(\gb)}$ in the theorem, 
 and in particular the isomorphism classes of the modules $S_{\bmi}(\bfm)$, $\gD_{\bmi}(\bfm)$, $\bar{\gD}_{\bmi}(\bfm)$ 
 and $\bar{\nabla}_{\!\bmi}(\bfm)$,  depend only on the commutation equivalence class of $\bmi$.
\end{Rem}

\section{Generalized quantum affine Schur--Weyl duality}

\subsection{$Q$-data}\label{Subsection:Q-data}

In untwisted $ADE$ types, an important full subcategory $\cC_Q$ of $\mathcal{C}$ associated with a Dynkin quiver $Q$ was introduced in \cite{hernandez2015quantum},
and then generalized to the other types in \cite{kang2016symmetric,kashiwara2019categorical,oh2019categorical}.
Fujita and Oh \cite{fujita2020q} gave a unified framework for these categories using the notion of $Q$-data,
which we recall here.

In this subsection, \textit{we assume that $\fg$ is of untwisted affine type}.
Associated with the type of $\fg_0$, we define a pair ($\sg,\gs)$ of a simple Lie algebra $\sg$ of type $ADE$ and its Dynkin diagram automorphism $\gs$ as follows.
If $\fg_0$ is of type $ADE$, then set $\sg=\fg_0$ and $\gs=\id$. 
In other types, set $(\sg,\gs)$ as in Figure \ref{figure}.
We keep the notation for $\fg$ and $\sg$ in the previous sections. 
\begin{figure}[t]
\begin{align*}
 \raisebox{-20pt}{$\fg_0\colon B_n$}\ {\footnotesize \raisebox{-16pt}{
\xymatrix@C=10pt@R=8pt@M=0pt{
 \circ \ar@{-}[r] & \circ \ar@{.}[r] & \circ \ar@{-}[r] & \circ  \ar@{=}[r]|(.7)@{>} &\circ \\
 {\ \ \ 1 \ \ }& {\ \ \ 2 \ \ } & n-2 & n-1 & n}}
} \ \ \ \raisebox{-19pt}{$\Rightarrow$}\ \ 
 \raisebox{-18pt}{$\sg \colon A_{2n-1}$} \ {\footnotesize \raisebox{35pt}{
\xymatrix@C=10pt@R=20pt@M=0pt{
  \raisebox{-30pt}{$2n-1$} &  \raisebox{-30pt}{$2n-2$} &  \raisebox{-30pt}{$n+2$} &  \raisebox{-30pt}{$n+1$} \\
 \circ \ar@/_8pt/@{<-->}[d]_-\gs \ar@{-}[r]& \circ\ar@{.}[r] \ar@/_8pt/@{<-->}[d]& \circ\ar@{-}[r]\ar@/_8pt/@{<-->}[d]& \circ\ar@{-}[rd]\ar@/_8pt/@{<-->}[d] & \\
 \circ \ar@{-}[r]& \circ\ar@{.}[r] & \circ\ar@{-}[r]& \circ\ar@{-}[r]\ar@{-}[r] &\circ \\
 \raisebox{24pt}{ \ \ \ \ $1$\ \ \ \ } & \raisebox{24pt}{\ \ $2$\ \ } & \raisebox{24pt}{$n-2$}& \raisebox{24pt}{$n-1$} &\raisebox{24pt}{$n$} }}
}\end{align*}\ \\[-45pt]
\begin{align*}
 \raisebox{-20pt}{$\fg_0\colon C_n$} {\footnotesize \raisebox{-16pt}{
\xymatrix@C=10pt@R=8pt@M=0pt{
 \circ \ar@{-}[r] & \circ \ar@{.}[r] & \circ \ar@{-}[r] & \circ  \ar@{=}[r]|(.3)@{<} &\circ \\
 {\ \ \ 1 \ \ }& {\ \ \ 2 \ \ } & n-2 & n-1 & n}}
} \ \ \ \ \  \raisebox{-19pt}{$\Rightarrow$}\ \ 
\raisebox{-18pt}{$\sg \colon D_{n+1}$} {\footnotesize \raisebox{36pt}{
\xymatrix@C=10pt@R=20pt@M=0pt{
  &  &  &  & \raisebox{-30pt}{$n+1$} \\
 & & & & \circ \ar@/^8pt/@{<-->}[d]^-\gs\\
 \circ \ar@{-}[r]& \circ\ar@{.}[r] & \circ\ar@{-}[r]& \circ\ar@{-}[r]\ar@{-}[ru] &\circ \\
 \raisebox{24pt}{ \ \ \ \ $1$\ \ \ \ } & \raisebox{24pt}{\ \ $2$\ \ } & \raisebox{24pt}{$n-2$}& \raisebox{24pt}{$n-1$} &\raisebox{24pt}{$n$}} }
}\end{align*}\ \\[-50pt]
\begin{align*}
 \raisebox{-53pt}{$\fg_0\colon F_4$} \hspace{15pt}{\footnotesize \raisebox{-52pt}{
\xymatrix@C=26pt@R=8pt@M=0pt{
 \circ \ar@{-}[r] & \circ \ar@{=}[r]|(.7)@{>} & \circ \ar@{-}[r]  & \circ \\
  1  & 2 & 3 & 4}}
} \hspace{23pt}\raisebox{-52pt}{$\Rightarrow$}\ \
\raisebox{-51pt}{$\sg \colon E_6$} \hspace{15pt} {\footnotesize
\xymatrix@C=28pt@R=20pt@M=0pt{
 \raisebox{-27pt}{$5$} &  \raisebox{-27pt}{$4$}& &  &  \\
 \circ\ar@{-}[r]\ar@/_8pt/@{<-->}[d]_-\gs& \circ\ar@/_8pt/@{<-->}[d]\ar@{-}[rd]& &  \\
 \circ\ar@{-}[r] & \circ \ar@{-}[r] &\circ \ar@{-}[r] & \circ \\
 \raisebox{25pt}{$1$} & \raisebox{25pt}{$2$} & \raisebox{25pt}{$3$} & \raisebox{25pt}{$6$}}
}\end{align*}\ \\[-30pt]
\begin{align*}
 \raisebox{-23pt}{$\fg_0\colon G_2$}\hspace{20pt} {\footnotesize \raisebox{-22pt}{
\xymatrix@C=26pt@R=8pt@M=0pt{
 \circ \ar@3{-}[r]|(.7)@{>} & \circ  \\
  1  & 2}}
} \hspace{35pt} \raisebox{-22pt}{$\Rightarrow$}\ \ 
\raisebox{-21pt}{$\sg \colon D_4$} \ \ \ \ \ \  {\footnotesize
\xymatrix@C=28pt@R=20pt@M=0pt{
 \raisebox{-4pt}{$1$ $\circ$} \ar@<10pt>@/_20pt/@{<--}[dd]_-\gs  \ar@<-1pt>@{-}[rd]&\\
 3 \circ \ar@<0pt>@/_7pt/@{<--}[u] \ar@<0pt>@/^7pt/@{-->}[d] \hspace{-1pt}\ar@{-}[r]& \circ \ 2\\
 \raisebox{5pt}{$4$ $\circ$} \ar@{-}[ru] & }\hspace{30pt}
}\end{align*}
\caption{$\sg$ and $\gs$ for nonsimply-laced $\fg_0$}\label{figure}
\end{figure}

For $\imath,\jmath \in J$ (resp.~$i,j \in I_0$), we write $\imath \sim \jmath$ (resp.~$i \sim j$) if $a_{\imath\jmath}^{\sg} < 0$ (resp.~$a_{ij}<0$).
Let $J^\gs$ denote the set of $\gs$-orbits of $J$,
and define a matrix $(c_{ST})_{S,T \in J^\gs}$ by
\begin{equation}\label{eq:matrixB}
 c_{ST}=\begin{cases} 2 & \text{if $S =T $}, \\
                                - \max\{1, |T|/|S|\} & \text{if there are $\imath \in S$, $\jmath \in T$ such that $\imath \sim \jmath$},\\ 

                                0                                  & \text{otherwise}, \end{cases}
\end{equation}
where $|S|$ denotes the cardinality of $S$.
Then the matrix $(c_{ST})$ coincides with the Cartan matrix $(a_{ij})_{i,j \in I_0}$ of $\fg_0$.
If $\gs\neq \id$, we identify $J^\gs$ with $I_0$ via the unique bijection $I_0 \ni i \mapsto S_i\in J^\gs$ 
such that $c_{S_iS_j}= a_{ij}$ for $i,j \in I_0$.
Otherwise, we have $J^\gs=J=I_0$ naturally.
Let $J \ni \imath \mapsto \bar{\imath} \in I_0$ be the canonical map, and set
\[ J^{(i)} = \{ \imath \in J \mid \bar{\imath} = i\}  \ \ \text{ for } i \in I_0.
\]
We easily see that $|J^{(i)}|=\ss_i$.

\begin{Def}\normalfont
 A function $\xi\colon J \to \Z$ is called a \textit{height function} on $(J,\gs)$ if the following conditions are satisfied:
 \begin{itemize}
  \item[(i)] For $\imath,\jmath \in J$ such that $\imath \sim \jmath$ and $\ss_{\bar{\imath}}=\ss_{\bar{\jmath}}$, we have $|\xi_{\imath}-\xi_{\jmath}|=\ss_{\bar{\imath}}$.
  \item[(ii)] For $i,j \in I_0$ such that $i \sim j$ and $\ss_{i}<\ss_{j}$, there exists a unique element $\jmath^{\circ} \in J^{(j)}$ such that
   $|\xi_{\imath}-\xi_{\jmath^\circ}| =1$ and $\xi_{\gs^k(\jmath^\circ)}=\xi_{\jmath^\circ}-2k$ for $0\leq k < |\gs|$, where $\imath \in J^{(i)}$ is the unique element
   and $|\gs|$ is the order of $\gs$.
 \end{itemize}
 Such a triple $\mathbf{Q}=(J, \gs, \xi)$ is called a \textit{$Q$-datum}.
\end{Def}

The following assertions are easily seen from the definition, or by a case-by-case check.
\begin{Lem}[{\cite[Lemma 2.9]{fujita2020q}}] \label{Lem:fundamental_Q_data}
 Let $\bQ = (J,\gs,\xi)$ be a $Q$-datum and $i \in I_0$.\\
 {\normalfont(i)} For any $\imath,\imath' \in J^{(i)}$, we have $\xi_{\imath}\equiv \xi_{\imath'}$ $(\mathrm{mod} \ 2)$.\\
 {\normalfont(ii)} For any $\imath \in J^{(i)}$ and $l \in \Z$, we have $\xi_{\gs^l(\imath)} \equiv \xi_{\imath} -2l \ (\mathrm{mod} \ 2\ss_i)$.\\
 {\normalfont(iii)} Assume that $j \in I_0$ is such that $i\sim j$ and $\ss_i = \ss_j$, and $\imath \in J^{(i)}$.
   Then $\jmath \in J^{(j)}$ satisfies $\xi_\jmath \equiv \xi_\imath+\ss_{i}$ $(\mathrm{mod} \ 2\ss_i)$ if and only if $\imath \sim \jmath$.
\end{Lem}

Fix a $Q$-datum $\bQ=(J,\gs,\xi)$. 
We define a subset $\hat{J} =\hat{J}_\bQ \subseteq J \times \Z$ by
\[ \hat{J} = \{ (\imath, p) \mid p \equiv \xi_\imath \ \text{(mod $2\ss_{\bar{\imath}}$)}\}.
\]
We also define $\hat{I}_0=\hat{I}_{0,\bQ} \subseteq I_0 \times \Z$ by $\hat{I}_{0}=\{ (\bar{\imath},p)\mid (\imath,p) \in \hat{J}\}$.
The canonical map $\hat{J} \to \hat{I}_0$ is bijective by Lemma \ref{Lem:fundamental_Q_data}, and we have
\[ \hat{I}_0 = \{(i,p) \in I_0 \times \Z \mid p \equiv \gee_i \ (\text{mod $2$})\}, 
\]
where $\gee\colon I_0 \to \{0,1\}$ is the unique function satisfying $\gee_{\bar{\imath}} \equiv \xi_{\imath}$ (mod $2$) for all $\imath \in J$.

An element $\imath \in J$ is called a \textit{source} of $\mathbf{Q}$ if we have $\xi_{\imath} > \xi_{\jmath}$ for all $\jmath \in J$ with $\imath \sim \jmath$.
Given a source $\imath \in J$ of $\mathbf{Q}$, a new height function $s_{\imath}\xi\colon J \to \Z$ is defined by
\[ (s_{\imath}\xi)_{\jmath} = \xi_\jmath- 2\ss_{\bar{\imath}}\gd_{\imath\jmath} \ \ \ \text{for } \jmath \in J,
\]
and we write $s_{\imath}\bQ = (J,\gs,s_{\imath}\xi)$.
We say a reduced word $(\imath_1,\ldots, \imath_L)$ of $w_0$ is \textit{adapted to} $\bQ$ 
if $\imath_k$ is a source of $s_{\imath_{k-1}}\cdots s_{\imath_1}\bQ$ for $1\leq k \leq L$.
\begin{Lem}[{\cite[Theorem 2.24]{fujita2020q}}]
 The set of all the reduced words of $w_0$ adapted to $\bQ$ forms a single commutation equivalence class.
\end{Lem}
Let $\bmi=(\imath_1,\ldots, \imath_L)$ be a reduced word of $w_0$ adapted to $\bQ$,
%
%
%
%
and set 
\[ \gb_k =s_{\imath_1}\cdots s_{\imath_{k-1}}(\ga_{\imath_k}^\sg)\in R_\sg^+ \ \text{ for $1\leq k \leq L$}.
\]
Define a map $\wti{\Omega}_{\bQ}\colon R_\sg^+ \to \hat{J}$ by
\[ \wti{\Omega}_{\bQ}(\gb_k) = \big(\imath_k, (s_{\imath_{k-1}}\cdots s_{\imath_1}\xi)_{\imath_{k}}\big).
\]

\begin{Lem}[{\cite[Theorem 2.24]{fujita2020q}}]\label{Lem:image_of_Omega}\ \\
 {\normalfont(i)} The map $\wti{\Omega}_{\mathbf{Q}}$ does not depend on the choice of $\bmi$.\\
 {\normalfont(ii)} The image $\mathrm{Im}\, \wti{\Omega}_{\bQ} \subseteq \hat{J}$ is described as follows:
 \[ \mathrm{Im}\,\wti{\Omega}_{\bQ} = \{ (\imath,p) \mid \xi_{\imath^*}-|\gs|h^\vee < p \leq \xi_{\imath}, \ p \equiv \xi_{\imath} \ (\mathrm{mod} \ 2\ss_{\bar{\imath}})\},
 \]
 where $\imath^* \in J$ is the index determined by $w_0(\ga_{\imath}^\sg)=-\ga_{\imath^*}^\sg$, and $h^\vee$ is the dual Coxeter number of $\fg_0$.
\end{Lem}


Let $\tau_{\bQ} \in W_\sg\gs$ be the \textit{generalized $\gs$-Coxeter element} associated with $\bQ$ introduced in \cite[Definition 2.33]{fujita2020q}.
For $\imath \in J$, set
\[ \ggg_{\imath}^{\bQ} = (1-\tau_{\bQ}^{\ss_{\bar{\imath}}})\gL_{\imath}^\sg,
\]
which belongs to $R_\sg^+$.
The properties of $\tau_{\bQ}$ needed in this paper are the following.

\begin{Prop}\label{Prop:Coxeter_element}
 Let $1\leq k \leq L$, and assume that $\wti{\Omega}_{\bQ}(\gb_k) = (\imath_k,p)$.\\
 {\normalfont(i)} We have $\gb_k = \tau_{\bQ}^{(\xi_{\imath_k}-p)/2}(\ggg_{\imath_k}^{\bQ})$.\\
 {\normalfont(ii)} We have $s_{\imath_1}\cdots s_{\imath_k} (\gL_{\imath_k}^\sg) = \tau_{\bQ}^{(\xi_{\imath_k}-p+2\ss_{\bar{\imath}_k})/2}(\gL_{\imath_k}^\sg)$.
\end{Prop}

\begin{proof}
 The assertion (i) follows from \cite[Theorem 2.35]{fujita2020q}, and the assertion (ii) follows from \cite[Lemma 8.2 (3)]{FHOO21}.
\end{proof}

%
%

Let $z$ be an indeterminate. 
The \textit{quantum Cartan matrix} of $\fg_0$ is the $\Z[z^{\pm 1}]$-valued $(I_0 \times I_0)$-matrix $A(z)=\big(a_{ij}(z)\big)_{i,j\in I_0}$ defined by
\[ a_{ij}(z) = \begin{cases} z^{\ss_i} + z^{-\ss_i} & \text{if $i = j$},\\
                             [a_{ij}]_z & \text{if $i\neq j$},\\ \end{cases}
\]
where $[k]_z = (z^k-z^{-k})/(z-z^{-1})$ for $k \in \Z$.
%
%
Denote by $\wti{A}(z) = \big(\wti{A}_{ij}(z)\big)_{i,j\in I_0} \in \mathrm{GL}_{I_0}\big(\Q(z)\big)$ the inverse of $A(z)$, and let
\[ \wti{A}_{ij}(z) = \sum_{u \in \Z} \wti{a}_{ij}(u) z^u
\]
denote the formal Laurent expansion of the $(i,j)$-entry $\wti{A}_{ij}(z)$ at $z=0$.

\begin{Prop}\label{Prop:inverse_Cartan}
 Let $i,j \in I_0$, and take $\imath \in J^{(i)}$ and $\jmath \in J^{(j)}$ arbitrarily.\\
 {\normalfont (i)} For any $u \in \Z$, we have
 \[ \wti{a}_{ij}(u) -\wti{a}_{ij}(-u) = \begin{cases} \big(\gL^\sg_{\imath}, \tau_{\bQ}^{(u+\xi_{\jmath}-\xi_{\imath}-\ss_i)/2}(\ggg_{\jmath}^{\bQ})\big) & 
                                                      \text{if $u+\xi_\jmath -\xi_\imath-\ss_i \in 2\Z$},\\
                                    0 & \text{otherwise}.
                      \end{cases}
 \]
 {\normalfont(ii)} Set $\ss_{ij} = \min(\ss_{i},\ss_j)$. Then we have
  \[ \wti{a}_{ij}(\xi_{\jmath}-\xi_{\imath}-\ss_i - 2l\ss_{ij}) = 0 = \wti{a}_{ij}(\xi_{\imath}-\xi_{\jmath}-\ss_i-2l\ss_{ij}) \ \ \text{ for any $l \in \Z_{\geq 0}$}.
  \] 
\end{Prop}

\begin{proof}
 The assertion (i) follows from \cite[Theorem 4.13]{FHOO21}, and (ii) follows from [\textit{loc.~cit.}, Lemma 8.3].
\end{proof}

\subsection{Full subcategories of $\cC$}\label{subsection:Full_subcategories}

From now on, we again assume that $\fg$ is of general (untwisted or twisted) affine type.
When $\fg$ is of untwisted type, set $(\sg,\gs)$ to be as in the previous subsection.
When $\fg$ is of twisted type $X_N^{(t)}$ ($t \in \{2,3\}$), on the other hand, set $\sg$ to be the simple Lie algebra of type $X_N$ and $\gs = \id$. 
Hence we do not identify $J^\gs(=J)$ with $I_0$ in twisted types, although we do in untwisted types as in the previous subsection.

In the following, when $\fg$ is untwisted (resp.~twisted) 
we only consider $\ell$-weights of the forms $\bm{\varpi}_{i,q^p}$ (resp.~ $\mathbf{t}(\bm{\varpi}_{\imath,q^p}^\sg)$) with $i \in I_0$ (resp.~$\imath \in J$) and $p\in \Z$,
where $\mathbf{t}\colon \mathcal{P}_\sg \to \mathcal{P}$ is the map given in (\ref{eq:map_t}).
Hence by abuse of notation, we will use the simplified notation $\bm{\varpi}_{i,p}$ (resp.~$\mathbf{t}(\bm{\varpi}_{\imath,p}^\sg)$) and $\bm{\ga}_{i,p}$ 
(resp.~$\mathbf{t}(\bm{\ga}_{\imath,p}^\sg)$), instead of $\bm{\varpi}_{i,q^p}$ (resp.~$\mathbf{t}(\bm{\varpi}_{\imath,q^p}^\sg)$) and $\bm{\ga}_{i,q^p}$ 
(resp.~$\mathbf{t}(\bm{\ga}_{\imath,q^p}^\sg)$).

Fix a $Q$-datum $\bQ=(J,\gs,\xi)$.
We define the following sublattices of $\cP$ and $\cQ$ respectively:
\[ \cP_\Z = \begin{cases} \displaystyle \bigoplus_{(i,p) \in \hat{I}_0} \Z\bm{\varpi}_{i,p} & \text{$\fg\colon$untwisted},\\
                          \displaystyle \bigoplus_{(\imath,p) \in \hat{J}} \Z\,\mathbf{t}(\bm{\varpi}^\sg_{\imath,p}) & \text{$\fg\colon$twisted},\end{cases} \ \ \ 
   \cQ_\Z = \begin{cases} \displaystyle \bigoplus_{(i,p+\ss_i) \in \hat{I}_0} \Z\bm{\ga}_{i,p} & \text{$\fg\colon$untwisted},\\
                          \displaystyle \bigoplus_{(\imath,p+1) \in \hat{J}} \Z\,\mathbf{t}(\bm{\ga}^\sg_{\imath,p}) 
                          & \text{$\fg\colon$twisted}.\end{cases}\]
It is easily seen from Lemma \ref{Lem:fundamental_Q_data} that $\cQ_\Z =\cP_\Z \cap \cQ$.
Set $\cP_\Z^+ = \cP_\Z \cap \cP^+$ and $\cQ_\Z^+ = \cQ_\Z \cap \cQ^+$.
We define a $\Z$-linear map $\Omega_\bQ$ from the set $\Z^{R^+_\sg}$ of $\Z$-valued functions on $R_\sg^+$ to $\mathcal{P}_\Z$ as follows:
for $\ga \in R_\sg^+$ such that 
$\tilde{\Omega}_{\bQ}(\ga) = (\imath,p) \in \hat{J}$, set
\[ \Omega_\bQ(\bfm_\ga) = \begin{cases} \bm{\varpi}_{\bar{\imath},p} &\text{if $\fg$ is untwisted}, \\
                                   \mathbf{t}(\bm{\varpi}^{\sg}_{\imath,p}) & \text{if $\fg$ is twisted},\end{cases}
\]
where $\bfm_\ga \in \Z^{R^+_\sg}$ is defined as in Subsection \ref{Subsection:Affine_highest_weight_category}.
Set
\[ \cP_{\bQ} = \mathrm{Im}\, \Omega_{\bQ}\subseteq \cP_\Z \ \text{ and } \ \cP_{\bQ}^+ = \Omega_{\bQ} (\Z_{\geq 0}^{R_\sg^+}) \subseteq \cP_\Z^+.
\]

\begin{Def}\normalfont
 Let $\mathcal{C}_{\bQ}$ be the full subcategory of $\mathcal{C}$ consisting of objects $M$ such that any simple subquotient of $M$ is isomorphic to
 $L(\bm{\pi})$ for some $\bm{\pi} \in \mathcal{P}^+_{\bQ}$.
\end{Def}

\begin{Lem}\label{Lem:coincidence_of_lattices}
 Assume that $\fg$ is untwisted, and define a subset $\hat{K} \subseteq \hat{I}_0$ by
 \[ \hat{K} = \{ (\bar{\imath},p) \mid \imath \in J, \ (\imath,p-\ss_{\bar{\imath}}) \in \mathrm{Im}\, \wti{\Omega}_{\bQ} \text{ and } 
   (\imath,p+\ss_{\bar{\imath}}) \in \mathrm{Im}\, \wti{\Omega}_{\bQ}\}.
 \]
 Then we have
 \begin{equation}\label{eq:coincidence_of_lattices}
  \cP_{\bQ} \cap \cQ = \bigoplus_{(i,p) \in \hat{K}} \Z \bm{\ga}_{i,p}.
 \end{equation}
\end{Lem}

\begin{proof}
 Let $(i,p) \in \hat{K}$. 
 We can check from the definition of height functions and Lemma \ref{Lem:image_of_Omega} that for any $j \in I_0$ such that $i \sim j$, 
 $\bm{\varpi}_{j,p} \in \cP_{\bQ}$ (resp.~$\bm{\varpi}_{j,p\pm 1} \in \cP_{\bQ}$, $\bm{\varpi}_{j,p\pm 2}, \bm{\varpi}_{j,p} \in \cP_{\bQ}$) if $\ss_i\leq \ss_j$ (resp.~$2=\ss_i>\ss_j$, 
 $3=\ss_i>\ss_j$),
 which implies $\bm{\ga}_{i,p} \in \cP_{\bQ}$.
 Hence the containment $\supseteq$ is proved.
 Let us prove the opposite containment.
 Since $\cQ_\Z=\cP_\Z \cap\cQ$, the left-hand side of (\ref{eq:coincidence_of_lattices}) coincides with $\cP_{\bQ} \cap \cQ_\Z$.
 Assume that 
 \begin{equation}\label{eq:containment_in_Q}
  \sum_{(\imath,p+\ss_{\bar{\imath}}) \in \hat{J}} l_{\imath,p} \bm{\ga}_{\bar{\imath},p} \in \cP_{\bQ}\cap \cQ_\Z
 \end{equation}
 with $l_{\imath,p} \in \Z$ (recall that the canonical map $\hat{J} \to \hat{I}_0$ is bijective).
 It suffices to show that $(\bar{\imath},p) \in \hat{K}$ for all $(\imath,p)$ such that $l_{\imath,p}\neq 0$.
 Set $J' = \{\imath \in J \mid l_{\imath,p} \neq 0 \text{ for some } p \}$.
 For each $\imath \in J'$, set $p(\imath)=\max\{ p\mid l_{\imath,p}\neq 0\}$,
 and let $\imath_0 \in J'$ be an element such that 
 \[ p(\imath_0) +\ss_{\bar{\imath}_0}= \max\{p(\imath) +\ss_{\bar{\imath}} \mid \imath \in J'\}.
 \]
 We will prove $p(\imath) +\ss_{\bar{\imath}} \leq \xi_\imath$ for all $\imath \in J'$ by downward induction on $p(\imath)+\ss_{\bar{\imath}}$.
 Write the sum in (\ref{eq:containment_in_Q}) as $\sum_{(\imath,p) \in \hat{J}} m_{\imath,p} \bm{\varpi}_{\bar{\imath},p}$.
 Since $m_{\imath_0,p(\imath_0)+\ss_{\bar{\imath}_0}} =l_{\imath_0,p(\imath_0)} \neq 0$,
 $p(\imath_0)+\ss_{\bar{\imath}_0}\leq \xi_{\imath_0}$ follows from (\ref{eq:containment_in_Q}) and Lemma \ref{Lem:image_of_Omega}. 
 Assume that $\imath \in J'\setminus \{\imath_0\}$ satisfies $p(\imath)+\ss_{\bar{\imath}} > \xi_{\imath}$.
 Since (\ref{eq:containment_in_Q}) forces $m_{\imath,p(\imath)+\ss_{\bar{\imath}}} =0$, there exists $\jmath \in J'$ satisfying $\bar{\imath} \sim \bar{\jmath}$ 
 and 
 \[ l_{\jmath,p(\imath)+\ss_{\bar{\imath}}-a_{\bar{\imath}\bar{\jmath}} +1-2k}\neq 0 \ \ \text{ for some $1\leq k \leq -a_{\bar{\imath}\bar{\jmath}}$}.
 \]
 Then by the induction hypothesis, we have
 \[ \xi_{\jmath} \geq p(\jmath)+\ss_{\bar{\jmath}} \geq p(\imath)+\ss_{\bar{\imath}}+\ss_{\bar{\jmath}}+a_{\bar{\imath}\bar{\jmath}}+1 
    \geq \xi_{\imath} + 2\ss_{\bar{\imath}} +\ss_{\bar{\jmath}} + a_{\bar{\imath}\bar{\jmath}}+1,
 \]
 where the third inequality follows from $p(\imath)+\ss_{\bar{\imath}} \in \hat{J}$ and $p(\imath)+\ss_{\bar{\imath}} > \xi_{\imath}$.
 We easily see that this contradicts the defining conditions of $\xi$, and hence $p(\imath) +\ss_{\bar{\imath}} \leq \xi_{\imath}$ is proved for all $\imath \in J'$.
 By a similar argument we can also show that
 \[ \min\{p \mid l_{\imath,p} \neq 0\} -\ss_{\bar{\imath}} > \xi_{\imath^*} -|\gs|h^\vee \ \ \text{ for all $\imath \in J'$}.
 \] 
 Then $\{(\bar{\imath},p)\mid l_{\imath,p} \neq 0 \}\subseteq \hat{K}$ follows from Lemma \ref{Lem:image_of_Omega}. 
 The proof is complete.
\end{proof}


Fix a reduced word $\bmi=(\imath_1,\ldots,\imath_L)$ of $w_0$ adapted to $\bQ$, and set $\gb_k = s_{\imath_1}\cdots s_{\imath_{k-1}}(\ga_{\imath_k}^\sg)$ for $1\leq k \leq L$.

\begin{Lem}\label{Lem:compatible_with_words}
 Set $L_k = L\big(\Omega_{\bQ}(\bfm_{\gb_k})\big) \in \cC_{\bQ}$ for $1\leq k \leq L$.
 Then for any $1\leq k < l \leq L$, we have $d_{L_k,L_l}(1) \neq 0$.
\end{Lem}

\begin{proof}
 See \cite[Lemma 5.2]{kang2016symmetric}, \cite[Lemma 6.4]{kashiwara2019categorical}, and \cite{oh2019categorical}.
\end{proof}

For $\gb \in Q_\sg$, set $\widetilde{KP}(\gb) = \{\bfm \in \Z^{R^+_\sg}\mid \sum_{\ga \in R^+_\sg} \bfm(\ga)\ga=\gb\}$.
Note that $\widetilde{KP}(\gb) \cap \Z_{\geq 0}^{R^+_\sg} = KP(\gb)$ for $\gb \in Q^+_\sg$.
As in Subsection \ref{Subsection:Affine_highest_weight_category}, we identify $\bfm \in \widetilde{KP}(\gb)$ with $(\bfm_1,\ldots,\bfm_L) \in \Z^L$ via $\bfm_k=\bfm(\gb_k)$.
Define $\nu_k \in P_\sg$ ($1\leq k \leq L$) as in (\ref{eq:def_of_ggg}) using the given reduced word $\bmi$.

\begin{Lem}\label{Lem:Lem_for_comparing_orders}
 Assume that $\fg$ is untwisted, and $(\imath,p) \in \hat{J}$ satisfies $(\imath, p - \ss_{\bar{\imath}}) \in \mathrm{Im}\, \wti{\Omega}_{\bQ}$ and 
  $(\imath,p+\ss_{\bar{\imath}}) \in \mathrm{Im}\,\wti{\Omega}_{\bQ}$. Set $\bfm = \Omega_{\bQ}^{-1}(\bm{\ga}_{\bar{\imath},p}) \in \Z^{R^+_\sg}$.\\
 {\normalfont(i)} We have $\bfm  \in \widetilde{KP}(0)$.\\
 {\normalfont(ii)} Let $1\leq k_0 \leq L$ be such that $\wti{\Omega}_{\bQ}(\gb_{k_0}) = (\imath,p+\ss_{\bar{\imath}})$. Then we have 
  \begin{equation}\label{eq:order_of_alpha}
   \sum_{t=1}^k \bfm_t (\nu_k, \gb_t) = \gd_{k,k_0} \ \ \text{ for } 1\leq k \leq L.
  \end{equation}
\end{Lem}

\begin{proof}
 The assertion (i) follows from \cite[Eq.~(4.3)]{fujita2020q}.
 Let us prove the assertion (ii).
 For each $1\leq t \leq L$, write $\wti{\Omega}_{\bQ}(\gb_t) = (\imath_t,p_t) \in \hat{J}$.
 Then for each $1\leq k,t\leq L$, it follows from Propositions \ref{Prop:Coxeter_element} and \ref{Prop:inverse_Cartan} (i) that
 \begin{align}\label{eq:nugb}
  (\nu_k,\gb_t) &=  -\big(\tau_{\bQ}^{(\xi_{\imath_k}-p_k+2\ss_{\bar{\imath}_k})/2}(\gL_{\imath_k}^\sg),\tau_{\bQ}^{(\xi_{\imath_t}-p_t)/2}(\ggg_{\imath_t}^{\bQ})\big)\nonumber\\
  & = -\big(\gL_{\imath_k}^\sg,\tau_{\bQ}^{(\xi_{\imath_t}-\xi_{\imath_k}+p_k-p_t-2\ss_{\bar{\imath}_k})/2}(\ggg_{\imath_t}^{\bQ})\big)\nonumber\\
  &  = \wti{a}_{\bar{\imath}_k\bar{\imath}_t}(p_t-p_k+\ss_{\bar{\imath}_k}) - \wti{a}_{\bar{\imath}_k\bar{\imath}_t}(p_k-p_t-\ss_{\bar{\imath}_k}).
 \end{align}
 If $t\leq k$ (resp.~$t>k$), by applying Proposition \ref{Prop:inverse_Cartan} (ii) to the height function $s_{\imath_{t-1}}\cdots s_{\imath_1}\xi$ (resp.~$s_{\imath_k}\cdots s_{\imath_1}\xi$), 
 it follows that 
 \begin{equation}\label{eq;inverse_Cartan}
  \wti{a}_{\bar{\imath}_k\bar{\imath}_t}(p_k-p_t-\ss_{\bar{\imath}_k})=0 \ \ \ \ \text{ (resp.}~\wti{a}_{\bar{\imath}_k\bar{\imath}_t}(p_t-p_k+\ss_{\bar{\imath}_k})=0).
 \end{equation}
 Fix $1\leq k \leq L$ arbitrarily, and set $\jmath = \imath_k$ and $j=\bar{\jmath}$. 
 Setting $i=\bar{\imath}$, it follows from (\ref{eq:nugb}), (\ref{eq;inverse_Cartan}) and the definition of $\ell$-roots in (\ref{eq:l_root}) that
 \begin{align*}
  &\sum_{t=1}^k \bfm_t (\nu_k, \gb_t) = \sum_{t=1}^k \bfm_t \wti{a}_{j\bar{\imath}_t}(p_t-p_k+\ss_{j})
   =\sum_{t=1}^L \bfm_t \wti{a}_{j\bar{\imath}_t}(p_t-p_k+\ss_{j})\\
  &= \wti{a}_{ji}(p-\ss_{i}-p_k+\ss_j)+\wti{a}_{j i}(p+\ss_i-p_k+\ss_j)- \sum_{i'\neq i}\!
  \sum_{k=1}^{-a_{i'i}}\wti{a}_{ji'}(p-a_{i'i}+1-2k-p_k+\ss_j)\\
  &=\Big[ \wti{A}(z)A(z)\Big]_{(j,i),z^{p-p_k+\ss_{j}}} = \gd_{ij}\gd_{p_k,p+\ss_j},
 \end{align*}
 where $\Big[ C(z)\Big]_{(j,i),z^r}$ denotes the coefficient of $z^r$ of the $(j,i)$-entry of $C(z)$.
 This implies (\ref{eq:order_of_alpha}), and the assertion is proved.
\end{proof}

For each $\gb \in Q_\sg^+$, we define the finite subset
\[ \cP_{\bQ,\gb}^+ = \Omega_\bQ\big(KP(\gb)\big) \subseteq \cP_{\bQ}^+.
\]
Obviously we have $\cP_{\bQ}^+ = \bigsqcup_{\gb \in Q^+_\sg} \cP_{\bQ,\gb}^+$.

\begin{Prop}\label{Prop:isom_of_posets}
 For any $\gb \in R^+_\sg$, the map $\Omega_\bQ$ induces an isomorphism of posets from $(KP(\gb), \preceq_{\bmi})$ to $(\cP_{\bQ,\gb}^+, \leq)$.
\end{Prop}

\begin{proof}
 By definition, the assertion in twisted types follows from the one for $\wh{\sg}$, and hence we may assume that $\fg$ is untwisted.
 The bijectivity is obvious. 
 It remains to show for $\bfm, \bfn \in KP(\gb)$ that $\Omega_\bQ(\bfm) \leq \Omega_\bQ(\bfn)$ if and only if $\bfm \preceq \bfn$.
 Set 
 \[ \mathcal{L} = \{ 1\leq k \leq L \mid \text{there is no $k'$ such that $k'>k$ and $\imath_{k'}=\imath_k$}\}.
 \] 
 If $k \in \mathcal{L}$, it follows from the definition of $\nu_k$ that $( \nu_k, \gb_l) = 0$ for $k<l\leq L$.
 Hence it follows for any $\bfm \in KP(\gb)$ that
 \[ \sum_{t=1}^k \bfm_t(\nu_k,\gb_t) = \sum_{t=1}^L \bfm_t( \nu_k,\gb_t) = ( \nu_k,\gb) \ \text{ for } k \in \mathcal{L}.
 \]
 Therefore for $\bfm,\bfn \in KP(\gb)$ such that $\bfm \preceq \bfn$, if we set
 \[ c_k = \sum_{t=1}^k (\bfn_t-\bfm_t) ( \nu_k,\gb_t) \in \Z_{\geq 0} \ \text{ for $k \notin \mathcal{L}$},
 \] 
 then we have  
 \[ \bfn = \bfm + \sum_{k \notin \mathcal{L}} c_k\Omega_{\bQ}^{-1}(\bm{\ga}_{\bar{\imath}_k,p_k})
 \] 
 by Lemmas \ref{Lem:Lem_for_comparing_orders} and \ref{Lem:McNamara_order} (ii), 
 where we set $\wti{\Omega}_{\bQ}(\gb_k)=(\imath_k,p_k+\ss_{\bar{\imath}_k})$ (note that $k \notin \mathcal{L}$ implies that $(\imath_k,p_k-\ss_{\bar{\imath}_k})$ also belongs to 
 $\mathrm{Im}\, \wti{\Omega}_{\bQ}$). 
 Hence $\Omega_{\bQ}(\bfm) \leq \Omega_{\bQ}(\bfn)$ is proved. 
 On the other hand, the `only if' part is proved from Lemmas \ref{Lem:coincidence_of_lattices} and \ref{Lem:Lem_for_comparing_orders}.
 The proof is complete.
\end{proof}

For $\gb \in Q^+_\sg$, let $\mathcal{C}_{\bQ,\gb}$ be the full subcategory of $\cC_{\bQ}$ consisting of modules whose simple subquotients are isomorphic to 
$L(\bm{\pi})$ for some $\bm{\pi} \in \cP^+_{\bQ,\gb}$.

\begin{Lem}\label{Lem:direct_sum_of_categoryC}
 The category $\cC_{\bQ}$ has a direct sum decomposition 
 \[ \cC_{\bQ} \cong \bigoplus_{\gb \in Q_\sg^+} \cC_{\bQ,\gb}.
 \]
 Moreover, we have $\cC_{\bQ,\gb_1} \otimes \cC_{\bQ,\gb_2}\subseteq \cC_{\bQ,\gb_1+\gb_2}$ for $\gb_1,\gb_2 \in Q^+_\sg$,
 and hence $\cC_{\bQ}$ is stable under taking tensor products.
\end{Lem}

\begin{proof}
 For the proof of the former assertion, by Lemma \ref{Thm:Chari_Moura}  
 it is enough to show for any $\gb \in Q^+_\sg$ that if $\bm{\pi} \in \cP_{\bQ,\gb}^+$ and $\bm{\gs} \in (\bm{\pi}+\mathcal{Q})\cap \cP^+_{\bQ}$, then 
 $\bm{\gs} \in \cP^+_{\bQ,\gb}$. 
 By definition, this assertion in twisted types follows from the one for $\hat{\sg}$, and therefore we may assume that $\fg$ is untwisted. 
 Since $\bm{\pi}-\bm{\gs} \in \cP_{\bQ} \cap \mathcal{Q}$, $\bm{\gs} \in \cP^+_{\bQ,\gb}$ follows from Lemmas \ref{Lem:coincidence_of_lattices} and \ref{Lem:Lem_for_comparing_orders} (i),
 and hence the former assertion is proved.
 The latter assertion follows from (\ref{eq:tensor_lweight}).
\end{proof}


\subsection{Generalized quantum affine Schur--Weyl duality functor}\label{subsection:GQASWD}

In the rest of this paper, we fix a $Q$-datum $\bQ=(J,\gs,\xi)$ and a reduced word $\bmi = (\imath_1,\ldots,\imath_L)$ of $w_0$ adapted to $\bQ$.
For $\bfm \in \Z_{\geq 0}^{R^+_{\sg}}$, we will write $\bm{\pi}_{\bfm} \in \cP^+_{\bQ}$ for $\Omega_{\bQ}(\bfm)$.
When $\bfm=\bfm_\ga$ for $\ga \in R^+_{\sg}$, we will also write $\bm{\varpi}_{\ga}$ for $\Omega_{\bQ}(\bfm_\ga)$.
We set $\gb_k = s_{\imath_1}\cdots s_{\imath_{k-1}} (\ga_{\imath_k}^{\sg})$ for $1\leq k \leq L$.

First we recall the following proposition.

\begin{Prop}\label{Prop:quiver_for_R}
 Set $L_\imath = L\big(\bm{\varpi}_{\ga_{\imath}^\sg}\big)$ for $\imath \in J$,
 and let $d_{\imath\jmath} \in \Z_{\geq 0}$ $(\imath,\jmath \in J)$ denote the order of the zero of the denominator $d_{L_{\imath},L_{\jmath}}(z)$ at $z=1$. 
 Let $\mathscr{S}$ be a quiver whose vertex set is $J$ such that for $\imath,\jmath \in J$, 
 there are $d_{\imath\jmath}$-many arrows from $\imath$ to $\jmath$.
 Then the underlying graph of $\mathscr{S}$ coincides with the Dynkin diagram of $\sg$, and for any $\imath,\jmath \in J$ such that $\imath \sim \jmath$ we have 
 $\imath \to \jmath$ in $\mathscr{S}$ if $\xi_{\imath} < \xi_{\jmath}$. 
\end{Prop}

\begin{proof}
 In untwisted $ADE$ types, the assertion was proved in \cite{kang2015symmetric} under some conjecture, 
 which was proved later in \cite{fujita2018tilting,oh2019categorical}.
 In twisted types, this was proved in \cite{kang2016symmetric,oh2019categorical}. 
 In the remaining types, this was proved in \cite{kashiwara2019categorical,oh2019categorical}.
\end{proof}

For $\gb \in Q_\sg^+$, let $R(\gb)$ be the symmetric quiver Hecke algebra at $\gb$ associated with $\sg$ and polynomials $(Q_{\imath\jmath})_{\imath,\jmath \in J}$ defined by
\[ Q_{\imath \jmath}(u,v)=\begin{cases} (u-v)^{d_{\imath \jmath}}(v-u)^{d_{\jmath\imath}} & \text{if } \imath \neq \jmath, \\
                             0 & \text{if } \imath=\jmath,\end{cases}
\]
where $d_{\imath\jmath}$ are as in the above proposition.
Let $\wh{R}(\gb)$ be its completion.

Write $L_\imath = L(\bm{\varpi}_{\ga_\imath^\sg})$ for $\imath \in J$.
We write $\wh{W}_\imath$ for the completion of the affinization $(L_\imath)_{\aff}$ with indeterminate $w_\imath$, namely,
\[ \wh{W}_\imath = (L_\imath)_{\aff}\otimes_{\bk[z_{L_\imath}^{\pm 1}]} \bk\llbracket w_\imath\rrbracket,
\]
where $z_{L_\imath}$ acts on $\bk\llbracket w_i\rrbracket$ via $z_{L_\imath}=1 + w_\imath$.
We have $\wh{W}_\imath \cong \wh{W}(\bm{\varpi}_{\ga_\imath^\sg})$ as mentioned in (\ref{eq:W_and_affinization}).
For $\gb \in Q^+_\sg$ with $|\gb|=m$,
we set
\[ \wh{W}_{\bmj} = \wh{W}_{\jmath_1} \hat{\otimes} \cdots \hat{\otimes} \wh{W}_{\jmath_m} \ \text{ for $\bmj=(\jmath_1,\ldots,\jmath_m) \in J^\gb$},\ \ 
   \text{ and } \ \  \wh{W}^{\otimes \gb} = \bigoplus_{\bmj \in J^\gb} \wh{W}_{\bmj}.
\]
Let $e(\bmj)$ denote the projection $\wh{W}^{\otimes \gb} \to \wh{W}_{\bmj}$ for $\bmj \in J^\gb$,
and regard $\wh{W}^{\otimes \gb}$ as a right module over $\wh{\bP}_\gb=\bigoplus_{\bmj \in J^\gb}e(\bmj)\bk\llbracket x_1,\ldots,x_m\rrbracket \subseteq \wh{R}(\gb)$,
where $\bk\llbracket x_1,\ldots,x_m\rrbracket$ acts on $\wh{W}_{\bmj}$ via the natural isomorphism $\bk\llbracket x_1,\ldots,x_m\rrbracket \stackrel{\sim}{\to} \bk\llbracket
w_{\jmath_1},\ldots,w_{\jmath_m} \rrbracket$.
By extending this to a right $\wh{R}(\gb)$-action by defining suitable actions of $\tau_l$ using normalized $R$-matrices,
we can define a $(U_q',\wh{R}(\gb))$-bimodule structure on $\wh{W}^{\otimes \gb}$ (see \cite{kang2018symmetric}).
Then a functor 
\[ \cF_\gb\colon \modfg{\wh{R}(\gb)} \to U_q'\mathrm{\mathchar`-mod}, \ \ \ \ M \mapsto \wh{W}^{\otimes \gb} \otimes_{\wh{R}(\gb)} M
\]
is defined, where $U_q'\mathrm{\mathchar`-mod}$ denotes the category of $U_q'$-modules.
Recall that the center of $\wh{R}(\gb)$ is identified with $\wh{\bS}_\gb = \bigotimes_{\imath=1}^N\bk\llbracket w_{\imath, 1},\ldots,w_{\imath, m_{\imath}}\rrbracket^{\mathfrak{S}_{m_{\imath}}}$,
where we set $\gb= \sum_\imath m_\imath \ga_{\imath}^\sg$.
The center $\wh{\bS}_\gb$ also acts on $\cF_\gb(M)$, and hence we often consider $\cF_{\gb}$ as a functor to the category of $(U_q',\wh{\bS}_\gb)$-bimodules.
Set 
\[ \cF = \bigoplus_{\gb \in Q_\sg^+} \cF_\gb\colon \bigoplus_{\gb \in Q^+_\sg} \modfg{\wh{R}(\gb)} \to U_q'\mathrm{\mathchar`-mod},
\]
which we call the \textit{generalized quantum affine Schur--Weyl duality functor}.
This functor has the following distinguished properties.

\begin{Thm}[\cite{kang2018symmetric}]\label{Thm:Properties_of_F}
 {\normalfont(i)} For any $\gb$, $\wh{W}^{\otimes \gb}$ is a flat right $\wh{R}(\gb)$-module, and hence the functor $\cF$ is exact.\\
 {\normalfont(ii)} For any $M_1 \in \modfg{\wh{R}(\gb_1)}$ and $M_2 \in \modfg{\wh{R}(\gb_2)}$, we have
  \[ \cF(M_1 \circ M_2) \cong \cF(M_1)\hat{\otimes} \cF(M_2)
  \]
  as $U_q'$-modules.
\end{Thm}

\begin{proof} 
 For finite-dimensional modules, the assertions are proved in \cite{kang2018symmetric}.
 Hence the assertions hold in general by Remark \ref{Rem:completion} (2).
 It is also easy to check that the proofs in [\textit{loc.~cit.}] also work in our completed setting. 
\end{proof}

In the rest of this paper, we will write $S(\bfm)$ for $S_{\bmi}(\bfm)$, $\bar{\gD}(\bfm)$ for $\bar{\Delta}_{\bmi}(\bfm)$, and so on.

\begin{Thm}[\cite{kang2015symmetric,kashiwara2019categorical,oh2019categorical}]\label{Thm:correspondence_of_simples}
 Let $\gb \in Q^+_\sg$.\\
 {\normalfont(i)} For any $\bfm\in KP(\gb)$, we have $\cF\big(S(\bfm)\big) \cong L\big(\bm{\pi}_{\bfm}\big)$.
  In particular, $\cF$ induces a functor from $\modfd{\wh{R}(\gb)}$ to $\cC_{\bQ,\gb}$.\\
 {\normalfont(ii)} For any $\bfm \in KP(\gb)$, $\cF\big(\bar{\gD}(\bfm)\big) \cong W\big(\bm{\pi}_{\bfm}\big)$ and $\cF\big(\bar{\nabla}(\bfm)\big) \cong W^\vee(\bm{\pi}_{\bfm}\big)$.
\end{Thm}

\begin{proof}
 For any $\ga \in R^+_\sg$, $\cF\big(S(\ga)\big) \cong L(\bm{\varpi}_\ga)$ follows from \cite[Theorem 4.3.4]{kang2015symmetric}, \cite[Theorem 5.1]{kang2016symmetric},
 \cite[Theorem 6.3]{kashiwara2019categorical}, and \cite[Remark 5.12, Lemma 6.10]{oh2019categorical}.
 Then the assertion (ii) is proved from Theorem \ref{Thm:equivalence_for_Weyl_mod} and Theorem \ref{Thm:Properties_of_F} (ii)
 by using Lemma \ref{Lem:compatible_with_words}.
 By the argument in \cite[Theorem 5.4]{kang2016symmetric}, $\cF\big(S(\bfm))$ is proved to be isomorphic to the head of $W\big(\bm{\pi}_{\bfm}\big)$, namely, $L\big(\bm{\pi}_{\bfm}\big)$.
 Hence the assertion (i) is also proved.
\end{proof}


\subsection{Image of standard modules under the functor}\label{Subsection:Image}

The purpose of this subsection is to prove the following.

\begin{Prop}\label{Prop:correspondence of standard}
 For any $\ga \in R^+_\sg$, we have 
 \[ \cF_\ga\big(\gD(\ga)\big) \cong \wh{L(\bm{\varpi}_{\ga})}_{\aff}
 \]
 as $(U_q',\wh{\bS}_{\ga})$-bimodules, where $\wh{\bS}_{\ga}$ acts on the right-hand side via the composition of $\iota\colon \wh{\bS}_\ga \to \wh{\bS}_{\bfm_\ga}$ 
 given in Subsection \ref{Subsection:Affine_highest_weight_category}
 and the natural isomorphism $\wh{\bS}_{\bfm_\ga} = \bk\llbracket w \rrbracket \stackrel{\sim}{\to} \bk\llbracket (z_{L(\bm{\varpi}_\ga)} -1)\rrbracket$.
\end{Prop}

Let $\ga \in R^+_{\sg}$, and let $i \in I_0$ be such that $\mathsf{cl}(\bm{\varpi}_{\ga}) = \varpi_i$.
Let $\kappa\colon \wh{\bS}_{\bfm_\ga} \to \wh{\bA}_{\bm{\varpi}_{\ga}}$ be the $\bk$-algebra isomorphism given by the composition
\[ \wh{\bS}_{\bfm_\ga} \stackrel{\sim}{\to} \bk\llbracket (z_{L(\bm{\varpi}_\ga)}-1) \rrbracket \stackrel{\sim}{\to}
   \bk\llbracket (z_{L(\bm{\varpi}_\ga)}^{d_i}-1)\rrbracket = \wh{\bA}_{\bm{\varpi}_{\ga}},
\]
where the second isomorphism comes from the natural injection $\bk[z^{\pm d_i}] \hookrightarrow \bk[z^{\pm 1}]$
(recall the isomorphism (\ref{eq:W_and_affinization})).
For $\bfm \in \Z_{\geq 0}^{R^+_{\sg}}$, we naturally extend this to $\kappa\colon \wh{\bS}_{\bfm} \stackrel{\sim}{\to} \wh{\bA}_{\bm{\pi}_{\bfm}}$.
The above proposition, together with Theorem \ref{Thm:Properties_of_F}, Lemma \ref{Lem:compatible_with_words} and Proposition \ref{Prop:Fujita_isom}, implies the following.

\begin{Cor}\label{Cor:correspondence_of_standard}
 For any $\gb \in Q^+_{\sg}$ and $\bfm \in KP(\gb)$, we have
 \[ \cF_\gb\big(\gD(\bfm)\big) \cong \wh{W}(\bm{\pi}_{\bfm})
 \]
 as $(U_q',\wh{\bS}_{\gb})$-bimodules, where $\wh{\bS}_{\gb}$ acts on the right-hand side via $\kappa\circ \iota$.
\end{Cor}

We shall prove Proposition \ref{Prop:correspondence of standard} by the induction on the height $|\ga|$.
The case $|\ga|=1$ follows from \cite[Proposition 3.5]{kang2018symmetric}.
Assume that $|\ga|>1$.
Let $(\gb_k,\gb_l) \in (R^+_{\sg})^2$ be a pair with $k<l$ such that $\ga=\gb_k+\gb_l$ and there is no other pair $(\gb_{k'},\gb_{l'})$ with $k<k'<l'<l$ and $\gb_{k'}+\gb_{l'}=\ga$.
Such a pair is called a \textit{minimal pair} for $\ga$.
Write $\gb=\gb_k$ and $\ggg=\gb_l$.
By \cite[Theorem 4.10]{brundan2014homological}, there is an exact sequence
\[ \gD(\gb) \circ \gD(\ggg) \stackrel{\phi}{\to} \gD(\ggg)\circ \gD(\gb) \stackrel{\psi}{\to} \gD(\ga) \to 0
\]
of $\wh{R}(\ga)$-modules satisfying $\phi(w\circ \id) = (\id \circ w)\phi$ and $\phi(\id \circ w) = (w \circ \id)\phi$.
By applying $\cF_\ga$ and using the induction hypothesis, we obtain an exact sequence
\[ \wh{L(\bm{\varpi}_\gb)}_x \hat{\otimes} \wh{L(\bm{\varpi}_\ggg)}_y \stackrel{\phi'}{\to} \wh{L(\bm{\varpi}_\ggg)}_y \hat{\otimes} \wh{L(\bm{\varpi}_\gb)}_x
   \stackrel{\psi'}{\to} \cF_\ga\big(\gD(\ga)\big) \to 0
\]
of $(U_q',\bS_\ga)$-bimodules, where 
\[ \wh{L(\bm{\varpi}_\gb)}_x = L(\bm{\varpi}_\gb)_{\aff} \otimes_{\bk[z_{L(\bm{\varpi}_\gb)}^{\pm 1}]} \bk\llbracket x\rrbracket
\]
and $z_{L(\bm{\varpi}_\gb)}$ acts on $\bk\llbracket x \rrbracket$ as $z_{L(\bm{\varpi}_\gb)}=1+x$, and $\wh{L(\bm{\varpi}_\ggg)}_y$ is defined similarly.
Note that $\phi'$ is $\bk\llbracket x,y\rrbracket$-linear.
Write $W_{\gb,\ggg} = \wh{L(\bm{\varpi}_\gb)}_x \hat{\otimes} \wh{L(\bm{\varpi}_\ggg)}_y$ and $W_{\ggg,\gb}=\wh{L(\bm{\varpi}_\ggg)}_y \hat{\otimes} \wh{L(\bm{\varpi}_\gb)}_x$.
The module $W_{\gb,\ggg}$ is a deformed local Weyl module by Proposition \ref{Prop:Fujita_isom} and Lemma \ref{Lem:compatible_with_words},
which is generated by the $P_\cl$-weight space $(W_{\gb,\ggg})_{\mathsf{cl}(\bm{\varpi}_{\gb}+\bm{\varpi}_{\ggg})}$ as a $U_q'$-module.
By taking the quotient by $\wh{\fm}_{\bm{\varpi}_{\gb}+\bm{\varpi}_{\ggg}}$, $\phi'$ induces a nonzero homomorphism from the local Weyl module 
$L(\bm{\varpi}_\gb) \otimes L(\bm{\varpi}_\ggg)$ to $L(\bm{\varpi}_\ggg) \otimes L(\bm{\varpi}_\gb)$, which is a scalar multiplication of the normalized $R$-matrix.
Hence we may assume that $\phi' = R^{\mathrm{norm}}_{L(\bm{\varpi}_\gb),L(\bm{\varpi}_\ggg)}$.
Let $M\subseteq W_{\ggg,\gb}$ be the $U_q'$-submodule generated by the weight space $(W_{\ggg,\gb})_{\mathsf{cl}(\bm{\varpi}_\gb+\bm{\varpi}_{\ggg})}$.
It follows that $\cF_\ga\big(\gD(\ga)\big)$ is isomorphic to $W_{\ggg,\gb}/M$.

\begin{Lem}
 The order of zero of $d_{L(\bm{\varpi}_\ggg),L(\bm{\varpi}_{\gb})}(z)$ at $z=1$ is one.
\end{Lem}

\begin{proof}
 See \cite[Theorems 3,11 and 4.31]{oh2019categorical}.
\end{proof}

Hence $(x-y)R^{\mathrm{norm}}_{L(\bm{\varpi}_\ggg),L(\bm{\varpi}_{\gb})}\colon W_{\ggg,\gb} \to W_{\gb,\ggg}$ is well-defined.
Then, since the image of 
\[ R^{\mathrm{norm}}_{L(\bm{\varpi}_{\gb}),L(\bm{\varpi}_{\ggg})} \circ (x-y)R^{\mathrm{norm}}_{L(\bm{\varpi}_{\ggg}),L(\bm{\varpi}_{\gb})} = (x-y)\colon W_{\ggg,\gb} \to W_{\ggg,\gb}
\]
is contained in $M$, $\cF_\ga\big(\gD(\ga)\big)$ is also isomorphic to the quotient of
\[ W_{\ggg,\gb}/(x-y) \cong \big(L(\bm{\varpi}_\ggg) \otimes L(\bm{\varpi}_\gb)\big)_\aff\sphat
\]
by the submodule generated by the weight space with weight $\mathsf{cl}(\bm{\varpi}_\gb+\bm{\varpi}_\ggg)$.
Now Proposition \ref{Prop:correspondence of standard} is easily proved from this,
since $L(\bm{\varpi}_\ga)$ is isomorphic to the quotient of $L(\bm{\varpi}_{\ggg}) \otimes L(\bm{\varpi}_{\gb})$ by the submodule generated by the weight space with weight 
$\mathsf{cl}(\bm{\varpi}_\gb+\bm{\varpi}_\ggg)$.


\section{Endomorphism algebra of $\wh{W}^{\otimes \gb}$}

\subsection{Main theorem and corollary}\label{subsection:Main theorems and corollaries}

For $\gb \in Q_\sg^+$, define a $\bk$-algebra $\mathbb{E}^\gb$ by
\[ \bE^\gb= \End_{\wh{R}(\gb)^{\mathrm{opp}}}(\wh{W}^{\otimes \gb}),
\]
and regard $\cF_\gb= \wh{W}^{\otimes \gb} \otimes_{\wh{R}(\gb)} -$ as a functor from $\modfg{\wh{R}(\gb)}$ to $\modfg{\bE^\gb}$.
Write $\gb = \sum_{\imath} m_{\imath} \ga_{\imath}^\sg$, and set $\gl = \sum_{\imath} m_\imath \mathsf{cl}(\varpi_{\ga_\imath^\sg})$.
Then, since $\wh{W}^{\otimes \gb}$ is integrable and $(\wh{W}^{\otimes \gb})_\mu =0$ for $\mu \in P_0^+ \setminus P^+_{0,\leq \gl}$, $\wh{W}^{\otimes \gb}$ is a 
$(U_{\leq \gl},\wh{R}(\gb))$-bimodule by Lemma \ref{Lem:quotient_to_U_gl},
and hence we obtain a $\bk$-algebra homomorphism
\[ \Phi_\gb \colon U_{\leq \gl} \to \bE^{\gb}.
\]
Obviously the center $\wh{\bS}_\gb$ of $\wh{R}(\gb)$
acts faithfully on $\wh{W}^{\otimes \gb}$, and hence we may (and do) consider $\wh{\bS}_\gb$ as a central subalgebra of $\bE^\gb$.
Note that $\bE^\gb$ is finitely generated as an $\wh{\bS}_\gb$-module since so is $\wh{W}^{\otimes \gb}$.
By the isomorphisms
\begin{gather*}
 L(\bm{\pi}_{\bfm}) \cong \cF_\gb\big(S(\bfm)\big), \ \ \ W(\bm{\pi}_{\bfm}) \cong \cF_\gb\big(\bar{\gD}(\bfm)\big),\\ W^\vee(\bm{\pi}_{\bfm}) \cong \cF_\gb\big(\bar{\nabla}(\bfm)\big), \ \ \
   \wh{W}(\bm{\pi}_{\bfm}) \cong \cF_\gb\big(\gD(\bfm)\big)
\end{gather*}
for $\bfm \in KP(\gb)$ given in Theorem \ref{Thm:correspondence_of_simples} and Corollary \ref{Cor:correspondence_of_standard}, 
we regard these modules as $\bE^{\gb}$-modules hereafter.

In the next subsection, we will prove the following proposition.

\begin{Prop}\label{Prop:Main_Prop}
 Let $\gb \in Q^+_\sg$.\\
 {\normalfont(i)} The pullback functor $\Phi_\gb^*$ gives an equivalence 
  \[ \Phi_\gb^*\colon \modfd{\bE^\gb}\ni M \mapsto \Phi_\gb^*M \in \cC_{\bQ,\gb}
  \]
  from the category of finite-dimensional $\bE^\gb$-modules to $\cC_{\bQ,\gb}$.\\
 {\normalfont(ii)} The category $\modfg{\bE^\gb}$ is affine highest weight with respect to the poset $(\cP^+_{\bQ,\gb},\leq )$,
  and its standard modules are $\wh{W}(\bm{\pi})$, proper standard modules are $W(\bm{\pi})$, and proper costandard modules are $W^\vee(\bm{\pi})$ with $\bm{\pi} \in \cP^+_{\bQ,\gb}$.
\end{Prop}

\begin{Rem}\normalfont
 Proposition \ref{Prop:Main_Prop} is an algebraic analog of \cite[Theorem 4.9]{fujita2017affine}, in which these statements are proved in untwisted $ADE$ types 
 with $\bE^\gb$ replaced by a completion of the convolution algebra of the equivariant $K$-groups of quiver varieties.
\end{Rem}

Now we give a proof of our main theorem, which is a direct consequence of Proposition \ref{Prop:Main_Prop}.

\begin{Thm}\label{Thm}
 For any $Q$-datum $\bQ$, the associated generalized quantum affine Schur--Weyl duality functor $\cF\colon \bigoplus_{\gb \in Q^+_\sg} \modfd{\wh{R}(\gb)} \to \cC_{\bQ}$ gives an
 equivalence of monoidal categories.
\end{Thm}

\begin{proof}
 Thanks to Theorem \ref{Thm:affine_hw_cat_of_R}, Proposition \ref{Prop:isom_of_posets} and Proposition \ref{Prop:Main_Prop} (ii),
 for any $\gb \in Q^+_\sg$ it follows from Theorem \ref{Thm:Fujita_theorem} that $\cF_\gb\colon \modfg{\wh{R}(\gb)} \to \modfg{\bE^\gb}$  
 gives an equivalence of categories, which induces an equivalence between full subcategories of finite-dimensional modules.
 Hence by Proposition \ref{Prop:Main_Prop} (i) and Lemma \ref{Lem:direct_sum_of_categoryC}, we see that $\cF\colon \bigoplus_{\gb \in Q^+_\sg} \modfd{\wh{R}(\gb)} \to \cC_{\bQ}$ gives an equivalence.
 Since $\cF$ is a monoidal functor by Theorem \ref{Thm:Properties_of_F}, the theorem is proved. 
\end{proof}

By the theorem, we obtain equivalences among several subcategories of $\cC$ corresponding to various choices of $Q$-data.
More importantly, the theorem also gives equivalences between categories of \textit{different types}, which we state as a corollary here.
Let $\fg^{(1)}$ be an affine Kac-Moody Lie algebra of type $X_n^{(1)}$ ($X_n \in \{A_n, D_n,E_6\}$), and let $\fg^{(t)}$ ($t \in \{2,3\}$) denote the algebra of a twisted type $X_n^{(t)}$.
Let ${}^L\fg^{(t)}$ be the Langlands dual Lie algebra of $\fg^{(t)}$.
Take $Q$-data $\bQ^{(1)}$, $\bQ^{(t)}$, $\bQ^L$ of respective types, and let $\cC_{\bQ^{(1)}}$, $\cC_{\bQ^{(t)}}$, $\cC_{\bQ^{L}}$ be the corresponding subcategories.
Note that these are categories of modules over different quantum affine algebras.

\begin{Cor}\label{Cor}
 The monoidal categories $\cC_{\bQ^{(1)}}$, $\cC_{\bQ^{(t)}}$, and $\cC_{\bQ^{L}}$ are mutually equivalent.
\end{Cor}

\begin{proof}
 This follows since the corresponding quiver Hecke algebras in the theorem are mutually isomorphic.
\end{proof}

\begin{Rem}\normalfont
 The equivalence between $\cC_{\bQ^{(1)}}$ and $\cC_{\bQ^{(t)}}$ preserves the dimensions of modules by \cite{kang2016symmetric,oh2019categorical}.
 This is not the case for the equivalence between $\cC_{\bQ^{L}}$ and $\cC_{\bQ^{(1)}}$, or $\cC_{\bQ^{L}}$ and $\cC_{\bQ^{(t)}}$.
 This equivalence does not preserve even fundamental modules (for an example of such a phenomenon, see \cite[Theorem 12.9]{hernandez2019quantum}).
\end{Rem}


\subsection{Proof of Proposition \ref{Prop:Main_Prop}}\label{subsection:proof_of_main}

As in Subsection \ref{Subsection:Modified_QAA}, fix a numbering $\{\gl_1,\ldots,\gl_l\}$ on the elements of $P_{0,\leq \gl}^+$ such that 
$\gl_l = \gl$ and $r<s$ holds whenever $\gl_r < \gl_s$.
Note that $\mathsf{cl}(\bm{\pi}) \leq \gl$ holds for any $\bm{\pi} \in \cP^+_{\bQ,\gb}$, since $L(\bm{\pi})$ is a subquotient of $\wh{W}^{\otimes\gb}$.
For $1\leq k\leq l$, set
\[ \cP^{+}_k = \{ \bm{\pi} \in \cP^+_{\bQ,\gb} \mid \mathsf{cl}(\bm{\pi}) = \gl_k\},
\]
and $KP_{k} = \Omega_{\bQ}^{-1} (\cP^+_k) \subseteq KP(\gb)$.
By Proposition \ref{Prop:isom_of_posets}, if $\bfm \prec \bfn$ holds for some $\bfm \in KP_{k}$ and $\bfn \in KP_{r}$, then we have $k<r$.
For $0\leq k \leq l$, we also set $KP_{\leq k} = \bigsqcup_{r\leq k} KP_r$, 
and define a two-sided ideal $R_k = \mathcal{O}^{KP_{\leq k}}\big(\wh{R}(\gb)\big)$ of $\wh{R}(\gb)$ (see Section \ref{Section:affine_quasi_hereditary}), which gives a chain of ideals
\[ \wh{R}(\gb)=R_0 \supseteq R_1 \supseteq \cdots \supseteq R_l = \{0\}.
\]
Set $R^k = \wh{R}(\gb)/R_k$.
For each $k$, fix a numbering $KP_k = \{\bfm_{k,1},\ldots,\bfm_{k,s_k}\}$.
By Theorem \ref{Thm:hereditary_corresp}, $\wh{R}(\gb)$ is an affine quasihereditary algebra. 
Then, by definition, the quotient $R^k$ is also an affine quasihereditary algebra, and for each $1\leq j \leq s_k$ an affine heredity ideal $R^k_{j}$ of $R^k$ is obtained by setting
\[ R^k_j = \mathcal{O}^{KP_{\leq k}\setminus \{\bfm_{k,j}\}}(R^k),
\]
since $\bfm_{k,j}$ is a maximal element of $KP_{\leq k}$.
Since $\mathrm{Hom}_{\modfg{\wh{R}(\gb)}}\!(R_j^k,R_{j'}^k) = 0$ for $j\neq j'$ (see Definition \ref{Def:affine_heredity_ideal}), we have 
\begin{equation}\label{eq:direct_sum}
 \bigoplus_{j=1}^{s_k} R^k_{j} = \mathcal{O}^{KP_{\leq k-1}}(R^k)=R_{k-1}/R_k.
\end{equation}
Moreover, thanks to the last paragraph of Subsection \ref{Subsection:SQHA}, we can apply Theorem \ref{Thm:affine_cellular} with $\cA=R^k$, $\mathcal{J}=R^k_j$, 
and $\pi=\bfm_{k,j}$.
Hence $R^k_{j}$ is $\psi$-invariant, and there is an idempotent $e_{k,j} \in R^k$ satisfying
\begin{align}\label{eq:affine_cellular_of_R}
 \begin{aligned}
  R^k_{j} = R^k e_{k,j}R^k\cong \gD(\bfm_{k,j}) \otimes_{\wh{\bS}_{\bfm_{k,j}}} \gD(\bfm_{k,j})^\psi.
 \end{aligned}
\end{align}
In the following we sometimes write $\wh{\bS}$ for $\wh{\bS}_{\bfm}$, when $\bfm$ is obvious from the context.

For each $0\leq k \leq l$, set $\wh{W}_k = \wh{W}^{\otimes\gb} \otimes_{\wh{R}(\gb)} R_{k}$.
Since $\wh{W}^{\otimes\gb}$ is a flat right $\wh{R}(\gb)$-module, we may regard $\wh{W}_k$ as an $(\bE^\gb,\wh{R}(\gb))$-submodule of $\wh{W}^{\otimes\gb}$.
We write $\bm{\pi}_{k,j}$ for $\bm{\pi}_{\bfm_{k,j}}$ hereafter.
For each $1\leq k \leq l$, it follows from (\ref{eq:direct_sum}), (\ref{eq:affine_cellular_of_R}) and Corollary \ref{Cor:correspondence_of_standard} that
\begin{align}\label{eq:subquotient_of_W}
 \wh{W}_{k-1}/\wh{W}_{k} &\cong \wh{W}^{\otimes \gb} \otimes_{\wh{R}(\gb)} R_{k-1}/R_k \cong \bigoplus_{j=1}^{s_k} \wh{W}^{\otimes \gb} \otimes_{\wh{R}(\gb)} 
 (\gD(\bfm_{k,j}) \otimes_{\wh{\bS}} \gD(\bfm_{k,j})^\psi)\nonumber\\
 &\cong \bigoplus_{j=1}^{s_k} \wh{W}(\bm{\pi}_{k,j}) \otimes_{\wh{\bS}} \gD(\bfm_{k,j})^\psi
\end{align} 
as $(\bE^\gb,\wh{R}(\gb))$-bimodules.
For each $\bm{\pi} \in \cP_{\bQ,\gb}^+$, setting $\mu=\mathsf{cl}(\bm{\pi})$, we write 
\[ \wh{W}^{\sharp}(\bm{\pi})=\wh{\bA}_{\bm{\pi}} \otimes_{\bA_{\mu}} \bV(\mu)^{\sharp}, \ \ \text{and} \ \ \wh{w}^\sharp_{\bm{\pi}} = 1 \otimes v_\mu \in \wh{W}^{\sharp}(\bm{\pi}).
\]
By Lemma \ref{Lem:dual_deformed}, we have $\wh{W}^{\sharp}(\bm{\pi}) \cong \Hom{\wh{\bA}_{\bm{\pi}}}(\cF_\gb(\gD(\bfm_L\gb_L) \circ \cdots \circ \gD(\bfm_1\gb_1)),\wh{\bA}_{\bm{\pi}})$
with $\Omega^{-1}_{\bQ}(\bm{\pi}) = \bfm$, and hence we can regard $\wh{W}^{\sharp}(\bm{\pi})$ as an $(\wh{\bA}_{\bm{\pi}},\bE^{\gb})$-bimodule.

\begin{Lem}\label{Lem:bimodule_isom}
 For each $1\leq k \leq l$, we have 
 \begin{equation}\label{eq:bimodule_isom}
  \Hom{\wh{R}(\gb)^{\mathrm{opp}}}(\wh{W}^{\otimes \gb},\wh{W}_{k-1}/\wh{W}_{k}) \cong \bigoplus_{j=1}^{s_k} \wh{W}(\bm{\pi}_{k,j}) \otimes_{\wh{\bA}_{\bm{\pi}_{k,j}}} \wh{W}^{\sharp}(\bm{\pi}_{k,j})
 \end{equation}
 as $(\bE^{\gb},\bE^{\gb})$-bimodules.
\end{Lem}

\begin{proof}
 By (\ref{eq:subquotient_of_W}), we have
 \begin{align}\label{eq:first_isom}
  \Hom{\wh{R}(\gb)^{\mathrm{opp}}}(\wh{W}^{\otimes \gb}, \wh{W}_{k-1}/\wh{W}_k)
  \cong \bigoplus_{j} \wh{W}(\bm{\pi}_{k,j}) \otimes_{\wh{\bS}} \Hom{\wh{R}(\gb)^{\mathrm{opp}}}(\wh{W}^{\otimes \gb}, \gD(\bfm_{k,j})^{\psi}).
 \end{align}
 For a moment fix $j$ arbitrarily, write $\bfm_{k,j}(\gb_r) = m_r$ for $1\leq r \leq L$, and set $R=\wh{R}(m_L\gb_L) \hat{\otimes} \cdots \hat{\otimes} \wh{R}(m_1\gb_1)$. 
 By Proposition \ref{Prop:coinduction}, we have
 \begin{align*}
  &\Hom{\wh{R}(\gb)^{\mathrm{opp}}}(\wh{W}^{\otimes \gb}, \gD(\bfm_{k,j})^{\psi})\\
  &\cong \mathrm{Hom}_{R^{\mathrm{opp}}}
  \Big(\wh{W}^{\otimes \gb}e(m_L\gb_L,\ldots,m_1\gb_1),\big(\gD(m_L\gb_L)
    \hat{\otimes} \cdots \hat{\otimes} \gD(m_1\gb_1)\big)^\psi\Big)\\
  &\cong \mathrm{Hom}_{R^{\mathrm{opp}}}\Big(\wh{W}^{\otimes m_L\gb_L} \hat{\otimes} \cdots \hat{\otimes} \wh{W}^{\otimes m_1\gb_1},\big(\gD(m_L\gb_L)
    \hat{\otimes} \cdots \hat{\otimes} \gD(m_1\gb_1)\big)^\psi\Big)\\
  &\cong \mathrm{Hom}_{\wh{R}(m_L\gb_L)^{\mathrm{opp}}}\big(\wh{W}^{\otimes m_L\gb_L},\gD(m_L\gb_L)^\psi\big) \hat{\otimes} \cdots \hat{\otimes} 
   \mathrm{Hom}_{\wh{R}(m_1\gb_1)^{\mathrm{opp}}}\big(\wh{W}^{\otimes m_1\gb_1},\gD(m_1\gb_1)^\psi\big)
 \end{align*}
 as $(\wh{\bS}_{\gb},\bE^{\gb})$-bimodules.
 For each $1\leq s \leq L$, $m_s\bfm_{\gb_s}$ is a minimal element of $KP(m_s\gb_s)$, 
 and hence there is a two-sided ideal $I_s$ of $\wh{R}(m_s\gb_s)$ such that $\wh{R}(m_s\gb_s)/I_s \cong \gD(m_s\gb_s) \otimes_{\wh{\bS}}
 \gD(m_s\gb_s)^\psi$ (note that $\gD(m\ga)=\gD(m\bfm_\ga)$ for $m \in \Z_{> 0}$ and $\ga \in R^+_\sg$ in our notation).
 Hence we have 
 \begin{align*}
  \mathrm{Hom}_{\wh{R}(m_s\gb_s)^{\mathrm{opp}}}&(\wh{W}^{\otimes m_s\gb_s},\gD(m_s\gb_s)^\psi )\cong \mathrm{Hom}_{\wh{R}(m_s\gb_s)^{\mathrm{opp}}}(\wh{W}^{\otimes m_s\gb_s}/
  I_s, \gD(m_s\gb_s)^\psi)\\
  \cong &\mathrm{Hom}_{\wh{R}(m_s\gb_s)^{\mathrm{opp}}}(\wh{W}(m_s\bm{\varpi}_{\gb_s})\otimes_{\wh{\bS}} \gD(m_s\gb_s)^\psi,\gD(m_s\gb_s)^\psi)\\
  \cong &\mathrm{Hom}_{\wh{\bS}}\big(\wh{W}(m_s\bm{\varpi}_{\gb_s}), \mathrm{End}_{\wh{R}(m_s\gb_s)^{\mathrm{opp}}}(\gD(m_s\gb_s)^{\psi})\big) \cong 
  \mathrm{Hom}_{\wh{\bS}}\big(\wh{W}(m_s\bm{\varpi}_{\gb_s}), \wh{\bS})
 \end{align*}
 as $(\wh{\bS}_{m_s\bfm_{\gb_s}},\bE^{m_s\gb_s})$-bimodules,
 where the last isomorphism follows from Lemma \ref{Lem:action_of_centers}.
 Hence we have 
 \begin{align*}
  \Hom{\wh{R}(\gb)^{\mathrm{opp}}}(\wh{W}^{\otimes \gb}, \gD(\bfm_{k,j})^{\psi}) &\cong \mathrm{Hom}_{\wh{\bS}_{\bfm_{k,j}}}(\wh{W}(m_L\bm{\varpi}_{\gb_L}) \hat{\otimes}\cdots \hat{\otimes} 
  \wh{W}(m_1\bm{\varpi}_{\gb_1}), \wh{\bS}_{\bfm_{k,j}})\\
  &\cong \wh{W}^{\sharp}(\bm{\pi}_{k,j}),
 \end{align*}
 and the assertion is proved.
\end{proof}

For $0\leq k \leq l$, set
\[ \bE_k =\{ f \in \bE^\gb\mid \mathrm{Im}\, f \subseteq \wh{W}_k\} = \mathrm{Ann}_{\bE^{\gb}}\big(\wh{W}^{\otimes \gb}/\wh{W}_k\big) ,
\]
which defines a chain of ideals
\[ \bE^{\gb} = \bE_0 \supseteq \bE_1 \supseteq \cdots \supseteq \bE_l = \{0\}.
\]
For $\mu \in P_0$ such that $(\wh{W}^{\otimes \gb})_\mu \neq \{0\}$, we will simply write $a_\mu$ for $\Phi_\gb(a_\mu)$. 
Let $U_{\leq \gl} = U_0 \supseteq \cdots \supseteq U_l = \{0\}$ be the chain of ideals given in Subsection \ref{Subsection:Modified_QAA}.

\begin{Lem}\label{Lem:comparing_subquotient}
 {\normalfont(i)} For each $1\leq k \leq l$, $\Phi_\gb\colon U_{\leq \gl} \to \bE^{\gb}$ induces a $(U_q',U_q')$-bimodule homomorphism
  \[ \Phi_{\gb,k}\colon U_{k-1}/U_k \to \bE_{k-1}/\bE_k.
  \]
 {\normalfont(ii)} For each $1\leq k \leq l$, we have $a_{\gl_k} \in \bE_{k-1}$.
\end{Lem}

\begin{proof}
 We see from (\ref{eq:subquotient_of_W}) that $\big(\wh{W}^{\otimes \gb}/\wh{W}_k\big)_\mu=0$ for $\mu \in P^+_0 \setminus \{\gl_1,\ldots,\gl_k\}$, 
 and hence we have $\Phi_\gb(U_k) \subseteq
 \mathrm{Ann}_{\bE^{\gb}}\big(\wh{W}^{\otimes \gb}/\wh{W}_k\big) = \bE_k$ by Lemma \ref{Lem:quotient_to_U_gl}, which implies the assertion (i).
 Now the assertion (ii) is obvious.
\end{proof}

\begin{Prop}\label{Prop:Pretend_to_Fujita}
 {\normalfont(i)} For each $1\leq k \leq L$, we have
 \[ \bE_{k-1}/\bE_k \cong \bigoplus_j \wh{W}(\bm{\pi}_{k,j}) \otimes_{\wh{\bA}_{\bm{\pi}_{k,j}}} \wh{W}^{\sharp}(\bm{\pi}_{k,j})
 \]
 as $(\bE^{\gb},\bE^{\gb})$-bimodules.\\
 {\normalfont(ii)} For any $m \in \Z_{>0}$, the composition 
  \[ \Phi_{\gb}^m\colon U_{\leq \gl} \stackrel{\Phi_\gb}{\to} \bE^{\gb} \twoheadrightarrow \bE^{\gb} \Big/ \big(\wh{\bS}_{\gb}^+\big)^m 
  \]
  is surjective.
\end{Prop}

\begin{proof}
  By definition, there is a natural injective $(\bE^\gb,\bE^\gb)$-bimodule homomorphism $\phi_k$
  from $\bE_{k-1}/\bE_k$ to $\Hom{\wh{R}(\gb)^{\mathrm{opp}}}(\wh{W}^{\otimes \gb},\wh{W}_{k-1}/\wh{W}_{k})$.
  By composing this with (\ref{eq:bimodule_isom}), we obtain an injective homomorphism
  \[ \psi_k \colon \bE_{k-1}/\bE_k \hookrightarrow \bigoplus_j \wh{W}(\bm{\pi}_{k,j}) \otimes_{\wh{\bA}_{\bm{\pi}_{k,j}}} \wh{W}^{\sharp}(\bm{\pi}_{k,j}).
  \]
  In order to show the assertion (i), we need to show the surjectivity of $\psi_k$.
  First we claim that 
  \[ \psi_k (\bar{a}_{\gl_k}) \in \sum_j \wh{\bA}_{\bm{\pi}_{k,j}}^\times (\wh{w}_{\bm{\pi}_{k,j}} \otimes \wh{w}^\sharp_{\bm{\pi}_{k,j}}),
  \]
  where $\bar{a}_{\gl_k}$ is the image of $a_{\gl_k}$ under the natural projection, and $A^\times$ denotes the subset of units of an algebra
  $A$.
  For a $(U_q',U_q')$-bimodule $M$ and $\mu \in P_0$, set $M_{(\mu,\mu)} = \{v \in M \mid q^hv=q^{\langle h,\mu\rangle}v=vq^h \ \text{ for $h \in P_\cl^\vee$}\}$.
  There are natural isomorphisms
  \begin{align*}
   &\Hom{\wh{R}(\gb)^{\mathrm{opp}}}(\wh{W}^{\otimes \gb}, \wh{W}_{k-1}/\wh{W}_k)_{(\gl_k,\gl_k)} \cong \Hom{\wh{R}(\gb)^{\mathrm{opp}}}(\wh{W}^{\otimes \gb}/\wh{W}_k,
     \wh{W}_{k-1}/\wh{W}_k)_{(\gl_k,\gl_k)}\\
   &\cong\mathrm{End}_{\wh{R}(\gb)^{\mathrm{opp}}}\big(\bigoplus_j \gD(\bfm_{k,j})^\psi\big) 
    \cong \bigoplus_j \mathrm{End}_{\wh{R}(\gb)^{\mathrm{opp}}} \big(\gD(\bfm_{k,j})^\psi\big) \cong \bigoplus_j \bS_{\bfm_{k,j}},
  \end{align*}
  where the second isomorphism follows since we have
  \[ (\wh{W}^{\otimes \gb}/\wh{W}_k)_{\gl_k} \cong \bigoplus_j \gD(\bfm_{k,j})^\psi \cong(\wh{W}_{k-1}/\wh{W}_k)_{\gl_k} 
  \]
  by (\ref{eq:subquotient_of_W}), and the third follows from (\ref{eq:definition_of_standard}).
  Obviously $\bar{a}_{\gl_k} \in \left(\bE_{k-1}/\bE_{k}\right)_{(\gl_k,\gl_k)}$,
  and the image $\phi_k(\bar{a}_{\gl_k})\in \Hom{\wh{R}(\gb)^{\mathrm{opp}}}(\wh{W}^{\otimes \gb}, \wh{W}_{k-1}/\wh{W}_k)$ 
  corresponds to $(1,1,\ldots,1)\in \bigoplus_j \bS_{\bfm_{k,j}}$ via the above isomorphism.
  Hence for each $j$, we see that $\mathrm{pr}_j\circ \psi_k(\bar{a}_k)$ generates 
  \[ \Big(\wh{W}(\bm{\pi}_{k,j}) \otimes \wh{W}^{\sharp}(\bm{\pi}_{k,j})\Big)_{(\gl_k,\gl_k)} = 
     \wh{\bA}_{\bm{\pi}_{k,j}} (\wh{w}_{\bm{\pi}_{k,j}} \otimes \wh{w}^\sharp_{\bm{\pi}_{k,j}}),
  \]
  where $\mathrm{pr}_j$ denotes the projection to the $j$-th summand.
  Now the claim is proved.

  For any $m \in \Z_{>0}$, we have 
  \begin{align}\label{eq:chinese_remainder_thm}
   \Big(\bigoplus_j \wh{W}(\bm{\pi}_{k,j}) \otimes_{\wh{\bA}_{\bm{\pi}_{k,j}}}\wh{W}^{\sharp}(\bm{\pi}_{k,j})\Big)\Big/ &(\wh{\bS}^+_\gb)^m 
   \cong \bigoplus_j \Big(\wh{W}(\bm{\pi}_{k,j}) \otimes_{\wh{\bA}_{\bm{\pi}_{k,j}}}\wh{W}^{\sharp}(\bm{\pi}_{k,j})\Big/\wh{\fm}_{\bm{\pi}_{k,j}}^m\Big)\nonumber\\
   &\cong \bigoplus_j\wh{W}(\bm{\pi}_{k,j}) \otimes_{\wh{\bA}_{\bm{\pi}_{k,j}}} \left(\frac{\wh{\bA}_{\bm{\pi}_{k,j}}}{\wh{\fm}^m_{\bm{\pi}_{k,j}}}\right) 
   \otimes_{\wh{\bA}_{\bm{\pi}_{k,j}}}
   \wh{W}^{\sharp}(\bm{\pi}_{k,j})\nonumber\\
   & \cong \bigoplus_j\bV(\gl_k) \otimes_{\bA_{\gl_k}} \left(\frac{\bA_{\gl_k}}{\fm_{\bm{\pi}_{k,j}}^m}\right)\otimes_{\bA_{\gl_k}} \bV(\gl_k)^{\sharp}\nonumber\\
  & \cong \bV(\gl_k) \otimes_{\bA_{\gl_k}} \left(\frac{\bA_{\gl_k}}{\prod_j \fm_{\bm{\pi}_{k,j}}^m}\right)\otimes_{\bA_{\gl_k}} \bV(\gl_k)^{\sharp}
  \end{align}
  by the Chinese remainder theorem, and hence the composition
  \begin{equation}\label{eq:composition_is_surjective}
   U_{k-1}/U_k \stackrel{\Phi_{\gb,k}^m}{\to}\left(\bE_{k-1}/\bE_k\right)\Big/ (\wh{\bS}^+_{\gb})^m \stackrel{\psi^m_k}{\to}
     \left(\bigoplus_j \wh{W}(\bm{\pi}_{k,j}) \otimes_{\wh{\bA}_{\bm{\pi}_{k,j}}}\wh{W}^{\sharp}(\bm{\pi}_{k,j})\right)\Big/ (\wh{\bS}^+_\gb)^m
  \end{equation}
  is surjective by Theorem \ref{Thm:affine_cellularity_of_U} and the above claim, which implies that the map $\psi_k^m$ is surjective for any $m$.
  Since $\psi_k$ is an $\wh{\bS}_\gb$-module homomorphism, the surjectivity of $\psi_k$ also follows, and the assertion (i) is proved.

  Fix $m \in \Z_{>0}$. 
  Since $\psi_k^m$ in (\ref{eq:composition_is_surjective}) is an isomorphism, we see that the map $\Phi_{\gb,k}^m$ is surjective for any $k$, 
  and then by applying the five lemma to the following diagram
  \begin{equation}\label{eq:commutative diagram}
   \xymatrix{ 
   0 \ar[r] & U_{k-1}/U_k\ar[r]^{a_k} \ar[d]_{\Phi_{\gb,k}^m} &  U_{\leq \gl}/U_k \ar[r]^{b_k}\ar[d]_{\Phi_{\gb,\leq k}^m} & U_{\leq \gl}/U_{k-1} \ar[r] \ar[d]_{\Phi_{\gb,\leq {k-1}}^m} & 0\\
            & \big(\bE_{k-1}/\bE_k\big)\big/(\wh{\bS}^+_{\gb})^m \ar[r] & \big(\bE^{\gb}/\bE_k\big)\big/(\wh{\bS}^+_{\gb})^m \ar[r] & 
  \big(\bE^{\gb}/\bE_{k-1}\big)\big/(\wh{\bS}^+_{\gb})^m\ar[r] & 0,}
  \end{equation}
  we can inductively show that the map $\Phi_{\gb,\leq k}^m$ is surjective for any $1\leq k \leq l$.
  Since $\Phi^m_{\gb} = \Phi^m_{\gb,\leq l}$, the assertion (ii) is proved.
\end{proof}

Although the remaining part of our proof of Proposition \ref{Prop:Main_Prop} is similar to that of \cite[Theorem 4.9]{fujita2017affine}, we will give it for completeness.
Using the notation in the proof of the above proposition, set 
\[ K_k^m = \mathrm{Ker}\,\Phi_{\gb,k}^m \subseteq U_{k-1}/U_k \ \text{ and } \ K_{\leq k}^m = \mathrm{Ker}\, \Phi_{\gb,\leq k}^m \subseteq U_{\leq \gl}/U_{k}
\]
for each $1\leq k \leq l$ and $m \in \Z_{>0}$.

\begin{Lem}\label{Lem:final_Lemma}
 For any $1\leq k \leq l$ and $m_1,m_2 \in \Z_{>0}$, there exists some $m > m_1+m_2$ satisfying:
 \begin{itemize}
  \setlength{\leftskip}{-0.5cm}
  \setlength{\parskip}{0.2cm} 
  \item[{\normalfont(i)}] $K_k^m \subseteq K_k^{m_1} \cdot K_k^{m_2}$;
  \item[{\normalfont(ii)}] $K_{\leq k}^m \subseteq K_{\leq k}^{m_1}\cdot K_{\leq k}^{m_2}$.
 \end{itemize}
\end{Lem}

\begin{proof}
 For any $k$, it follows from (\ref{eq:chinese_remainder_thm}) and Theorem \ref{Thm:affine_cellularity_of_U} that
 \[ K_k^m = \Big(\prod_{j=1}^{s_j} \fm_{\bm{\pi}_{k,j}}^m\Big)\cdot U_{k-1}/U_k,
 \]
 and $(U_{k-1}/U_k)^2 = U_{k-1}/U_k$ holds since this is generated by an idempotent.  
 Now the assertion (i) is easily proved from these equalities.
 Let us prove the assertion (ii) by the induction on $k$.
 The case $k=1$ follows from (i).
 Assume that $k>1$.
 By (i), we can take an integer $M$ such that $M>m_1+m_2$ and $K_k^M \subseteq K_k^{m_1}\cdot K_k^{m_2}$ hold. 
 By the Artin-Rees lemma, there is a positive integer $M'$ such that $M'>M$ and 
 \[ (\bE_{k-1}/\bE_k)\cap \Big((\wh{\bS}_{\gb}^+)^{M'}(\bE^{\gb}/\bE_k)\Big) \subseteq (\wh{\bS}_{\gb}^+)^M (\bE_{k-1}/\bE_k),
 \] 
 which implies
 \begin{equation}\label{equation:description_of_kernel}
  \mathrm{Ker}(\Phi_{\gb,\leq k}^{M'} \circ a_k) \subseteq K_k^M \subseteq K_k^{m_1}\cdot K_k^{m_2},
 \end{equation}
 where $a_k$ is the map given in the diagram (\ref{eq:commutative diagram}).
 Then by applying the five lemma to the diagram
  \begin{equation*}
   \xymatrix{ 
   0 \ar[r] & U_{k-1}/U_k\ar[r]^{a_k} \ar[d]_{\Phi_{\gb,\leq k}^{M'}\circ a_k} &  U_{\leq \gl}/U_k \ar[r]^{b_k}\ar[d]_{\Phi_{\gb,\leq k}^{M'}} & 
   U_{\leq \gl}/U_{k-1} \ar[r] \ar[d]_{\Phi_{\gb,\leq {k-1}}^{M'}} & 0\\
   0 \ar[r]         & \mathrm{Im}(\Phi_{\gb,\leq k}^{M'} \circ a_k) \ar[r] & \big(\bE^{\gb}/\bE_k\big)\big/(\wh{\bS}_{\gb}^+)^{M'} \ar[r] & \big(\bE^{\gb}/\bE_k\big)\big/(\wh{\bS}_{\gb}^+)^{M'}
   \ar[r] & 0,}
  \end{equation*}
  we obtain an exact sequence
  \[ 0 \to \mathrm{Ker}\, (\Phi_{\gb,\leq k}^{M'} \circ a_k) \to K_{\leq k}^{M'} \to K_{\leq k-1}^{M'} \to 0.
  \]
  By the induction hypothesis, there is $m>M'$ such that $K_{\leq k-1}^m \subseteq K_{\leq k-1}^{M'}\cdot K_{\leq k-1}^{M'}$,
  and we claim that this $m$ satisfies the assertion (ii).
  Let $X \in K_{\leq k}^m$ be an arbitrary element.
  Since $b_k(X) \in K_{\leq k-1}^m \subseteq K_{\leq k-1}^{M'} \cdot K_{\leq k-1}^{M'}$ and there is a surjection $K_{\leq k}^{M'} \cdot K_{\leq k}^{M'} \twoheadrightarrow K_{\leq k-1}^{M'}
  \cdot K_{\leq k-1}^{M'}$,
  there is $Y \in K_{\leq k}^{M'} \cdot K_{\leq k}^{M'}$ such that $b_k(X-Y) =0$.
  Then there is $Z \in \mathrm{Ker}\, (\Phi_{\gb,\leq k}^{M'} \circ a_k)$ such that $a_k(Z) = X-Y$.
  Now we have $a_k(Z) \in a_k\big(K_k^{m_1}\cdot K_k^{m_2}\big) \subseteq K_{\leq k}^{m_1} \cdot K_{\leq k}^{m_2}$ by (\ref{equation:description_of_kernel}), 
  and hence $X \in K_{\leq k}^{m_1}\cdot K_{\leq k}^{m_2}$ holds, as required.
  The proof is complete.
\end{proof}

\noindent\textit{Proof of Proposition \ref{Prop:Main_Prop}}. (i)
Since every object in $\modfd{\bE^\gb}$ can be regarded as an $\big(\bE^\gb/(\wh{\bS}_{\gb}^+)^m\big)$-module for a sufficiently large $m$, 
it follows from Proposition \ref{Prop:Pretend_to_Fujita} (ii) that the functor $\Phi^*_\gb$ is fully faithful.
In order to prove the essentially surjectivity of $\Phi_\gb^*$, it is enough to show for any $M \in \cC_{\bQ,\gb}$ that there is some $m \in \Z_{>0}$ such that
$(\Ker\, \Phi_\gb^m)M=0$.
We shall prove this by the induction on the length of $M$.
If $M$ is simple, the assertion holds (with $m=1$) by Theorem \ref{Thm:correspondence_of_simples} (i).
Assume that the length of $M$ is larger than $1$, and take a nonzero proper submodule $N \subsetneq M$.
By the induction hypothesis, there are positive integers $m_1,m_2$ such that $(\mathrm{Ker}\, \Phi_\gb^{m_1})N = (\mathrm{Ker}\,\Phi_{\gb}^{m_2})(M/N) =0$.
By Lemma \ref{Lem:final_Lemma} (ii) with $k=l$, there is $m$ such that $\mathrm{Ker}\,\Phi_{\gb}^m \subseteq \mathrm{Ker}\, \Phi_\gb^{m_1}\cdot \mathrm{Ker}\,\Phi_{\gb}^{m_2}$,
which implies $(\mathrm{Ker}\,\Phi_{\gb}^m)M = 0$. The proof is complete.\\
(ii) Before starting the proof, note that, for any $M,N \in \modfg{\bE^\gb}$, we have 
\begin{align}\label{eq:projlimit}
 \Hom{\bE^\gb}(M,N) &= \varprojlim_m \Hom{\bE^\gb}\big(M\big/(\wh{\bS}_\gb^+)^m,N\big/(\wh{\bS}_\gb^+)^m\big)\nonumber\\
 &= \varprojlim_m \Hom{U_q'}\big(M\big/(\wh{\bS}_\gb^+)^m,N\big/(\wh{\bS}_\gb^+)^m\big)
\end{align}
by the assertion (i).

For $1\leq k \leq l$ and $1\leq j \leq s_k$, let $R_{k,j} \subseteq \wh{R}(\gb)$ be the two-sided ideal obtained as the inverse image of 
$\bigoplus_{j< r \leq s_k}R^k_{r}$ under the quotient map $\wh{R}(\gb) \to R^k=\wh{R}(\gb)/R_k$.
Then 
\[ \wh{R}(\gb) =R_{0} \supsetneq R_{1,1}\supsetneq \cdots \supsetneq R_{l,s_l-1} \supsetneq R_{l,s_l}=\{0\}
\]
is an affine heredity chain of $\wh{R}(\gb)$.
Set $\wh{W}_{k,j} = \wh{W}^{\otimes \gb} \otimes_{\wh{R}(\gb)} R_{k,j}\subseteq \wh{W}^{\otimes \gb}$, and
\[ \bE_{k,j} = \{f \in \bE^{\gb}\mid \mathrm{Im}\, f \subseteq \wh{W}_{k,j}\}.
\]
We shall prove that the sequence of two-sided ideals
\begin{equation}\label{eq:affine_heredity_chain}
 \bE^{\gb}=\bE_{0} \supsetneq \bE_{1,1}\supsetneq \cdots \supsetneq \bE_{l,s_l-1} \supsetneq \bE_{l,s_l}=\{0\}
\end{equation}
is an affine heredity chain. 
Fix $1\leq k \leq l$ and $1\leq j \leq s_k$ arbitrarily.
By the same proof with Proposition \ref{Prop:Pretend_to_Fujita} (i), we have
\[ \bE_{k,j-1}/\bE_{k,j} \cong \wh{W}(\bm{\pi}_{k,j}) \otimes_{\wh{\bA}_{\bm{\pi}_{k,j}}} \wh{W}^\sharp(\bm{\pi}_{k,j})
\]
as $(\bE^{\gb},\bE^{\gb})$-bimodules,
where we set $\bE_{k,0} = \bE_{k-1}$.
The right-hand side is isomorphic as a left $\bE^{\gb}$-module to $\wh{W}(\bm{\pi}_{k,j}\big)^{\oplus s}$ with $s=\dim W(\bm{\pi}_{k,j})$.
We write $\bE=\bE^\gb/\bE_{k,j}$ for brevity,
and claim that $\wh{W}\big(\bm{\pi}_{k,j}\big)$ is a projective object in $\modfg{\bE}$.
Let $N$ be an arbitrary object in $\modfg{\bE}$. 
Since $N$ is also an $(\bE^{\gb}/\bE_k)$-module, $N$ satisfies $N_\mu = 0$ for $\mu \in P^+_0\setminus \{\gl_1,\ldots,\gl_k\}$ by Lemma \ref{Lem:comparing_subquotient} (ii).
Hence by (\ref{eq:taking_highest_weight}) and (\ref{eq:projlimit}), we have
\begin{align}\label{eq:projlim_eigenspace}
 \Hom{\bE}\big(\wh{W}(\bm{\pi}_{k,j}),N\big) = \varprojlim_m \left(N\big/(\wh{\bS}_\gb^+)^m\right)_{\bm{\pi}_{k,j}}.
\end{align}
We easily see from this that $\mathrm{Hom}_{\bE}\big(\wh{W}(\bm{\pi}_{k,j}),-\big)$ is exact, and hence $\wh{W}(\bm{\pi}_{k,j})$ is projective, as required.
Moreover, since $\bE^{\gb}/\bE_{k,j-1}$ has a filtration whose subquotients are isomorphic to $\wh{W}(\bm{\pi})$ with $\bm{\pi} \ngeq \bm{\pi}_{k,j}$,
we have $\Hom{\bE}(\bE_{k,j-1}/\bE_{k,j},\bE^{\gb}/\bE_{k,j-1})=0$ by (\ref{eq:projlim_eigenspace}).
In addition, we have $\mathrm{End}_{\bE}\big(\wh{W}(\bm{\pi}_{k,j})\big) \cong \wh{\bA}_{\bm{\pi}_{k,j}}$ by (\ref{eq:projlim_eigenspace}),
and we easily see from this that the head of $\wh{W}(\bm{\pi}_{k,j})$ is $L(\bm{\pi}_{k,j})$.
Hence, $\wh{W}(\bm{\pi}_{k,j})$ is a projective cover of $L(\bm{\pi}_{k,j})$.
Finally, $\wh{W}(\bm{\pi}_{k,j})$ is free of finite rank over $\wh{\bA}_{\bm{\pi}_{k,j}}$ by Lemma \ref{Lem:properties_of_deformed_local}.
Now we conclude that (\ref{eq:affine_heredity_chain}) is an affine heredity chain, and hence $\bE^\gb$ is an affine quasihereditary algebra.

Then by Theorem \ref{Thm:hereditary_corresp}, $\modfg{\bE^\gb}$ is an affine highest weight category,
and now it follows from the definitions that deformed local Weyl modules $\wh{W}(\bm{\pi})$ are standard and local Weyl modules $W(\bm{\pi})$ are proper standard.
For the proof of the fact that dual local Weyl modules $\wh{W}^\vee(\bm{\pi})$ are proper costandard, see \cite[Proposition 4.18]{fujita2017affine}.

What we have shown is that $\modfg{\bE^{\gb}}$ is an affine highest weight category with respect to a poset whose order is possibly stronger than $\leq$.
However, at this point we can apply Theorem \ref{Thm:Fujita_theorem} to show that $\cF_\gb$ gives an equivalence between $\modfg{\wh{R}(\gb)}$ and $\modfg{\bE^{\gb}}$.
Hence the poset can be replaced by $(\cC_{\bQ,\gb},\leq)$, and the proof is complete.\qed

\def\cprime{$'$} \def\cprime{$'$} \def\cprime{$'$} \def\cprime{$'$}


\end{document}